\documentclass{amsart}

\usepackage{
  amsmath,
  amssymb,
  amsthm,
  hyperref,
  paralist,
  thmtools,
  tikz,
  tikz-cd,
  pgfplots,
  subcaption,
  verbatim,
}
\pgfplotsset{compat=1.18}
\usepackage[normalem]{ulem}
\usepackage[margin=1.5in]{geometry}
\usetikzlibrary{arrows, automata, backgrounds, calc, positioning, shapes, shapes.geometric, decorations.pathreplacing}
\usepackage[nocompress]{cite}
\usepackage[textsize=scriptsize]{todonotes}
\presetkeys{todonotes}{inline}{}

\makeatletter
\providecommand\@dotsep{5}
\makeatother

\renewcommand{\k}{\mathbf{k}}

\newcommand{\bfS}{\mathbf S}

\newcommand{\calC}{\mathcal C}

\newcommand{\aonehat}{\widehat{A}_1}

\newcommand{\from}{\colon}
\newcommand{\heart}{\heartsuit}
\newcommand{\elealpha}{e^{\le}_{\alpha}}
\newcommand{\egalpha}{e^{>}_{\alpha}}

\DeclareMathOperator{\id}{id}
\DeclareMathOperator{\Z}{\mathbf Z}
\DeclareMathOperator{\R}{\mathbf R}
\DeclareMathOperator{\Q}{\mathbf Q}
\DeclareMathOperator{\C}{\mathbf C}
\DeclareMathOperator{\Aut}{Aut}
\DeclareMathOperator{\Cone}{Cone}
\DeclareMathOperator{\dbl}{dbl}
\DeclareMathOperator{\End}{End}

\DeclareMathOperator{\Hom}{Hom}
\DeclareMathOperator{\HN}{HN}
\DeclareMathOperator{\gradedHom}{Hom^{\bullet}}

\DeclareMathOperator{\homBar}{\overline{hom}}

\DeclareMathOperator{\Path}{Path}

\DeclareMathOperator{\std}{std}
\DeclareMathOperator{\Stab}{Stab}

\DeclareMathOperator{\Tr}{Tr}
\DeclareMathOperator{\tr}{\mathsf{tr}}

\DeclareMathOperator{\PSL}{PSL}
\DeclareMathOperator{\loopmap}{loop}

\DeclareMathOperator{\coker}{coker}
\DeclareMathOperator{\ob}{ob}

\DeclareMathOperator{\dgmod}{dgmod-}

\DeclareMathOperator{\sgn}{s}

\declaretheorem[parent=section]{theorem}
\declaretheorem[sibling=theorem]{lemma}
\declaretheorem[sibling=theorem]{corollary}

\declaretheorem[sibling=theorem]{proposition}
\declaretheorem[sibling=theorem, style=remark]{remark}
\declaretheorem[sibling=theorem, style=remark]{warning}
\declaretheorem[sibling=theorem, style=definition]{definition}
\declaretheorem[sibling=theorem, style=definition]{example}

\newcommand{\minus}{\setminus}

\usepackage[capitalize]{cleveref}
 
\title[Thurston compactification of the space of stability conditions]{A Thurston compactification of the space of stability conditions}
\author{Asilata Bapat \and Anand Deopurkar \and Anthony M. Licata}

\begin{document}

\begin{abstract}
  We propose compactifications of the moduli space of Bridgeland stability conditions of a triangulated category.
  Our construction arises from a viewing a stability condition as a metric on the underlying category and is inspired by the Thurston compactification of the Teichm\"uller space of hyperbolic metrics on a surface.
  The key ingredient in the construction are maps from the stability manifold to an infinite projective space.
  We prove that, under suitable hypotheses, these maps are injective and their image has a compact closure.
  We identify a family of points in the boundary that are categorical analogous to the intersection functionals in Teichm\"uller theory.

  We study in detail the geometry of the resulting compactification for the 2-Calabi--Yau categories of quivers, and fully work out the cases of the \(A_2\) and \(\aonehat\) quivers.
  To do so, we carefully examine the dynamics of Harder--Narasimhan multiplicities under auto-equivalences of the category.
  We introduce a finite automaton to study this dynamics and employ it in our analysis of the \(A_{2}\) and \(\aonehat\) categories.
\end{abstract}

\maketitle

\section{Introduction}\label{sec:intro}
The space of Bridgeland stability conditions has emerged as an important invariant of triangulated categories, evidenced by its prominent role in homological algebra, algebraic geometry, mirror symmetry, representation theory, and mathematical physics.
This paper takes the point of view of studying stability conditions through the metrics they define on the underlying triangulated category~\cite{ike:21}.
This perspective infuses the study of stability conditions with ideas from metric geometry, opening up new possibilities inspired by the study of moduli spaces of metrics elsewhere in mathematics.

In recent years, a significant body of work has established connections between metric geometry and homological algebra, particularly by relating Teichm\"uller theory and Bridgeland stability~\cite{hai.kat.kon:17,dim.hai.kat.ea:14,bri.smi:15,bri.qiu.sut:20,tho:06,smi:15, fan.fil.hai.ea:21,ike:21}.
A particularly promising  prospect of this connection is a theory of dynamics in triangulated categories parallel to the theory of dynamics in surface geometry.
To further develop this subject, we must consider not only spaces of Bridgeland stability conditions but also their modular compactifications.

In this paper, we propose a family of such compactifications, inspired by Thurston's compactification of Teichm\"uller space.
The first part of the paper establishes general results about this construction by studying a stability condition via its associated family of metrics.
The highlights of this part of the paper are~\Cref{cor:mass-determines-stables,thm:product-of-masses-injective}.
\Cref{cor:mass-determines-stables} shows that the stable objects of a stability condition (though not necessarily the central charge) are determined by any one of its associated metrics.
\Cref{thm:product-of-masses-injective} shows that in fact the entire stability condition (including the central charge) is determined by any pair of its associated metrics.

The second part of the paper focuses on a family of important categories, namely the 2-Calabi--Yau categories associated to quivers.
In this case, a stability condition is completely determined by any one of its associated metrics.
In~\Cref{sec:a2,sec:aonehat} respectively, we work out our proposed compactification explicitly for two small but highly non-trivial cases, namely the \(A_2\) and \(\widehat A_1\) quivers.
To do so, we introduce a number of general ideas, including the use of automata to describe the evolution of Harder--Narasimhan filtrations (\Cref{sec:intro-automata} and~\Cref{sec:automata}) and the notion of rectifiable filtrations (\Cref{sec:rectifiable-filtrations}).

\subsection{The construction}\label{sec:weight}
Thurston's compactification of Teichm\"uller space proceeds by embedding it in an infinite projective space, and then taking its closure.
To do this, one first fixes a suitable collection of curves $\mathbf{S}$.  The embedding sends a point of Teichm\"uller space, represented by a metric $\mu$, to the real-valued function on $\mathbf{S}$ defined by the length with respect to $\mu$.

In our categorical analogue of Thurston's construction, Bridgeland stability conditions provide the metrics.  The analogue of the set of curves tends to be flexible---the best choice may vary for different categories. 
In the examples studied in the second part of the paper, we take \(\bfS\) to be the set of spherical objects.  Indeed, as described by Khovanov and Seidel~\cite{kho.sei:02}, in a 2-Calabi--Yau category of type A there is a precise correspondence between spherical objects and curves on the punctured disc.
In greater generality, a reasonable choice for ``curves'' may be the objects that are stable for some stability condition.
The categorical analogue of the length of a curve with respect to a metric is the mass of such an object with respect to a stability condition, whose definition we now recall.

Let $\mathcal C$ be a triangulated category.
Given a stability condition $\tau$ on $\calC$ with central charge $Z$, the \emph{mass} of an object $x$ is defined as
\begin{equation}\label{eqn:mass}
  m_{\tau}(x) = \sum_i |Z(a_i)|,
\end{equation}
where the $a_i$ are the semistable Harder--Narasimhan factors of $x$.
The notion of mass admits a $q$-analogue for any positive real number $q > 0$. 
The \emph{$q$-mass} of an object $x$ is defined as
\[ m_{q,\tau}(x) = \sum_i q^{\phi(a_i)}|Z(a_i)|,\]
were $\phi(a_i)$ is the phase of the semistable object $a_i$.

The \(q\)-mass satisfies the triangle inequality~\cite{ike:21}.
That is, in a distinguished triangle
\( x \to y \to z \xrightarrow{+1}\),
we have the inequality
\( m_{q,\tau}(y) \leq m_{q,\tau}(x) + m_{q,\tau}(z)\).
It easily follows that the function that sends a morphism \(f\) to the non-negative real number \(m_{q,\tau}(\operatorname{Cone}(f))\) defines a translation invariant metric on the category in the sense of~\cite{nee:20}.
In our paper, all the metrics on categories are of this kind.
Thus, we sometimes refer to the function \(m_{q,\tau}\) as a metric.

Let $\mathbf{S}$ be a subset of the set of objects of $\mathcal{C}$.
Denote by $\mathbf{R}^{\mathbf{S}}$ the space of functions from $\mathbf{S}$ to $\mathbf{R}$, endowed with the product topology.
Let $\mathbf{P}^{\mathbf{S}}$ denote the quotient $(\mathbf{R}^{\mathbf{S}} \setminus \{0\})/\mathbf{R}^\times$.
Let \(\Stab(\mathcal{C})\) be the space of stability conditions on \(\mathcal{C}\), and recall that it admits a standard action of \(\mathbf{C}\).
For each \(q\), the association $\tau \mapsto m_{q,\tau}$ defines a continuous map
\[ m_q \from \Stab(\mathcal C)/ \mathbf{C} \to \mathbf{P}^{\mathbf{S}}.\]
Our proposed ``Thurston compactification'' of $\Stab(\mathcal C)/\mathbf{C}$ is the closure of the image.

\subsection{The results}
Recall the salient features of Thurston's compactification of Teichm\"uller space.
\begin{enumerate}
\item The length function up to scaling determines the metric up to isotopy.
  That is, the map taking a metric to its length function is injective.
  In fact, it is a homeomorphism onto its image.
\item The closure of the image is a compact real manifold with boundary. 
\item The boundary contains a distinguished collection of functionals, namely the intersection functionals.
  This collection is in fact dense in the boundary.
  Finally, all points of the boundary have a modular interpretation as projective measured foliations. 
\end{enumerate}

We now address each of these points in our categorical context.
In the remarks below, we have chosen \(\bfS\) to be the set of objects \(x\) such that there is a stability condition in which \(x\) is semi-stable.
We remind the reader that the best choice for \(\mathbf{S}\) may vary in applications.

\subsubsection{Injectivity and homeomorphic embedding}\label{sec:homeomorphic-embedding}
The family of functions $m_{q,\tau}$ determines the stability condition $\tau\in \Stab(\mathcal C)/ \mathbf{C}$.
In fact, we prove in~\Cref{thm:product-of-masses-injective} that for any two distinct positive real numbers $q_1,q_2$, the product map $m_{q_1}\times m_{q_2}$ is injective.
We conjecture that a product of mass maps gives a homeomorphic embedding of $\Stab(\mathcal C)/ \mathbf{C}$.

We prove in~\Cref{cor:mass-determines-stables} that for any single \(q\), the function \(m_{q,\tau}\) determines the stable objects of \(\tau\).
On the other hand, we note that a single \(m_{q,\tau}\) may fail to determine the central charge of \(\tau\).
As a result, we cannot expect in general that a single \(m_q\) gives an embedding of the stability manifold (a phenomenon also observed in~\cite{kik.kos.ouc:22}).

Nevertheless, for 2-CY categories of quivers, we prove in~\Cref{prop:injectivity-general} that for any \(q\), the map $m_q$ is injective.
Moreover, in types $A_2$ and $\widehat{A_1}$, we prove in~\Cref{prop:homeo1,prop:homeo2} that \(q = 1\) already gives a homeomorphic embedding.
(The case of \(q \neq 1\) is handled in the follow-up work~\cite{bap.bec.lic:23} using similar techniques).
Our key technical tool is \emph{Harder--Narasimhan automata}, which we briefly discuss in~\Cref{sec:intro-automata} and develop in~\Cref{sec:automata}.

\subsubsection{Properties of the closure}\label{sec:compact-closure}
We prove that the closure of the image of $m_q$ is compact under mild hypotheses.
Precisely, as long as the set $\mathbf S$ contains a classical generator of the category, the closure of the image of $m_q$ is compact (\Cref{prop:pre-compactness}).
Categories that have a classical generator include derived categories of coherent sheaves on a quasi-projective variety, and the category of perfect complexes over a dg algebra.

For applications to dynamics, it is particularly important to understand when the compactification is homeomorphic to a closed Euclidean ball, since in that case every autoequivalence will act with a fixed point on the compactification.
For the $A_2$ 2-Calabi--Yau category, we prove that the closure of $m_1$ is homeomorphic to a closed Euclidean ball (\Cref{prop:r3-triangle}).

We expect our statement about $A_2$ to generalise to other important cases, such as 2-CY quiver categories and derived categories of K3 surfaces.

\subsubsection{Functionals on the boundary}\label{sec:modular-intepretation}

Assume that \(\mathcal{C}\) is \(\k\)-linear and of finite type.
Let \(a \in \mathcal{C}\) be a spherical object.
Associated to \(a\), we have a function on \(\bfS\) defined by
\[
  x \mapsto \dim \gradedHom(a,x) = \sum_{i} \dim \Hom^{i}(a,x).
\]
In~\Cref{def:hombar}, we define a close analogue of this function which also incorporates the \(q\), denoted by \(\homBar_{q}(a)\).
This is a categorical analogue of the intersection functionals of curves.
We prove in~\Cref{cor:limit-mass} that, in direct analogy with geometry, the function \(\homBar_q(a)\) represents a boundary point in our compactification.

Whether or not the \(\homBar\) functionals are dense in the boundary is a subtle question, whose answer depends on \(q\).
In the case of the 2-CY category of type $A_2$ and $q=1$, the answer is ``yes''---they are essentially the rational numbers $\mathbf{Q} \cup \{\infty\}$ on the boundary circle $\mathbf{R} \cup \{\infty\}$.
For the same category and \(q \neq 1\), however, they are not dense in the boundary, but a fractal subset, which may be identified with the left $q$-rational numbers of~\cite{bap.bec.lic:23}.

Giving a moduli interpretation to the boundary points, as a kind of categorical analogue of measured foliations, remains an excellent outstanding problem for future work.

\subsection{Departures from Teichm\"uller theory}\label{sec:comparison}
The categorical picture, although motivated by Thurston's geometric picture, also differs from it interesting ways.

While a stability condition defines a family of metrics on the underlying triangulated category, those metrics are not ``hyperbolic'' in a meaningful sense.
They behave more like flat metrics, which is consonant with~\cite{bri.smi:15}, where the authors establish a precise relationship between quadratic differentials and stability condition in certain 3-CY settings.

In the 2-CY category of type \(A_{n}\), the stability manifold and the Teichm\"uller space of a disk with \((n+1)\) punctures are both universal covers of the complement of the hyperplane arrangement of type $A_{n}$.
For $n > 2$, however, we believe that no homeomorphism between them extends to a homeomorphism between the Thurston compactification and any of our compactifications.
In Thurston's construction, the Teichm\"uller space maps homeomorphically onto the interior of the Thurston compactification.
In our case, on the other hand, we see that for $n > 2$, the interior of our compactification contains additional points which are not themselves stability conditions.
(Some of them can be identified with degenerate stability conditions constructed by Bolognese~\cite{bol:20}.)
It seems likely that a closer Teichm\"uller-theoretic analogue of our construction is a compactification of the moduli space of flat metrics on a surface, constructed, for example, by Duchin, Leininger, and Rafi~\cite{duc.lei.raf:10}.
It would be interesting to relate our construction and theirs.

Our theory has a natural $q$-analogue, so we obtain a family of compactifications of the space of stability conditions, depending on a positive real parameter $q$.
This should not be confused with existing literature on quantizations of Teichm\"uller space (e.g.~\cite{foc.gon:08}), which pertain to quantizing algebras of functions.

\begin{figure}
  \centering
  \begin{tikzpicture}
    \draw[thick, dotted, fill=black!10!white] (0,0) circle (2cm);
    \draw (0,0) node {$\Stab(\mathcal C)/\mathbf{C}$};
    \draw[fill=black] (-30:2cm) circle (0.05cm);
    \draw (-30:2.3cm) node {\tiny $P_1$};
    \draw[fill=black] (90:2cm) circle (0.05cm);
    \draw (90:2.3cm) node {\tiny $P_2 \to P_1$};
    \draw[fill=black] (210:2cm) circle (0.05cm);
    \draw (210:2.3cm) node {\tiny $P_2$};
    \draw[fill=black] (270:2cm) circle (0.05cm);
    \draw (270:2.3cm) node {\tiny $P_1 \to P_2$};
    \draw[fill=black] (30:2cm) circle (0.05cm);
    \draw (20:3.3cm) node {\tiny $P_2 \to P_1 \to P_1[1]$};
    \draw[fill=black] (150:2cm) circle (0.05cm);
    \draw (160:3.3cm) node {\tiny $P_2 \to P_2 \to P_1[1]$};
  \end{tikzpicture}
  \caption{The compactification of $\Stab(\mathcal C)/\mathbf{C}$. The spherical objects appear as a dense subset of the boundary.}\label{fig:main}
\end{figure}

\subsection{Harder--Narasimhan automata}\label{sec:intro-automata}
After establishing the general results, we focus on the 2-CY categories \(\mathcal{C}_{\Gamma}\) associated to finite connected quivers \(\Gamma\).
The Artin--Tits braid group \(B_{\Gamma}\) acts (conjecturally faithfully) on the 2-CY category \(\mathcal{C}_{\Gamma}\).
To understand the compactification of the stability manifold, we heavily use this action.
Along the way, we develop techniques that may be broadly applicable to the study of stability conditions in the presence of a group action.

The key question we address is how the HN filtration of an object transforms under an auto-equivalence.
To make this precise, fix a stability condition and let $\Sigma$ be the set of indecomposable semi-stable objects up to shift.
For an object $x$, let $HN(x) \in \mathbf{Z}^\Sigma$ be its HN multiplicity vector---the function that assigns to \(s \in \Sigma\) the number of times it appears (up to shift) in the HN filtration of \(x\).
Suppose a group $G$ acts on our category by auto-equivalences.
This action may not induce a compatible linear action on $\mathbf{Z}^\Sigma$.
Nevertheless, in many situations---including the $A_2$ and $\aonehat$ cases studied in the latter part of the paper---it induces a ``piecewise linear'' action.
Roughly speaking, this means the following.
We can partition the objects of the category into subsets, called \emph{states}.
For each ordered pair of states, there are certain allowable elements of $G$ that send all objects of the first state into objects of the second state and induce linear transformations on their HN multiplicities.
We formalise this structure as a decorated finite automaton called an \emph{HN automaton} (see~\Cref{sec:automata}).

An HN automaton allows us to express the evolution of HN multiplicities as a linear map.
This can lead to a direct method to compute the categorical entropy of an autoequivalence (see~\cite{hen:22}).

We construct HN automata in type $A_2$ and $\aonehat$.  These automata recognize all braids, meaning that every braid can be represented by a sequence of allowable arrows in the decorated graph.  We use these automata as a tool to prove all of the basic properties of our compactifications.
As an additional application, we recover a theorem of Rouquier--Zimmermann~\cite[Proposition 4.8]{rou.zim:03} in the type $A_2$ case and prove a new analogue of their theorem in the type $\aonehat$ case.
We anticipate automata to play an important role in the study of group actions on triangulated categories.

\subsection{Organisation}\label{sec:org}\
In~\Cref{sec:background}, we recall some background on Bridgeland stability conditions, spherical objects and spherical twists.
In~\Cref{sec:mass}, we prove foundational results about the mass functions associated to stability conditions, including the fact that the mass function determines the stable objects (\Cref{cor:mass-determines-stables}).
In~\Cref{sec:proj-embedding}, we use the mass to define a map from the stability manifold to a projective space, and establish its basic properties.
These include the compactness of the closure of the image (\Cref{prop:pre-compactness}), injectivity for a pair of masses (\Cref{thm:product-of-masses-injective}), and the presence of \(\homBar\) functionals in the closure (\Cref{cor:limit-mass}).
In~\Cref{sec:automata}, we introduce the notion of Harder--Narasimhan automata.

So far, all the discussion applies generally.
From~\Cref{sec:cgamma} onwards, we focus on specific categories, namely the 2-Calabi--Yau categories of quivers.
\Cref{sec:cgamma} begins with the definition and basic properties of the categories, and proceeds to prove the injectivity of a single mass map (\Cref{prop:injectivity-general}).
\Cref{sec:a2} contains the \(A_2\) case in detail.
Using a suitable HN automaton, we prove that the mass map is a homeomorphic embedding (\Cref{prop:homeo1}; the boundary is a circle (\Cref{prop:pmf} and~\Cref{prop:b-closure}); and construct an explicit homeomorphism from the closure to a 2-disk (\Cref{prop:r3-triangle}).
Along the way, we reprove a result of Rouquier--Zimmermann~\cite[Proposition 4.8]{rou.zim:03}.
\Cref{sec:aonehat} contains analogous results for the \(\widehat{A}_1\) case, including a new analogue of Rouquier--Zimmermann's result.

The paper has two appendices.
They contain technical results, which may have broader use.
\Cref{sec:rectifiable-filtrations} addresses the question of when a filtration in a triangulated category can be re-arranged to obtain the Harder--Narasimhan filtration.
\Cref{sec:an} studies homological properties of self-extensions of a spherical object.

\subsection*{Acknowledgements}
We thank Louis Becker and Edmund Heng for carefully reading multiple drafts of this paper, finding errors, and suggesting improvements.
We thank Tom Bridgeland for enlightening conversations and pointing out connections to other work.
We benefited from interactions with Fabian Haiden, Curtis McMullen, Amnon Neeman, Alexander Polishchuk, and Catharina Stroppel.
We are extremely grateful to the anonymous referee for suggesting a number of detailed improvements to the exposition; the paper has been significantly updated as a result.
A.M.L. is supported by the Australian Research Council Award FT180100069 and DP180103150.
A.D. is supported by the Australian Research Council Award DE180101360.

\section{Background}\label{sec:background}
\subsection{Stability conditions}\label{subsec:stability-conditions}
We assume familiarity with the theory of Bridgeland stability conditions (see~\cite{bri:07}), but we recall the basic definitions.
Let $\mathcal{C}$ be a triangulated category.
A stability condition $\tau$ on $\mathcal{C}$ is specified by two compatible structures $(\mathcal P,Z)$:
\begin{enumerate}
\item the \emph{slicing} $\mathcal{P}$ consists of additive subcategories $\mathcal{P}(\phi)$ for each real number $\phi$;
\item the \emph{central charge} $Z$ is a homomorphism of additive groups from the Grothendieck group $K(\mathcal{C})$ to the complex numbers $\mathbf{C}$.
\end{enumerate}
The slicing and the central charge satisfy the following conditions:
\begin{enumerate}
\item $\mathcal{P}(\phi + 1) = \mathcal{P}(\phi)[1]$;
\item if $a \in \mathcal{P}(\phi)$ and $b \in \mathcal{P}(\psi)$ with $\phi > \psi$, then $\Hom(a,b) = 0$;
\item\label{HNfiltration} for every nonzero object $x \in \mathcal{C}$, there exists a sequence of morphisms
  \[0 = x_0 \to x_1 \to \dots \to x_n = x\]
  such that if we define \(a_i\) so that each triangle in 
  \[
    \begin{tikzcd}[column sep=tiny]
      0 = x_0 \ar{rr}  & & x_1\ar{rr} & & x_2 &  \cdots  & x_{n-1}\ar{rr} & & x_n = x,  \\
      & a_1\ar[dashed]{ul}{+1}\ar[<-]{ur} & & a_2\ar[dashed]{ul}{+1}\ar[<-]{ur}& & \cdots & &  a_n\ar[dashed]{ul}{+1}\ar[<-]{ur}
    \end{tikzcd}
  \]
  is exact, then $a_i \in \mathcal{P}(\phi_i)$ and $\phi_1 > \cdots > \phi_n$;
\item\label{mass} for every non-zero $a \in \mathcal{P}(\phi)$, there is some positive real number $r$ such that
  \[Z([a]) = r \cdot e^{i \pi \phi}.\]
\end{enumerate}

The additive category $\mathcal{P}(\phi)$ turns out to be abelian.
Its objects are called \emph{semistable of phase $\phi$}.

  We call a sequence of maps \(x_0 \to \cdots \to x_n = x\) a \emph{filtration} of \(x\).
  For each \(i\), let \(a_i\) be such that \(x_{i-1} \to x_{i} \to a_{i} \xrightarrow{+1}\) is exact.
  We then say that the objects \(a_i\) are \emph{factors} of the filtration.

  For a morphism \(f \colon x_{i-1} \to x_i\), we use \(\Cone(f)\) to denote any object that fits into a distinguished triangle \(x_{i-1} \xrightarrow{f} x_{i} \to \Cone(f) \xrightarrow{+1}\).
  Such an object is unique up to a possibly non-unique isomorphism.
  We only use this notation in contexts where the ambiguity of the isomorphism is irrelevant.

The specific filtration described in~\eqref{HNfiltration} is called the \emph{Harder--Narasimhan (HN) filtration} of $x$.
Its factors are called \emph{HN factors} and their phases are called \emph{HN phases}.
The HN filtration is unique up to isomorphisms (see, e.g.~\cite[\S 3]{bri:07}).

Fix a stability condition \(\tau = (\mathcal{P}, Z)\) on \(\mathcal{C}\).
We recall  from~\cite[\S 3]{bri:07} the induced truncation structure on \(\mathcal{C}\).
This structure depends on the slicing, but brevity, the notation hides this dependence.
Let \(I \subset \mathbf{R}\) be an interval.
Denote by \(\mathcal{C}_I \subset \mathcal{C}\) the full subcategory consisting of objects whose HN phases lie in \(I\).
The notation \(\mathcal C_{\geq \alpha}\) means \(\mathcal C_{[\alpha, \infty)}\), and likewise for \(\mathcal C_{< \alpha}\) and \(\mathcal C_{\geq \alpha}\) and \(\mathcal C_{> \alpha}\).
The pair of categories \(\mathcal C_{> \alpha}\) and \(\mathcal C_{\leq \alpha+1}\) define a \(t\)-structure on \(\mathcal C\) with the heart \(\mathcal C_{(\alpha,\alpha+1]}\).
As a result, we have a truncation functor \[\tr_{> \alpha} \colon \mathcal C \to \mathcal C_{> \alpha},\] right adjoint to the inclusion \(\mathcal C_{\leq \alpha} \to \mathcal C\)~\cite[1.3]{bel.ber.del:82}.
We abbreviate \(\tr_{> \alpha}x\) to \(x_{> \alpha}\).
We have similar functors \(\tr_{< \alpha}\) and \(\tr_{\geq \alpha}\), and \(\tr_{\leq \alpha}\), and the corresponding notation \(x_{< \alpha}\) and \(x_{\geq \beta}\) and \(x_{\leq \alpha}\).
For an interval \(I = [\alpha,\beta]\), we define \(\tr_I\) as
\[ \tr_I = \tr_{\geq \alpha} \circ \tr_{\leq \beta} = \tr_{\leq \beta} \circ \tr_{\geq \alpha},\]
and set \(x_I = \tr_Ix\).
We have analogous functors and notation for other kinds of intervals, namely \((\alpha,\beta)\), \((\alpha,\beta]\), and \([\alpha,\beta)\).

Let $\Stab(\mathcal{C})$ be the set of all stability conditions satisfying an additional condition called local finiteness~\cite[Definition 5.7]{bri:07}.
The main result of~\cite{bri:07} states that each connected component of \(\Sigma \subset \Stab(\mathcal{C})\) is a complex manifold, and the forgetful map
\[ \Stab(\mathcal{C}) \ni (\mathcal{P}, Z) \mapsto Z \in \Hom(K(\mathcal{C}), \mathbf{C})\]
is a local homeomorphism to a linear subspace \(V(\Sigma)\).

The complex numbers \(\mathbf{C}\) act on $\Stab(\mathcal C)$ as follows.
Say $\tau = (\mathcal P,Z)$ and $\omega = a + i\pi b$, then
\[(\omega \cdot \mathcal P)(\phi) = \mathcal P(\phi - b) \text{ and } \omega \cdot Z = e^\omega Z.\]
The action of \(\mathbf{C}\) on \(\Stab(\mathcal{C})\) lifts the action of \(\mathbf{C}^{*}\) on \(\Hom(K(\mathcal{C}), \mathbf{C})\) by scaling.
It follows that the induced forgetful map
\[\Stab(\mathcal{C})/\mathbf{C} \to \Hom(K(\mathcal{C}), \mathbf{C})/\mathbf{C}^{*}\]
maps each connected component \(\Sigma/\mathbf{C}\) locally homeomorphically to \(V(\Sigma)/ \mathbf{C}^{*}\).

Observe that the purely imaginary numbers \(i\mathbf{R} \subset \mathbf{C}\) simply shift the slicing and rotate the central charge.
We call such transformations \emph{translations}.

\subsection{Spherical objects and twists}\label{subsec:sphericals}
Let \(\k\) be a field and \(\mathcal{C}\) a \(\k\)-linear triangulated category.
Assume that \(\mathcal{C}\) is of finite type, that is, for all objects \(x,y\) of \(\mathcal{C}\), the vector space
\[\gradedHom(x,y) = \bigoplus_{n} \Hom(x,y[n])\]
is a finite-dimensional \(\k\)-vector space.

We recall the notion of a spherical object introduced in~\cite[Definition 2.9]{sei.tho:01} (see also~\cite[Chapter 8]{huy:06}).
An object \(x\) of \(\mathcal{C}\) is said to have a Serre dual if the functor \(\gradedHom(x,-)\) is representable.  If \(x\) has a Serre dual, the representing object \(Sx\) is unique up to a unique isomorphism, and there are functorial isomorphisms
\[\gradedHom(x,y) \cong \gradedHom(y,Sx)^{\vee}.\]
A \emph{\(d\)-Calabi--Yau} object is an object \(x\) such that its shift \(x[d]\) is its Serre dual.
A \emph{\(d\)-spherical object} is a \(d\)-Calabi--Yau object \(x\) whose endomorphism algebra is isomorphic to the cohomology algebra of the \(d\)-sphere:
\[\gradedHom(x,x) \cong H^{\bullet}(S^d, \k).\]
Note that this algebra is simply \(k[t]/t^{2}\), where  \(t\) has degree \(d\).
We denote by \(\loopmap_x\) an unspecified homogeneous generator of degree \(d\) of \(\gradedHom(x,x)\).

If \(x\) is a spherical object and \(y\) is any object, then there is a perfect pairing
\begin{equation}\label{eqn:dcyiso}
  \Hom^i(x,y) \times \Hom^{d-i}(y,x) \to \Hom^d(x,x) \cong \k,
\end{equation}
induced by composition (\cite[Proposition 2.2]{kel:08}).

If every object in the category \(\mathcal{C}\) is \(d\)-Calabi--Yau (for the same \(d\)), then the category \(\mathcal{C}\) is said to be 
a \(d\)-Calabi--Yau category.  Thus, if \(\mathcal{C}\) is a \(d\)-Calabi--Yau category, then for any \(x, y \in \mathcal{C}\), there exist functorial isomorphisms
\[\Hom^i(x,y) \cong \Hom^{d-i}(y,x)^{\vee}.\]
Moreover, it follows that if \(x\) is an indecomposable object in a \(d\)-Calabi--Yau category such that the \(\k\)-vector space \(\gradedHom(x,x)\) is two-dimensional, then \(x\) is automatically a \(d\)-spherical object of \(\mathcal{C}\).

There is also a notion of a \emph{strongly} \(d\)-Calabi--Yau category, in which the pairing corresponding to the functorial isomorphism~\eqref{eqn:dcyiso} is required to be anti-symmetric (see~\cite{kel:08}).
The categories we study in the later part of the paper are strongly \(d\)-Calabi--Yau.

Assume that \(\mathcal{C}\) admits a dg-enhancement.
This is true, for instance, if \(\mathcal{C}\) is algebraic in the sense of~\cite{kel:06} or enhanced in the sense of~\cite{bon.kap:91}, and will hold in all the examples we consider.
Fix a dg-enhancement on \(\mathcal{C}\), which guarantees that \(\mathcal{C}\) has functorial cones.
Then any spherical object $x$ gives rise to an autoequivalence $\sigma_x \from \mathcal C \to \mathcal C$ called the \emph{spherical twist} in $x$ (see~\cite[\S 2.2]{sei.tho:01}).
For \(y \in \mathcal{C}\), the object \(\sigma_x(y)\) is defined to be the cone of the evaluation map
\[ x \otimes \gradedHom(x,y) \xrightarrow{\operatorname{ev}} y.\]
The twist \(\sigma_x\) and its inverse \(\sigma_x^{-1}\) give rise to distinguished triangles
\[
\gradedHom(x,y) \otimes x \xrightarrow{\operatorname{ev}} y \to \sigma_x(y) \xrightarrow{+1} \text{ and }
\sigma_{x}^{-1}y \to y \xrightarrow{\operatorname{coev}} \gradedHom(y,x)^{*} \otimes x \xrightarrow{+1}.
\]
The evaluation map \(\operatorname{ev}\) is self-explanatory.
The co-evaluation map \(\operatorname{coev}\) is the adjoint to the evaluation map
\[ \gradedHom(y,x) \otimes y \to x.\]

\section{The mass associated to a stability condition}\label{sec:mass}
Let $\mathcal{C}$ be a triangulated category and let $\tau$ be a stability condition on $\mathcal{C}$.
Fix a positive real number $q$; for us the most important case is $q=1$.

Let \(x\) be an object of $\mathcal{C}$.
Let
  \[
    \begin{tikzcd}[column sep=tiny]
      0 = x_0 \ar{rr}  & & x_1\ar{rr} & & x_2 &  \cdots  & x_{n-1}\ar{rr} & & x_n = x,  \\
      & a_1\ar[dashed]{ul}{+1}\ar[<-]{ur} & & a_2\ar[dashed]{ul}{+1}\ar[<-]{ur}& & \cdots & &  a_n\ar[dashed]{ul}{+1}\ar[<-]{ur}
    \end{tikzcd}
  \]
  be the HN filtration of $x$, where $a_i$ is semi-stable of phase $\phi_{i}$.
  We define the \emph{\((q,\tau)\)-mass} of \(x\) by the formula
  \[ m_{q,\tau}(x) = \sum_i q^{\phi_i}|Z(a_i)|.\] 
  We take the mass of the zero object to be zero.

  A key property of the mass is that it satisfies a triangle inequality.
      \begin{proposition}[Triangle inequality]\label{prop:triangle}
     Let $x \to y \to z \xrightarrow{+1}$ be an exact triangle.
     Then
     \[ m_{q,\tau}(y) \leq m_{q,\tau}(x) + m_{q,\tau}(z).\]
   \end{proposition}
   \begin{proof}
 See~\cite[Theorem 3.8]{ike:21}.
   \end{proof}
   
   We devote~\Cref{sec:rectifiable-filtrations} to understanding more about when equality holds in~\Cref{prop:triangle}.
   The following is a particularly simple instance.
   \begin{proposition}\label{prop:highlow}
     Let $x \to y \to z \xrightarrow{+1}$ be an exact triangle.
     Assume that the lowest HN phase of \(x\) is greater than or equal to the highest HN phase of \(z\).
     Then
     \[ m_{q,\tau}(y) = m_{q,\tau}(x) + m_{q,\tau}(z).\]
   \end{proposition}
   \begin{proof}
     Using the exact triangle, we can splice together filtrations of \(x\) and \(z\) to get a filtration of \(y\), as follows.
     Suppose we are given \(0 = x_0 \to \cdots \to x_n = x\) with factors \(a_1, \dots, a_n\) and \(0 = z_0 \cdots \to z_m = z\) with factors \(b_1, \dots, b_m\).
     Set \(y_{m+n} = y\), and recall that \(z_m = z\).
     For each \(i = m, \ldots, 1\), define \(y_{i+n-1}\) by descending induction so that it fits into a map of distinguished triangles as follows:
     \[
       \begin{tikzcd}
         y_{n+i-1}\ar{d} \ar{r}& z_{i-1}\ar{d} \ar{r}& x[1] \ar{r}{+1}\ar{d}{=} & {}\\
         y_{n+i} \ar{r}& z_i \ar{r}& x[1]\ar{r}{+1} & {}.
       \end{tikzcd}
     \]
     By the octahedral axiom, we see that for each \(i = m, \ldots, 1\), there is a distinguished triangle
     \[y_{n+i-1} \to y_{n+i} \to b_i\xrightarrow{+1}.\]
     Note that \(y_n = x\).
     For each \(i = n, \ldots, 1\), set \(y_{i-1} = x_{i-1}\).
     Then for \(i = n, \ldots, 1\), there is a distinguished triangle
     \[y_{i-1} \to y_i \to a_i\xrightarrow{+1}.\]
     Therefore we have constructed a filtration
     \begin{equation}\label{eq:yfilt}
       0 = y_0 \to y_1 \to \cdots \to y_{m+n} = y
     \end{equation}
     with factors \(a_1,\ldots, a_n,b_1, \ldots, b_m\).
     
     Suppose the filtrations \((x_{i})\) and \((y_{j})\) are the HN filtrations of \(x\) and \(y\).
     By hypothesis, we have \(\phi(a_{n}) \geq \phi(b_{1})\).
     If \(\phi(a_{n}) > \phi(b_{1})\), then~\eqref{eq:yfilt}  is the HN filtration of \(y\), and hence
     \begin{equation}\label{eq:massadd}
       m_{q,\tau}(y) = \sum m_{q,\tau}(a_i) + \sum m_{q,\tau}(b_j) = m_{q,\tau}(x) + m_{q,\tau}(z).
     \end{equation}
     If \(\phi(a_{n}) = \phi(b_{1})\), then shorten the filtration~\eqref{eq:yfilt} by dropping the \(y_{n}\) term so that in the middle it looks like
     \[ \cdots \to y_{n-1} \to y_{n+1} \to \cdots.\]
     Let \(c\) complete the triangle \(y_{n-1} \to y_{n+1} \to c \xrightarrow{+1}\).
     By the octahedral axiom, we have an exact triangle \(a_{n} \to c \to b_{1} \xrightarrow{+1}\).
     Since \(a_{n}\) and \(b_{1}\) are semistable of the same phase, so is \(c\).
     As a result, the shortened~\eqref{eq:yfilt} is the HN filtration of \(y\).
     Furthermore, we have
     \[ m_{q,\tau}(c) = m_{q,\tau}(a_{n}) + m_{q,\tau}(b_{1}),\]
     and hence~\eqref{eq:massadd} holds.
   \end{proof}

   The following key proposition shows that there are no unexpected mass-preserving decompositions of semi-stable objects.
   \begin{proposition}\label{prop:ss-geodesic}
     Let $y$ be a $\tau$-semi-stable object of phase $\phi$.
     Let
     \[ x \to y \to z \xrightarrow{+1}\]
     be an exact triangle such that 
     \[ m_{q,\tau}(y) = m_{q,\tau}(x) + m_{q,\tau}(z).\]
     Then \(x\) and \(z\) are also semi-stable of the same phase $\phi$.
   \end{proposition}
   The proof follows the same ideas as the proof of the triangle inequality in~\cite{ike:21}.
   It uses the following \(q\)-analogue of the triangle inequality for vectors, which is~\cite[Lemma 3.6]{ike:21}.
   \begin{lemma}\label{lem:q-triangle}
     Let \(a, b, c\) be complex numbers of the form
     \[a = r_{a}e^{i\pi\phi_{a}},\quad b = r_{b} e^{i\pi\phi_{b}},\quad c = r_{c}e^{i\pi\phi_{c}},\]
     where \(r_{a},r_{b},r_{c}\) are positive real numbers and \(\phi_{a}, \phi_{b}, \phi_{c}\) are real numbers lying in an interval of length 1.
     Suppose \(b = a+c\).
     Then
     \[ r_{a}q^{\phi_{a}} + r_{c}q^{\phi_{c}} \geq r_{b}q^{\phi_{b}}\]
     with equality if and only if \(\phi_{a} = \phi_{b} = \phi_{c}\).
   \end{lemma}
   \begin{lemma}\label{cor:q-heart}
     Let \(I\) be a half-open interval of length \(1\), and let \(x\) be an object in the \(I\)-heart associated to \(\tau\).
     Write \(Z(x) = r e^{i\pi\phi}\) for \(\phi \in I\).
     Then
     \[m_{q,\tau}(x) \geq r q^{\phi},\]
     with equality if and only if \(x\) is semi-stable of phase \(\phi\).
   \end{lemma}
   \begin{proof}
     Apply~\Cref{lem:q-triangle} inductively to the central charges of the HN factors of \(x\).
   \end{proof}
   
   \begin{proof}[Proof of~\Cref{prop:ss-geodesic}]
     Abbreviate \(m_{q,\tau}\) by \(m\) and assume that both \(x\) and \(z\) are non-zero.
     The statement is unaffected if we translate the slicing of \(\tau\).
     Translate so that \(y\) has phase \(0\).

     We first show that \(z\) lies in the \([0,1)\)-heart associated to \(\tau\).
     Taking cohomology with respect to the corresponding \(t\)-structure gives
     \[ 0 \to H^{-1}(z) \to H^{0}(x) \to H^{0}(y) = y \to H^{0}(z) \to H^{1}(x) \to 0.\]
     Let \(s = \coker (H^{-1}(z) \to H^{0}(x))\) and \(t = \ker(H^{0}(z) \to H^{1}(x))\).
     Then we have the short exact sequence
     \[ 0 \to s \to y \to t \to 0.\]
     By the triangle inequality (\Cref{prop:triangle}), we have
     \begin{equation}\label{eq:iy}
       m(y) \leq m(s) + m(t).
     \end{equation}
     The triangle inequality applied to the defining triangles of \(s\) and \(t\) gives
     \begin{equation}\label{eq:is}
       m(s) \leq m(H^{0}(x)) + q m(H^{-1}(z))
     \end{equation}
     and
     \begin{equation}\label{eq:it}
       m(t) \leq m(H^{0}(z)) + q^{-1}m(H^{1}(x)).
     \end{equation}
     By adding~\eqref{eq:is} and~\eqref{eq:it}, we get
     \[ m(s) + m(t) \leq  m(H^{0}(x)) +q^{-1}m(H^{1}(x)) +  qm(H^{-1}(z))+m(H^{0}(z)). \]
     By~\eqref{eq:iy}, the left hand side dominates \(m(y)\).
     For the right hand side, we use
     \begin{equation}\label{eq:iz}
       qm(H^{-1}(z))+m(H^{0}(z)) \leq \sum_{i} q^{-i}m(H^{i}(z)) = m(z),
     \end{equation}
     and similarly for \(x\).
     Thus, the right hand side is dominated by \(m(x) + m(z)\), which by the hypothesis is also \(m(y)\).
     Hence, equalities must hold in~\eqref{eq:iy},~\eqref{eq:is},~\eqref{eq:it}, and~\eqref{eq:iz}.
     From~\eqref{eq:iz}, we conclude that \(H^{i}(z) = 0\) for all \(i \not \in \{-1,0\}\).
     In the exact sequence
     \[ 0 \to s \to y \to t \to 0,\]
     the object \(y\) is semi-stable of the lowest possible phase in the heart, namely 0.
     It follows that \(s\) is also semi-stable of phase 0.

     Now consider the exact sequence
     \[ 0 \to H^{-1}(z) \to H^{0}(x) \to s \to 0.\]
     Since \(s\) is semi-stable of the lowest phase in the heart,~\Cref{prop:highlow} implies that
     \[m(H^{0}(x)) = m(s) + m(H^{-1}(z)). \]
     On the other hand, we know that~\eqref{eq:is} is an equality, so
     \[ m(s) = m(H^{0}(x)) + q m(H^{-1}(z)).\]
     The last two equations imply that \(H^{-1}(z) = 0\).
     That is, \(z\) lies in the \([0,1)\)-heart.

     A parallel argument shows that \(x\) lies in the \((-1,0]\)-heart.
     
     Write
     \[ Z(x) = r_{x}e^{i\pi\phi_{x}}, \quad Z(y) = r_{y}, \quad Z(z) = r_{z}e^{i\pi\phi_{z}},\]
   where \(r_{x},r_{y}, r_{z}\) are positive real numbers, \(\phi_{x} \in (-1,0]\) and \(\phi_{z} \in [0,1)\).
   \Cref{cor:q-heart} says
   \begin{equation}\label{eq:i1}
     m(x) \geq r_{x}q^{\phi_{x}} \text{ and } m(z) \geq r_{z}q^{\phi_{z}}
   \end{equation}
   with equalities if and only if \(x\) and \(z\) are semi-stable.
   Since \(Z(y) = Z(x) + Z(z)\) lies on the non-negative real axis, we must have
   \[ \phi_{z} \leq \phi_{x} + 1.\]
   Therefore,~\Cref{lem:q-triangle} gives
   \begin{equation}\label{eq:i2}
     r_{x}q^{\phi_{x}} + r_{z}q^{\phi_{z}} \geq r_{y}
   \end{equation}
   with equality if and only if \(\phi_{x} = \phi_{z} = 0\).
   Inequalities~\eqref{eq:i1} and~\eqref{eq:i2}, together with the hypothesis
   \(r_{y} = m(y) = m(x)+m(z)\)
   imply that equalities must hold in~\eqref{eq:i1} and~\eqref{eq:i2}
   Thus, \(x\) and \(z\) are semi-stable of phase 0.     
 \end{proof}

 \begin{corollary}\label{cor:ss-geodesic}
   Let \(y\) be a semi-stable object of phase \(\phi\).
   Let
   \[ 0= y_0 \to y_{1} \to \cdots \to y_{n} = y\]
   be a filtration of \(y\) with factors \(a_{i} = \Cone(y_{i-1} \to y_{i})\).
   Suppose that
   \[ m_{q,\tau}(y) = \sum_{i} m_{q,\tau}(a_{i}).\]
   Then for all \(i\), the object \(a_{i}\) is semi-stable of phase \(\phi\).
 \end{corollary}
 \begin{proof}
   Follows from~\Cref{prop:ss-geodesic} and induction on \(n\).
 \end{proof}
 
 \begin{corollary}\label{cor:stable-object-indivisible}
   Let $y$ be a $\tau$-stable object of phase $\phi$.
   Let
   \[ 0 = y_0 \to y_{1} \to \cdots \to y_{n} = y\]
   be a filtration of \(y\) with factors \(a_{i} = \Cone(y_{i-1} \to y_{i})\).
   Suppose that
   \[ m_{q,\tau}(y) = \sum_{i} m_{q,\tau}(a_{i}).\]
   Then for all \(i\) except possibly one, we have \(a_i = 0\).
 \end{corollary}
 \begin{proof}
   By~\Cref{cor:ss-geodesic}, we conclude that all \(a_{i}\) are $\tau$ semi-stable of the same phase $\phi$.
   Since \(y\) is stable, it is a simple object of abelian category of semi-stable objects of phase \(\phi\).
   Thus, its Jordan--Holder filtration is of length 1.
 \end{proof}
 \begin{remark}\label{rem:stable-HN-factors}
   The Harder--Narasimhan factors of an object \(x\) are semi-stable, but not necessarily stable.
   But recall that the abelian category of semi-stable objects of a given phase is of finite length~\cite[Definition 5.7]{bri:07}.
   As a result every semi-stable object has a finite Jorden--H\"older filtration whose factors are stable of the same phase.
   By splicining together these Jorden--H\"older filtrations, we obtain a filtration of \(x\) with stable factors.
   Note that the mass of \(x\) is the sum of the masses of the (stable) factors in this filtration.
 \end{remark}
 \begin{theorem}\label{cor:mass-determines-stables}
   Let $\tau$ and $\tau'$ be two stability conditions such that for each object \(x\) of \(\mathcal{C}\), we have $m_{q,\tau}(x) = m_{q,\tau'}(x)$.
   Then an object of \(\mathcal{C}\) is $\tau$-stable if and only if it is $\tau'$-stable.
 \end{theorem}
 \begin{proof}
   Let $x$ be a $\tau$-stable object.
   As described in~\Cref{rem:stable-HN-factors}, consider a filtration of \(x\) with $\tau'$-stable factors $y_1,\ldots,y_k$ such that
  \[{m}_{q,\tau'}(x) = \sum_{i=1}^k{m}_{q,\tau'}(y_i).\]
  Since ${m}_{q,\tau}(x) = {m}_{q,\tau'}(x)$ and \({m}_{q,\tau}(y_i) = {m}_{q,\tau'}(y_i)\) for every \(i\), we have:
  \[{m}_{q,\tau}(x) = \sum_{i=1}^k{m}_{q,\tau}(y_i).\]
  Since \(x\) is \(\tau\)-stable, we conclude by~\Cref{cor:stable-object-indivisible} that all but one of the objects \(y_i\) are zero.
  It follows that \(x\) is \(\tau'\)-stable.
\end{proof}

It is useful to record a generalisation~\Cref{cor:mass-determines-stables} that allows us to restrict the objects \(x\).
Let \(\bfS\) be a set of objects of \(\mathcal{C}\) satisfying the following conditions.
\begin{enumerate}
\item\label{iso} If \(x \in \bfS\) and \(x \cong y\), then \(y \in \bfS\).
\item\label{hn} If \(x \in \bfS\) and \(\tau\) is a stability condition, then the \(\tau\)-stable HN factors of \(x\) also lie in \(\bfS\).
\end{enumerate}
Recall that the \(\tau\)-stable HN factors of \(x\) are the subquotients in Jorden--H\"older filtrations of the \(\tau\)-semistable factors of \(x\).
Examples of such \(\bfS\) include:
\begin{enumerate}
\item[(a)] the union over \(\tau\) of \(\tau\)-stable objects,
\item[(b)] if \(\mathcal{C}\) is 2-Calabi--Yau, the set of spherical objects.
\end{enumerate}
In the second example, the condition~\eqref{hn} holds thanks to the Mukai lemma~\cite[Lemma 2.2]{huy:12}.
\begin{corollary}\label{cor:mass-determines-spherical-stables}
  Let \(\bf S\) be a set of objects of \(\mathcal{C}\) satisfying the conditions~\eqref{iso} and~\eqref{hn}.
  Let \(\tau\) and \(\tau'\) be two stability conditions such that for all \(x \in \bfS\), we have $m_{q,\tau}(x) = m_{q,\tau'}(x)$.
  Then an object \(x \in \bfS\) is $\tau$-stable if and only if it is $\tau'$-stable.
\end{corollary}
\begin{proof}
  The proof of~\Cref{cor:mass-determines-stables} goes through verbatim.
\end{proof}

\begin{proposition}\label{prop:generalinj}
   Let $\tau$ and $\tau'$ be stability conditions.
   Suppose that $q_{1}$ and $q_{2}$ are two distinct positive real numbers such that for each object \(x\) we have
   \begin{align*}
     m_{q_1,\tau}(x) &= m_{q_1,\tau'}(x)\text{ and}\\
     m_{q_2,\tau}(x) &= m_{q_2,\tau'}(x).
   \end{align*}
   Then $\tau = \tau'$.
 \end{proposition}
 \begin{proof}
   Let \(x\) be any stable object of \(\tau\).
   By~\Cref{cor:mass-determines-stables}, \(x\) is also a stable object of \(\tau'\).
   Let \(\phi\) be the \(\tau\)-phase of \(x\) and \(\phi'\) be the \(\tau'\)-phase of \(x\).
   We have
   \begin{align*}
     q_1^{\phi} |Z_{\tau}(x)| &= q_1^{\phi'}|Z_{\tau'}(x)| \text{ and }\\
     q_2^{\phi} |Z_{\tau}(x)| &= q_2^{\phi'}|Z_{\tau'}(x)|.
   \end{align*}
   Noting that both \(|Z_{\tau}(x)|\) and \(|Z_{\tau'}(x)|\) are nonzero, we have
   \((q_1/q_2)^{\phi} = (q_1/q_2)^{\phi'}\), which gives \(\phi = \phi'\) and hence \(|Z_{\tau}(x)| = |Z_{\tau'}(x)|\).
   We conclude that the phases and central charges of \(x\) in both \(\tau\) and \(\tau'\) are equal.
   Since this holds for all stable objects \(x\), we must have \(\tau = \tau'\).
 \end{proof}

\section{The projective embedding and its boundary}\label{sec:proj-embedding}
Let \(\mathcal{C}\) be a triangulated category and \(\Stab(\mathcal{C})\) the space of stability conditions on \(\mathcal{C}\).
Let \(\bfS\) be a non-empty set of non-zero objects of \(\mathcal{C}\).
Let \(\mathbf{R}^{\bfS}\) be the space of functions from \(\bfS \to \mathbf{R}\) with the product topology, and set
\[ \mathbf{P}^{\bfS} = \left(\mathbf{R}^{\bfS} \setminus \{0\}\right) / \text{scaling}.\]
Fix a positive real number \(q\).

Every stability condition \(\tau\) gives a non-zero function
\[ m_{q, \tau} \colon \bfS \to \mathbf{R}\]
that sends \(x \in \bfS\) to its \(\tau\)-mass \(m_{q,\tau}(x)\) defined in~\Cref{sec:mass}.
Changing \(\tau\) to a translate by the action of \(\C\) only rescales the function above.
As a result, we have a well-defined map
\begin{equation}\label{def:mass-map}
  m_q \colon \Stab(\mathcal{C})/\C \to \mathbf{P}^{\bfS},
\end{equation}
which we call the \emph{\(q\)-mass map}.
If \(q = 1\), we drop it from the notation.
The choice of \(\bfS\) is flexible.
In our applications, it will consist of all spherical objects.

The aim of this section is to prove some general properties of the mass map.
\subsection{Pre-compactness}
Recall that an object \(p \in \mathcal{C}\) is called a (classical) generator if every object \(y \in \mathcal{C}\) is a direct summand of some object \(x \in \mathcal{C}\) that admits a filtration
\[ 0 \to x_{0} \to \cdots \to x_{m} = x,\]
whose factors are shifts of \(p\).
\begin{proposition}[Pre-compactness]\label{prop:pre-compactness}
  Suppose the set \(\bfS\) contains objects \(p_{1}, \dots, p_{n}\) such that \(p = \oplus p_{i}\) is a generator of \(\mathcal{C}\).
  Then the closure of the image of \(m_{q}\) in \(\mathbf{P}^{\bfS}\) is compact.
\end{proposition}
\begin{proof}
  We first show that for every \(y \in \mathcal{C}\), there is some positive real number \(N(y)\) such that for every stability condition \(\tau\), we have
  \begin{equation}
    \label{eq:global-mass-bound}
    m_{q,\tau}(y) \leq N(y) m_{q,\tau}(p).
  \end{equation}

  To see this, take \(y \in \mathcal{C}\).
  Consider an object \(x \in \mathcal{C}\) such that \(y\) is a direct summand of \(x\), and 
 \(x\) has a filtration
  \[ 0 = x_0 \to x_1 \to \cdots \to x_{m} = x\]
  whose factors are  \(p[a_0], \dots, p[a_m]\).
  By the triangle inequality (\Cref{prop:triangle}), we have for every stability condition \(\tau\) that
  \[m_{q,\tau}(x) \leq (q^{a_0} + \cdots + q^{a_m}) m_{q,\tau}(p).\]
  Since \(y\) is a summand of \(x\), we also have \(m_{q,\tau}(y) \leq m_{q,\tau}(x)\).
  Taking, \(N(y)  = (q^{a_0} + \cdots + q^{a_m})\), we have the desired inequality~\eqref{eq:global-mass-bound}.

  Note in particular by~\eqref{eq:global-mass-bound} that \(m_{q,\tau}(p) > 0\).
  Let $\widetilde m_q \from \Stab(\mathcal C)/\mathbf{C} \to \R^{\bfS}$ be the lift of $m_q$ characterised by
  \[ \widetilde{m}_{q,\tau}(p) = \sum_{i}\widetilde{m}_{q,\tau}(p_i) = 1\]
  for every \(\tau \in \Stab(\mathcal{C})/\mathbf{C}\).
  Let $\widetilde B$ be the closure in $\R^{\bfS}$ of the image of $\widetilde m_{q}$.
  By~\eqref{eq:global-mass-bound}, we see that the image of \(\widetilde m_{q}\) lies in the compact set
  \[ \prod_{y \in \bfS}[0, N(y)].\]
  Thus, \(\widetilde B\) is a closed subset of a compact set, and hence compact.
  Note that by continuity, we have for all \(\mu \in \widetilde B\) that \(\sum \mu(p_{i}) = 1\), and therefore \(\widetilde B \subset \R^{\bfS} - 0\).

  Let \(\pi \colon (\R^{\bfS} - 0) \to {\mathbf P}^{\bfS}\) be the projection.
  Since \(\widetilde B\) is compact, so is \(\pi(\widetilde B)\).
  By construction, \(m_{q}(\Stab(\mathcal{C}) / \mathbf{C})\) is contained in the compact set \(\pi(\widetilde B) \subset {\mathbf P}^{\bfS}\).
  Therefore, its closure is a closed subset of a compact set, and hence compact.
\end{proof}

\begin{corollary}
  Under the same hypotheses as~\Cref{prop:pre-compactness}, fix positive real numbers \(q_{1}, \dots, q_{n}\).
  Consider the map
  \[ m_{q_1} \times \cdots \times m_{q_{n}} \colon \Stab(\mathcal{C})/ \mathbf{C} \to \mathbf{P}^{\bfS} \times \cdots \times \mathbf{P}^{\bfS}.\]
  Then the closure of its image is compact.
\end{corollary}
\begin{proof}
  Let \(m = m_{q_1} \times \cdots \times m_{q_{n}}\).
  Then the image of \(m\) is contained in the product of the images of \(m_{q_{i}}\).
  Hence, the closure of the image of \(m\) is contained in the product of their closures, which is compact by~\Cref{prop:pre-compactness}.
The result follows.
\end{proof}

\subsection{Injectivity}
The following proposition says that if \(\bfS\) is sufficiently large, then two mass maps together determine the stability condition up to translation.
\begin{theorem}\label{thm:product-of-masses-injective}
  Assume that \(\bfS\) contains all the stable objects in every stability condition on \(\mathcal{C}\).
  Then for \(n \geq 2\), the map
  \[ m = m_{q_1} \times \cdots \times m_{q_{n}} \colon \Stab(\mathcal{C})/ \mathbf{C} \to \mathbf{P}^{\bfS} \times \cdots \times \mathbf{P}^{\bfS}\]
  is injective.
\end{theorem}
\begin{proof}
It is sufficient to show that \(m\) is injective for \(n = 2\); it then follows for any \(n \geq 2\).

Let $\tau, \tau'\in \Stab(\mathcal{C})$, and suppose $m(\tau) = m(\tau')$.
By using the action of $\C$ on $\Stab(\mathcal{C})$, we may replace $\tau'$ by a stability condition $\tau''$ so that for all \(i = 1,2\), we have
\[ m_{q_i}(\tau) = m_{q_i}(\tau'') \in \R^{\bfS}.\]

Let $x \in \bfS$ be stable for $\tau$.
By the argument in the proof of~\Cref{cor:mass-determines-stables}, $x$ is also stable for $\tau''$.
Now, by the argument in the proof of~\Cref{prop:generalinj}, it follows that \(\tau = \tau''\).
Thus $\tau$ and $\tau'$ agree up to the action of \(\C\).
\end{proof}
  \begin{remark}
  Suppose \(\mathcal{C}\) is 2-Calabi--Yau and the spherical objects in \(\mathcal{C}\) span its Grothendieck group.
  Then~\Cref{thm:product-of-masses-injective} holds even if we replace \(\bfS\) by its subset consisting only of \emph{spherical} objects.
  See~\Cref{cor:mass-determines-spherical-stables}.
 \end{remark}
 \begin{remark}\label{rem:injectivity-fails}
  In general, the mass map for a single \(q\), namely
  \[m_q \colon \Stab(\mathcal{C})/\mathbf{C} \to \mathbf{P}^{\bfS},\]
  need not be injective.
  See~\cite[\S~6.1]{kik.kos.ouc:22} for an example, or consider the following example.
  For simplicity, take  \(q = 1\).
  Consider the quiver \(1 \to 2\) of type \(A_2\), and let \(\mathcal{C}\) be the bounded derived category of finite-dimensional representations of this quiver.
  Let \(S_1\) and \(S_2\) be the irreducible representations \(k \to 0\) and \(0 \to k\) respectively.
  Let \(E\) be the representation \(k \xrightarrow{1} k\), which fits into the exact sequence
  \[0 \to S_2 \to E \to S_1 \to 0.\]
  Recall that any object of \(\mathcal{C}\) is a direct sum of triangulated shifts of copies of \(S_1\), \(S_2\), and \(E\).
    Consider a stability condition in which \(S_1\) and \(S_2\) are stable objects of the standard heart, such that
    \(\phi(S_1) < \phi(S_2)\).
    Due to this inequality, \(E\) is not semistable, and so
    \[m(E) = m(S_1) + m(S_2).\]
    Consider two stability conditions \(\tau\) and \(\tau'\) in which \(S_1\) and \(S_2\) are stable objects of the \([0,1)\) heart, satisfying the following property:
    \[Z_{\tau}(S_1) = Z_{\tau'}(S_1) = 1, \text{ and}\]
    both \(Z_{\tau}(S_2)\) and \(Z_{\tau'}(S_2)\) are complex numbers in the upper half plane such that 
    \[Z_{\tau'}(S_2) = e^{ic}Z_{\tau}(S_2)\]
    for some small \(c > 0\).
    That is, \(Z_{\tau'}(S_2)\) is a slight rotation of \(Z_{\tau}(S_1)\).
    Then it is clear that \(\tau\) and \(\tau'\) are not equal points of \(\Stab(\mathcal{C})/\mathbf{C}\), whereas \(m_{\tau} = m_{\tau'}\).
\end{remark}

\begin{remark}
  For specific categories \(\mathcal{C}\), we can prove stronger injectivity results.
  For example, for the 2-CY categories \(\mathcal{C}_{\Gamma}\) associated to a connected graph \(\Gamma\), we prove in~\Cref{prop:injectivity-general} that the mass map for a single \(q\) is injective.
\end{remark}

\begin{remark}
  We would like the map \(m\) from~\Cref{thm:product-of-masses-injective} to be not only injective, but also a homeomorphism onto its image.
  Since \(m\) is continuous and injective, what remains to prove is that it maps open subsets of the domain to open subsets of the image.  
  At present we do not know how to prove this in general.  However, in examples we find that every stability condition $\tau$ has an open neighborhood $U$ whose closure is compact and such that $m(U)$ is cut out by finitely many (strict) inequalities involving the masses of finitely many objects in $S$.
  This implies that $m$ is an open map, and hence a homeomorphism onto the image. While we haven't formulated a general statement regarding the origin of such inequalities here, see~\Cref{prop:conditional-local-homeo} for a version written explicitly for 2-CY categories of quivers.  In particular, we use~\Cref{prop:conditional-local-homeo} to prove that a single mass map $m=m_q$ is a homeomophism onto its image in types \(A_{2}\) and \(\widehat{A}_{1}\) in~\Cref{prop:homeo1} and~\Cref{prop:homeo2} respectively.
\end{remark}

\subsection{Points in the boundary of the mass map}\label{sec:boundary}
In this section, we obtain a family of elements in \(\mathbf{P}^{\bfS}\) in the closure of the image of the stability manifold.
These are analogous to the intersection functionals in Teichm\"uller theory, and arise as limits under successive applications of a spherical twist.

We assume that the category \(\mathcal{C}\) is \(\k\)-linear of finite type with a fixed dg-enhancement.

\begin{definition}\label{def:hombar}
  Let \(a\) be a \(d\)-spherical object in \(\mathcal{C}\) and let \(q > 0\) be a fixed real number.
  Let \(x\) be any object of \(\mathcal{C}\) such that
  \[\Hom^i(x,x) = 0\text{ if }i < 0.\]
  \begin{enumerate}
  \item If \(x\) is indecomposable, set
    \[\homBar_q(a,x) =
      \begin{cases}
        \sum_i q^{-i}\dim(\Hom^i(a,x))&x \neq a[j] \text{ for any } j \in \mathbf{Z},\\
        q^j \cdot |q^{-d}-q^{-1}|&x \cong a[j].
      \end{cases}
    \]
  \item If \(x = \oplus_i y_i\) for indecomposable objects \(y_i\), set
    \[\homBar_q(a,x) = \sum_i \homBar_{q}(a,y_i).\]
  \end{enumerate}
\end{definition}
We should regard \(\homBar_{q}(x,y)\) as a categorical version of the minimal intersection number between two curves.
Then, in analogy with~\cite[\S 3.3]{fat.lau.poe:12}), we can regard the function \(\homBar_{q}(a,-) \colon \bfS \to \R\) as the intersection functional.

The following theorem shows that the quantity \(\homBar_q(a,x)\) governs the growth of the \(\tau\)-mass of an object under repeated applications of \(\sigma_a\).
\begin{theorem}\label{prop:limit-mass}
  Let \(\tau\) be a stability condition on \(\mathcal{C}\).
  Let \(a\) be a \(\tau\)-semi-stable d-spherical object, and let \(x\) be an object that has no endomorphisms of negative degree.
  Let \(q > 0\) and set $e=d-1$.
  \begin{enumerate}
  \item If \(0 < q \leq 1\), then
    \[
      \lim_{n \to \infty} \frac{m_{q,\tau}(\sigma_a^nx)}{q + q^{1-e} + \cdots + q^{1-(n-1)e}} = m_{q,\tau}(a) \cdot \homBar_q(a,x).
    \]
  \item If \(q \geq 1\), then
    \[
      \lim_{n \to \infty} \frac{m_{q,\tau}(\sigma_a^{-n}x)}{q^{1+e} + q^{1+2e} + \cdots + q^{1+ne}} = m_{q,\tau}(a) \cdot \homBar_q(a,x).
    \]
  \end{enumerate}
  In particular, if \(q = 1\), then
  \begin{equation}\label{eq:twist-limit}
    \lim_{n \to \pm\infty} \frac{m_\tau(\sigma_a^nx)}{n} = m_{\tau}(a) \cdot \homBar(a,x).
  \end{equation}
\end{theorem}

The rest of~\Cref{sec:boundary} is devoted to the proof of~\Cref{prop:limit-mass}.

Recall that for a \(d\)-spherical object $a$, we have a perfect pairing
\[ \Hom^i(x,y) \otimes \Hom^{d-i}(y,x) \to \k,\]
which is obtained by the composition followed by a map
\[ \Tr \colon \Hom^{d}(x,x) \to \k.\]
The graded vector space $\gradedHom(a,a)$ is two-dimensional, generated by the identity map $\id_a$ of degree zero, and a map $\loopmap_a$ of degree \(d\), which we choose so that \(\Tr (\loopmap_{a}) = 1\).

Let \(x\) be any object.
Then we have the distinguished triangle
\[ \gradedHom(a,x) \otimes a \xrightarrow{\alpha} x \to \sigma_a(x) \xrightarrow{+1},\]
where \(\alpha\) is the evaluation map.
Likewise, we have the distinguished triangle
\[ \sigma_{a}^{-1}x \to x \xrightarrow{\beta} \gradedHom(x,a)^{*} \otimes a \xrightarrow{+1},\]
where \(\beta\) is the co-evaluation map.
The two triangles together give the filtration
\begin{equation}\label{eq:twotwists}
  \sigma_{a}^{-1}x \to x \to \sigma_{a}x
\end{equation}
with factors \(\gradedHom(x,a)^{*} \otimes a\) and \(\gradedHom(a,x) \otimes a[1]\).
Using that \(a\) is \(d\)-spherical, we identify
\begin{equation}\label{eq:d-duality}
  \gradedHom(x,a)^{*} = \gradedHom(a,x)[d].
\end{equation}
\begin{lemma}\label{lem:connecting-map}
  In the setup above, suppose \(x\) does not contain any shift of \(a\) as a direct summand.
  Consider the filtration
  \[ \sigma_{a}^{-1}x \to x \to \sigma_{a}x.\]
  Then, with the identification as in~\eqref{eq:d-duality}, the connecting map
  \begin{equation}\label{eq:orig}
    \gradedHom(a, x) \otimes a [1] \to \gradedHom(a, x) \otimes a[d+1]
  \end{equation}
  is a shift of $\id \otimes \loopmap_a$.
\end{lemma}
\begin{proof}
  Shifted by \(-1\), the connecting map is the composite 
  \begin{equation}
    \gradedHom(a,x) \otimes a \xrightarrow{\alpha} x \xrightarrow{\beta} \gradedHom(x,a)^{*} \otimes a \xrightarrow{\gamma} \gradedHom(a,x)[d] \otimes a,
  \end{equation}
  where \(\alpha\) is evaluation, \(\beta\) is co-evaluation, and \(\gamma\) is the duality isomorphism.
  By the definition of evaluation and co-evaluation, \(\beta \circ \alpha\) is adjoint to the composite
  \[ \gradedHom(a,x) \otimes \gradedHom(x,a) \otimes a \to  \gradedHom(a,a) \otimes a \to a,\]
  where the first map is composition and the second is evaluation.
  Since \(x\) does not have any shift of \(a\) as a direct summand, no composition \(a[j] \to x \to a[j]\) can be the identity.
  Therefore, the image of the composition map
  \begin{equation}\label{eq:composition}
    \gradedHom(a,x) \otimes \gradedHom(x,a) \to \gradedHom(a,a)
  \end{equation}
  is the one-dimensional subspace spanned by \(\loopmap_{a}\).
  Interpreting \(\Tr \colon \Hom^{i}(a,a) \to \k\) as 0 on the summands exept \(\Hom^d(a,a)\), we can then write~\eqref{eq:composition} as
  \[ f \otimes g \mapsto \Tr(g \circ f) \cdot \loopmap_{a},\]
 The pairing that induces the duality isomorphism \[\gamma \colon \gradedHom(x,a)^{*} \cong \gradedHom(a,x)[d]\] is also given by
  \[ (f,g) \mapsto \Tr (g \circ f).\]
  It follows that the map \(\gamma \circ \beta \circ \alpha\) is \(\id \otimes \loopmap_{a}\).
\end{proof}

We recall the objects \(a_n\) studied in~\Cref{sec:an}.
The object \(a_n\) is characterised uniquely by the existence of a filtration with factors \(a[0], a[-(d-1)], \dots, a[-n(d-1)]\) such that the connecting map
\[ a[-i(d-1)] \to a[-(i-1)(d-1)]\]
is a shift of \(\loopmap_a\) (\Cref{prop:an-objects}).
By~\Cref{lem:twomaps}, we have non-zero maps
\begin{equation}\label{eq:twomaps}
  \begin{split}
    i_n &\colon a \to a_n, \quad l_n \colon a[-n(d-1)-d] \to a_n\\
    t_n &\colon a_n \to a[d], \quad q_n \colon a_n \to a[-n(d-1)]
  \end{split}
\end{equation}
forming distinguished triangles
\begin{equation}
  \begin{split}
    &a_{n-1}[-d] \xrightarrow{t_{n-1}[-d]} a \xrightarrow{i_n} a_n \xrightarrow{+1} \text{ and }\\
    &a_n\xrightarrow{q_n} a[-n(d-1)] \xrightarrow{l_{n-1}[1]} a_{n-1}[1] \xrightarrow{+1}.
  \end{split}
\end{equation}
We have
\begin{equation}
  \label{eq:composites}
  t_{n} \circ i_{n} = \loopmap_{a} \text{ and } q_{n} \circ l_{n} = \loopmap_{a}[-n(d-1)-d].
\end{equation}
The other composites are zero.
Up to scaling and shifts, the maps in~\eqref{eq:twomaps} are the only non-zero maps between \(a\) and \(a_{n}\).
  
\begin{lemma}\label{lem:tails}
  Let \(x \in \calC\) be any object that does not contain a shift of \(a\) as a direct summand and which does not have endomorphisms of negative degree.
  For every \(n \geq 0\), the natural map \(x \to \sigma_a^{n+1}x\) completes to a distinguished triangle
  \[ \gradedHom(a,x) \otimes a_n \to x \to \sigma_a^{n+1}x.\]
\end{lemma}
\begin{proof}
  Recall that we have the triangle
  \begin{equation}\label{eq:one-twist}
    \gradedHom(a,x) \otimes a \to x \to \sigma_a x \xrightarrow{+1}.
  \end{equation}
  By repeatedly applying \(\sigma_a\) to the triangle above, we get the triangles
  \begin{equation}
    \gradedHom(a,x) \otimes a[-i(d-1)] \to \sigma_a^i x \to \sigma_a^{i+1} x \xrightarrow{+1},
  \end{equation}
  which assemble into the diagram
  \begin{equation}\label{eq:long-twist-diagram}
    \begin{tikzcd}[column sep=0cm]
      x \ar{rr}&& \sigma_a x \ar{rr}\ar{dl}&& \cdots \ar{rr} && \sigma_a^nx \ar{rr}&& \sigma_a^{n+1} x \ar{dl} \\
      & \gradedHom(a,x) \otimes a[1] \ar[dashed]{ul} && \cdots && \cdots && \gradedHom(a,x) \otimes a[1-n(d-1)] \ar[dashed]{ul}.
    \end{tikzcd}
  \end{equation}
  By~\Cref{lem:connecting-map}, the connecting map between any two successive factors in~\eqref{eq:long-twist-diagram} is a shift of \(\id \otimes \loopmap_a\).
  Suppose \(y_n\) completes \(x \to \sigma_a^{n+1}x\) to a distinguished triangle
  \[ x \to \sigma_a^{n+1}x \to y_n \xrightarrow{+1}.\]
    We prove the following two statements by induction on \(n\).
  \begin{enumerate}
  \item We have an isomorphism \(y_n \cong \gradedHom(a,x) \otimes a_n[1]\).
  \item Via the above isomorphism, the map
    \begin{equation}\label{eqn:lastconmap}
      y_n \to \gradedHom(a,x) \otimes a[1-n(d-1)]
    \end{equation}
    induced by the diagram~\eqref{eq:long-twist-diagram} is a shift of \(\id \otimes q_{n+1}\).
  \end{enumerate}
  The base case \(n = 0\) follows from the distinguished triangle~\eqref{eq:one-twist}.
  Assume the result for \(n-1\).
  Then~\eqref{eq:long-twist-diagram} collapses to
  \[
    \begin{tikzcd}[column sep=0cm]
      x \ar{rr}&& \sigma_a^{n} x \ar{rr}\ar{dl}&& \sigma_a^{n+1} x \ar{dl} \\
      & \gradedHom(a,x) \otimes a_{n-1}[1]\ar[dashed]{ul} && \gradedHom(a,x) \otimes a[1-n(d-1)] \ar[dashed]{ul}.
    \end{tikzcd}
  \]
  We now prove that the connecting map
  \begin{equation}\label{eqn:bigcon}
    \gradedHom(a,x) \otimes a[1-n(d-1)] \to \gradedHom(a,x) \otimes a_{n-1}[2]
  \end{equation}
  is \(\id \otimes l_{n-1}[2]\).

  Chose a homogeneous basis of \(\gradedHom(a,x)\) so that the map~\eqref{eqn:bigcon} may be written as a matrix of maps from shifts of \(a\) to shifts of \(a_{n-1}\).
 We know that, up to scaling and shifts, the only two maps from \(a\) to \(a_{n-1}\) are
  \begin{enumerate}
  \item \(i_{n-1} \colon a \to a_{n-1}\), and
  \item \(l_{n-1} \colon a \to a_{n-1}[n(d-1)+1]\).
  \end{enumerate}
  The composite
  \[ \gradedHom(a,x) \otimes a[1-n(d-1)] \xrightarrow{\eqref{eqn:bigcon}} \gradedHom(a,x) \otimes a_{n-1}[2] \xrightarrow{\eqref{eqn:lastconmap}} \gradedHom(a,x) \otimes a[2-(n-1)(d-1)].\]
  is the connecting map in the last two steps of~\eqref{eq:long-twist-diagram}.
  By~\Cref{lem:connecting-map}, we know that it is a shift of \(\id \otimes \loopmap_a\).
  From~\eqref{eq:composites}, we have \(q_{n-1} \circ i_{n-1} = 0\) and \(q_{n-1} \circ l_{n-1} = \loopmap_a\).
  Therefore, we conclude that the diagonal entries of~\eqref{eqn:bigcon} must be \(l_{n-1}[2]\) and the off-diagonal entries must be multiples of shifts of \(i_{n-1}\), possibly zero.
  In fact, we need to prove that the off-diagonal entries are zero.
  To do so, we need another fact about \(a_{n-1}\) from the appendix: that the map \(i_{n-1}\) does not factor through \(\sigma_a^{n}x\) (\Cref{lem:i-does-not-factor}).
  We conclude that a non-zero multiple of a shift of \(i_{n-1}\) cannot be a summand of~\eqref{eqn:bigcon}, and so the off-diagonal entries are indeed zero.
  
  Having proved that~\eqref{eqn:bigcon} is a shift of \(\id \otimes l_{n-1}\), it follows from~\Cref{lem:twomaps} that \(y_{n+1}\) is isomorphic to \(\gradedHom(a,x) \otimes a_n\) and the map \(y_{n+1} \to \gradedHom(a,x) \otimes a[1-n(d-1)]\) is isomorphic to a shift of \(\id \otimes q_{n+1}\).
\end{proof}

We now employ the machinery of \emph{rectifiable filtrations}, as discussed in~\Cref{sec:rectifiable-filtrations}, specifically~\Cref{def:rectifiable-filtration}.
\begin{lemma}\label{lem:eventual-tail}
  Fix a stability condition \(\tau\).
  Let \(a\) be a \(\tau\)-semistable \(d\)-spherical object of \(\mathcal{C}\).
  Let $x \in \calC$ be any object that does not have endomorphisms of negative degree and does not have any shift of $a$ as a direct summand.
  There exists $N$ (depending on $x$ and $a$) such that for every $n \geq N$, the filtration
  \[ 0 \to \sigma_a^{n}x \to \sigma_a^{n+1}x  \]
  is rectifiable (with respect to \(\tau\)).
\end{lemma}
\begin{proof}
  The factors of the filtration are \(\sigma_a^nx\) and \(\gradedHom(a,x) \otimes a[1-n(d-1)]\).
  We need to show that the connecting map
  \[ \gradedHom(a, x) \otimes a[1-n(d-1)] \to \sigma_a^n x[1]\]
  is rectifiable.
  Let us show, equivalently, that its shift
  \[\chi \colon \gradedHom(a, x) \otimes a[-n(d-1)] \to \sigma_a^n x \]
  is rectifiable.

  Let $\alpha$ be a real number.
  We need to show that the map
  \begin{equation}\label{eq:chi}
    \left(\gradedHom(a,x) \otimes a[-n(d-1)]\right)_{\geq \alpha} \to (\sigma_a^nx)_{< \alpha+1}
  \end{equation}
  induced by $\chi$ is zero.
  Without loss of generality, assume that \(x_{\leq 0} = 0\), which can be achieved by replacing \(x\) by a sufficiently positive shift.

  First assume that \(\alpha \geq -1\).
  Then, if \(n\) is large enough, we have
  \[ \left(\gradedHom(a,x) \otimes a[-n(d-1)]\right)_{\geq \alpha} = 0.\]
  So the map~\eqref{eq:chi} is also zero.
  
  Now assume \(\alpha < -1\).
  By~\Cref{lem:tails}, we have the exact triangle
  \begin{equation}\label{eq:xnx}
    \gradedHom(a,x) \otimes a_{n-1} \to x \to \sigma_a^n x \xrightarrow{+1}.
  \end{equation}
  It is easy to see that, \(\alpha < -1\) and \(x_{\leq 0} = 0\) imply that the \(+1\) map in~\eqref{eq:xnx} induces an isomorphism
  \[ \left(\sigma_{a}^nx\right)_{< \alpha + 1} \to \left(\gradedHom(a,x) \otimes a_{n-1}[1] \right)_{<\alpha+1}.\]
  So the map in~\eqref{eq:chi} is equal to the map
  \begin{equation}\label{eq:chi2}
\left(\gradedHom(a,x) \otimes a[-n(d-1)]\right)_{\geq \alpha} \to \left(\gradedHom(a,x) \otimes a_{n-1}[1]\right)_{< \alpha+1}.
\end{equation}
The map above is induced by the connecting map in the filtration
\[ x \to \sigma_{a}^{n}x \to \sigma_{a}^{n+1}x.\]
We have seen in the proof of~\Cref{lem:tails} that this map a shift of \(\id \otimes l_{n-1}\) (see~\eqref{eqn:bigcon}).
Note that \(l_{n-1}\) is rectifiable by~\Cref{ex:easyrect}~\eqref{i:highlow} and so is \(\id \otimes l_{n-1}\) by~\Cref{ex:easyrect}~\eqref{i:dsum}.
Therefore, the map~\eqref{eq:chi2} is zero.
\end{proof}

\begin{proof}[Proof of~\Cref{prop:limit-mass}]
  We prove the result for \(0 < q \leq 1\); the case of \(q \geq 1\) is analogous.
We abbreviate \(m_{q,\tau}\) by \(m\) and \(\homBar_{q}\) by \(\homBar\).
  
  First, let \(x \cong a[j]\).
  Then
  \[ \sigma_{a}^{n} x = a[j-ne],\]
  and hence
  \[ m(\sigma_{a}^{n}x) = q^{j-en} m(a).\]
  Therefore, 
  \begin{align*}
    \lim_{n \to \infty} \frac{m(\sigma_a^nx)}{q + q^{1-e} + \cdots + q^{1-(n-1)e}} &=   \lim_{n \to \infty} \frac{q^{j-en}}{q + q^{1-e} + \cdots + q^{1-(n-1)e}}\\
    &= q^{j}\left(q^{-d}-q^{-1}\right).
  \end{align*}

  Next, suppose \(x\) is indecomposable and not isomorphic to a shift of \(a\).
  Let $N$ be as guaranteed by~\Cref{lem:eventual-tail}.
  Then, for every $i \geq N$, we have
  \begin{align*}
    m(\sigma_a^{i+1}x) &= m(\sigma_a^ix) + m\left(\gradedHom(a,x) \otimes a[1-ie]\right) \\
    &= m(\sigma_a^ix) + q^{1-ie}\cdot m(a).
  \end{align*}
  By taking the sum as \(i\) ranges from \(N\) to \(n-1\), we get
  \[
m(\sigma_{a}^{n}x) = \left(q^{1-(n-1)e} + \cdots + q^{1-Ne}\right) \cdot m(a) + m(\sigma_{a}^{N}x).
    \]
    Dividing by \(q^{1-(n-1)e} + \cdot + q^{1}\) and letting \(n \to \infty\) yields
    \[
\lim_{n \to \infty}\frac{m(\sigma_{a}^{n}x)}{q+q^{1-e}+ \cdots + q^{1-(n-1)e}} = m(a) \cdot \homBar_{q}(a,x),
\]
as desired.

    Finally, if \(x\) is decomposable, then the desired equality follows by adding up the equalities for each indecomposable summand.
\end{proof}

Let \(\bfS\) be a set of objects of \(\mathcal{C}\) such that no object of \(\bfS\) has endomorphisms of negative degree.
Recall from~\eqref{def:mass-map} the mass map
\[ m_q \colon \Stab(\mathcal{C})/\mathbf{C} \to \mathbf{P}^{\bfS}.\]
We also have a map $h_q \from \bfS \to \mathbf{P}^\bfS$
defined as
\[ h_q \colon x \mapsto [\homBar_q(x,-)]\]
For \(q = 1\), we write \(m\) and \(h\) to mean \(m_q\) and \(h_q\) respectively.
\begin{corollary}\label{cor:limit-mass}
  Let \(a\) be a spherical object that is semi-stable in a stability condition \(\tau\).
  Then, in $\mathbf{P}^\bfS$ we have the following.
  \begin{enumerate}
  \item If \(0 < q \leq 1\), then
    \[\lim_{n\to\infty} m_q(\sigma_{a}^n x) = h_q(a).\]
      \item If \(q \geq 1\), then
        \[\lim_{n\to\infty} m_q(\sigma_{a}^{-n} x) = h_{q}(a).\]
      \end{enumerate}
      In particular, if \(q = 1\), then
      \[ \lim_{n \to \pm\infty} m(\sigma_{a}^{n}x) = h(a).\]
      In all cases, $h_{q}(a)$ is in the closure of the image of the mass map.
    \end{corollary}
    \begin{proof}
      Follows immediately from~\Cref{prop:limit-mass}.
    \end{proof}
    
\begin{remark}[Degenerate stability conditions]
\Cref{prop:limit-mass} describes one way to approach the boundary of the stability manifold in $\mathbf{P}^\bfS$, namely by applying powers of a spherical twist.
There are more direct ways to approach the boundary.
Start with a stability condition $\tau$ with heart $\heart$ and central charge $Z \from K(\calC) \to \mathbf C$.
Recall that $Z$ must map non-zero objects of $\heart$ to non-zero complex numbers.
Take a deformation $Z_t$ of $Z$ such that $Z_0$ sends some objects of $\heart$ to $0$.
Then the limit of the stability condition described by $(\heart, Z_t)$ is not a stability condition, but it is in the closure.
Such limits correspond to the degenerate stability conditions in the sense of~\cite{bol:20}.

In the analogy with the Teichm\"uller space, applying a spherical twist corresponds to applying a Dehn twist.
On the other hand, degenerating $Z$ corresponds to contracting a (collection of) curve(s).
In geometry, both operations yield the same limit (see, e.g.~\cite{fat.lau.poe:12}), but the limits can be distinct in the categorical setting (see, e.g., the results of~\cite{bap.bec.lic:23}).
\end{remark}

\section{Harder--Narasimhan (HN) automata}\label{sec:automata}
The aim of this section is to introduce some general machinery for a piecewise generalisation of a \(G\)-set (resp.\ a piecewise-linear generalisation of a \(G\)-representation).
In~\Cref{sec:a2-automaton,subsec:aonehat-automaton}, we apply this to the action of the auto-equivalence group on the Harder--Narasimhan filtrations of an object.
The reader may choose to defer reading the remainder of this section until required for the applications.

Let \(\Theta\) be a quiver with vertex set \(\Theta_0\) and edge set \(\Theta_1\).
For an edge \(\alpha \in \Theta_1\), let \(s(\alpha)\) and \(t(\alpha)\) denote its source and target, respectively.
A \emph{path} in \(\Theta\) is a finite sequence of edges \((\alpha_1, \ldots,\alpha_n)\), such that for each \(i\), we have
\[s(\alpha_{i+1}) = t(\alpha_i).\]
A \(\Theta\)-representation \({X}\) with values in a fixed category consists of a collection of objects of the category \(\{X_v \mid v \in \Theta_0\}\) indexed by vertices, and a collection of morphisms \(\{\phi(X)_{\alpha} \colon X_{s(\alpha)} \to X_{t(\alpha)} \mid \alpha \in \Theta_1\}\) indexed by edges.
A morphism of \(\Theta\)-representations from \({X}\) to \({Y}\) consists of morphisms
\(\{f_v \colon X_v \to Y_v \mid v \in \Theta_0\}\) such that all squares of the form
\[
  \begin{tikzcd}
    X_{s(\alpha)} \arrow{r}{f_{s(\alpha)}} \arrow{d}{\phi(X)_{\alpha}} & Y_{s(\alpha)} \arrow{d}{\phi(Y)_{\alpha}}\\
    X_{t(\alpha)} \arrow{r}{f_{t(\alpha)}} & Y_{t(\alpha)}
  \end{tikzcd}
\]
commute.
A \(\Theta\)-set \({S}\) is a representation of \(\Theta\) in the category of sets.
Similarly, if \(A\) is any ring, then a \(\Theta\)-representation \({M}\) of \(A\)-modules is a representation of \(\Theta\) in the category of \(A\)-modules.
Clearly, any \(\Theta\)-representation of \(A\)-modules is also a \(\Theta\)-set.
A \(\Theta\)-subset \({T} \subset {S}\) is a morphism of \(\Theta\)-sets in which for each \(v \in \Theta_0\), the corresponding map \(T_v \to S_v\) is an inclusion.

Fix a group \(G\).
We say that a function \(\ell \colon \Theta_1 \to G\) is a \emph{\(G\)-labelling} of \(\Theta\).
We can extend this function to paths in \(\Theta\) in a natural way, as follows.
If \((\alpha_1, \ldots, \alpha_n)\) is a path, then
\[\ell((\alpha_1, \ldots, \alpha_n)) = \ell(\alpha_n) \cdots \ell(\alpha_1).\]

Consider a quiver \(\Theta\) with a \(G\)-labelling \(\ell \colon \Theta_1 \to G\).
Let \(S\) be a set with a \(G\)-action.
These data naturally give rise to a \(\Theta\)-set \({S}\), in which the object associated to each vertex \(v \in \Theta_0\) is simply \(S_v = S\), and the map associated to any edge \(\alpha\) is the action of \(\ell(\alpha)\) on \(S\).
Similarly, a \(G\)-representation \(M\) naturally gives rise to a \(\Theta\)-representation.
In particular, if \(\mathcal{C}\) is a (triangulated) category with a \(G\)-action, then \({\ob \mathcal{C}}\) is a \(\Theta\)-set.

We now take up the question of evolutions of HN filtrations under group actions.
Let \(\mathcal{C}\) be a triangulated category and \(\tau\) a stability condition on \(\mathcal{C}\).
Let \(\Sigma\) be the set of all indecomposable \(\tau\)-semistable objects in the \([0,1)\)-heart.
\begin{definition}\label{def:hn-multiplicity-vector}
  Let \(x\) be any object of \(\mathcal{C}\).
  The \emph{\(\tau\)-Harder--Narasimhan multiplicity vector} \(\HN_{\tau}(x)\) is the element of \(h \in \Z^{\Sigma}\) defined as follows:
  \[ h(s) = \sum_{n \in \Z} \text{Number of occurrences of \(s[n]\) among the \(\tau\)-HN factors of \(x\)}.\]
\end{definition}
\begin{remark}
  The definition has a natural \(q\)-analog \(h = \HN_{q,\tau}(x) \in \Z[q^{\pm}]^{\Sigma}\) defined by
  \[ h(s) = \sum_{n \in \Z} (\text{Number of occurrences of \(s[n]\) in the \(\tau\)-HN filtration of \(x\)}) \cdot q^{n}.\]
\end{remark}

We now come to the key definition that will capture the piecewise linear nature of the evolution of HN multiplicities.
Recall the notation:
\begin{description}
\item[\(\mathcal C\)] a triangulated category,
\item[\(\tau\)] a stability condition on \(\mathcal C\),
\item[\(\Sigma\)] the set of indecomposable \(\tau\) semi-stable objects in the \([0,1)\)-heart,
  \item[\(G\)] a group acting on \(\mathcal{C}\).
\end{description}

\begin{definition}\label{def:hn-automaton}
  A \(\tau\)-HN automaton for \(\mathcal{C}\) as above consists of the following data.
  \begin{enumerate}
  \item A quiver \(\Theta\) together with a \(G\)-labelling \(\ell \colon \Theta_1 \to G\).
  \item A \(\Theta\)-subset \({S} \subset {\ob \mathcal{C}}\).
  \item\label{def:autkey} A \(\Theta\)-representation \( M\) of \(\Z\)-modules that assigns the module \(\Z^{\Sigma}\) to every vertex such that the HN-multiplicity vector
    \[  S \xrightarrow{\HN_{\tau}}  M\]
    defines a map of \(\Theta\)-sets.    
  \end{enumerate}
\end{definition}
Let us make condition~\eqref{def:autkey} explicit.
Let \(\alpha \colon v \to w\) be an edge of \(\Theta\).
Then condition~\eqref{def:autkey} means that the following diagram commutes:
\[
  \begin{tikzcd}
    S_v \ar{r}{HN_\tau}\ar{d}{\ell(\alpha)}&  \Z^{\Sigma}\ar{d}{M(\alpha)} \\
    S_w \ar{r}{HN_{\tau}}& \Z^{\Sigma}.
  \end{tikzcd}
  \]
  Thus, for the objects in \(S_v\), the matrix \(M(\alpha)\) captures the evolution of HN multiplicities under the application of \(\ell(\alpha)\).

  We obtain the mass of an object by applying a linear function to the HN-multiplicity vector.
  Hence, an HN automaton also allows us to understand the evolution of the mass.

  \begin{remark}\label{rem:automaton-recognising}
    We establish some terminology for HN automata.
    Consider a \(\tau\)-HN automaton as explained in~\Cref{def:hn-automaton}.
    Following standard conventions in the theory of finite automata, we call the vertices of \(\Theta\) the \emph{states} of the automaton.
    Given a state \(v\), we say that \(S_v\) is the set of objects \emph{supported at the state \(v\)}.
    Note that an object may be supported at multiple states.
    We say that an object is \emph{recognised by $\Theta$} if it is supported at some state.

    We say that an expression $g = g_n\cdots g_1$ in \(G\) is \emph{recognised by $\Theta$} if there is a path \((\alpha_1, \ldots, \alpha_n)\) in \(\Theta\) such that
    \(g_i = \ell(\alpha_i)\), and hence
    \[g = \ell((\alpha_1,\ldots,\alpha_n)).\]
    Finally, let \(x\) be an object of \(\mathcal{C}\) supported at a state \(v\). Suppose that
    \(y = g_n \cdots g_1 \cdot x\) for some group elements \(g_n, \ldots, g_1\), and suppose also that
    \(g_n \cdots g_1\) is a recognised expression via a path starting at the state \(v\).
    In this case, we say that the expression \(y = g_n \cdots g_1 \cdot x\) is \emph{recognised by \(\Theta\)}.
  \end{remark}

\section{The category \texorpdfstring{$\calC_{\Gamma}$}{CGamma}}\label{sec:cgamma}
In this section we recall the definition and properties of 2-Calabi--Yau categories associated to quivers.
Fix $\Gamma$, a connected undirected graph without loops or multiple edges.
\begin{definition}
The \emph{Artin--Tits braid group associated to $\Gamma$}, denoted by \(B_{\gamma}\) is the group generated by $\sigma_i$, where $i$ ranges over the set of vertices of $\Gamma$, modulo the following relations:
\begin{align*}
  \sigma_i \sigma_j \sigma_i &= \sigma_j \sigma_i \sigma_j \text{ if there is an edge between $i$ and $j$,} \\
  \sigma_i \sigma_j &= \sigma_j \sigma_i \text{ otherwise.}
\end{align*}
\end{definition}
\begin{definition}
  The \emph{Coxeter group associated to \(\Gamma\)}, denoted by \(W_{\gamma}\), is the quotient of \(B_{\Gamma}\) by the following additional relations.
  For each vertex \(i\) of \(\Gamma\), set
  \[\sigma_i^2 = 1.\]
  Denote the image of \(\sigma_i\) in \(W_{\Gamma}\) by \(s_i\).
\end{definition}
Let \(V_{\Gamma}\) be the \(\mathbf{Z}\)-module with basis vectors \(v_i\) indexed by the vertices of \(\Gamma\).
Define a bilinear form on \(V_{\Gamma}\) by the following formula:
\[
  \langle v_i, v_j \rangle =
  \begin{cases}2&\text{ if } i = j,\\
    -1&\text{ if }i\text{ and }j\text{ are connected by an edge,}\\
    0&\text{otherwise}.
  \end{cases}
\]

\begin{definition}
  The \emph{standard representation} of \(W_{\Gamma}\) is the action on \(V_{\Gamma}\) defined by the formula
  \[s_i(v_j) = v_j - \langle v_i, v_j \rangle v_i.\]
\end{definition}
We describe a category $\calC_\Gamma$ with a (weak) action of $B_\Gamma$ that categorifies the standard representation of \(W_{\Gamma}\).
The construction is via the zig-zag algebra, which was introduced in~\cite{hue.kho:01}.
A construction of \(\mathcal{C}_{\Gamma}\) for type $A_n$ graphs is in~\cite{sei.tho:01,tho:06}.

We first recall the zig-zag algebra.
A version of this construction is also in~\cite[\S 2.3]{bap.deo.lic:22}, but in that version, we have not been careful with signs; see~\Cref{rem:zig-zag-signs,rem:strong-2-CY,rem:strong-2-CY-end}.

Let $\Gamma^{\dbl}$ denote the doubled quiver of \(\Gamma\).
This is the directed graph obtaind by replacing each edge of \(\Gamma\) by a pair of oppositely oriented edges between its endpoints.
For each original edge \(e\) in \(\Gamma\) between \(i\) and \(j\), we let \(e_{ij}\) and \(e_{ji}\) be the corresponding arrows in \(\Gamma^{\dbl}\) from \(i\) to \(j\) and \(j\) to \(i\) respectively.
Let $\Path(\Gamma^{\dbl})$ denote the path algebra of $\Gamma^{\dbl}$.  As a vector space over $\k$, the algebra $\Path(\Gamma^{\dbl})$ is spanned by all oriented paths in $\Gamma^{\dbl}$, and the multiplication is given by concatenation.
The length function on paths that declares edges to have length one induces a non-negative grading on $\Path(\Gamma^{\dbl})$.

We define the zigzag algebra of $\Gamma$, as follows.
\begin{definition}\label{def:zig-zag-algebra}
  Suppose $\Gamma$ has at least two vertices.
  Fix a sign \(\sgn_{ij} \in \{+1, -1\}\) for each arrow \(e_{ij}\) in \(\Gamma^{\dbl}\) such that \(\sgn_{ij} = - \sgn_{ji}\).
  Let $A(\Gamma) = A_{\sgn}(\Gamma)$ be the quotient of the path algebra \(\Path(\Gamma^{\dbl})\) by the two-sided ideal generated by the following elements:
  \begin{enumerate}
  \item all length three paths,
  \item length two paths whose source and target differ,
  \item for vertices \(i,j,k\) of \(\Gamma\) such that there are edges \(e_{ij}\) and \(e_{ik}\), the element
    \[\sgn_{ij} e_{ij}e_{ji} - \sgn_{ik} e_{ik}e_{ki}.\]
  \end{enumerate}
  If $\Gamma$ has one vertex, (type $A_1$), set $A(\Gamma) = \mathbf{C}[x]/x^2$.
  We call $A(\Gamma)$ the \emph{zigzag algebra} of $\Gamma$.
\end{definition}
Note that the third relations in the definition above imply that all degree-two paths starting and ending at the same vertex are equal up to sign.
We call one such non-zero path a ``loop'' at a vertex.
When \(\Gamma\) has at least two vertices, the relations of \(A(\Gamma)\) are all homogeneous, and so the natural grading on \(\Path(\Gamma^{\dbl})\) descends to a grading on $A(\Gamma)$.
When \(\Gamma\) has a single vertex, grade \(A(\Gamma) = \mathbf{C}[x]/x^2\) by setting $\deg(x)=2$.

The choice of signs does not matter.
\begin{proposition}\label{prop:zig-zag-orientations}
Consider another choice of signs \(s'_{ij} \in \{+1, -1\}\) such that \(s'_{ij} = -s'_{ji}\).
  Then the map
  \[ e_{ij}   \mapsto \sgn_{ij}\sgn'_{ij} e_{ij}\]
  induces an isomorphism of graded algebras \(A_{\sgn}(\Gamma) \cong A_{\sgn'}(\Gamma)\).
\end{proposition}
\begin{proof}
  We leave the proof to the reader.
\end{proof}

\begin{remark}\label{rem:zig-zag-signs}
  In~\cite{bap.deo.lic:22}, we gave a definition of the zig-zag algebra that ignored the choice of signs in~\Cref{def:zig-zag-algebra}, in effect taking \(\sgn_{ij} = \sgn_{ji} = 1\).
  This unsigned version also appears in~\cite{hue.kho:01} and~\cite{tho:06} for type \(A\).  
  Let \(B(\Gamma)\) denote this unsigned version of the zig-zag algebra, namely the algebra with the same generators and relations as \(A(\Gamma)\) from~\Cref{def:zig-zag-algebra}, but in which we set \(s_{ij} = 1\).

  We prefer to work with the algebra \(A(\Gamma)\) in this paper, as it gives rise to a strongly 2-Calabi--Yau category \(\mathcal{C}_{\Gamma}\); see~\Cref{rem:strong-2-CY,rem:strong-2-CY-end}.
  However, if the unoriented graph \(\Gamma\) is bipartite, the two definitions are equivalent, as stated in~\Cref{prop:zigzag-signed-unsigned}.
  In particular, the definitions are equivalent for finite type \(ADE\).
\end{remark}
The following is easy to check; we omit the proof.
\begin{proposition}\label{prop:zigzag-signed-unsigned}
  Let \(\Gamma = \{\Gamma_1, \Gamma_2\}\) be an unoriented bipartite graph without self-loops or multiple edges, in which each edge connects a vertex of \(\Gamma_1\) with a vertex of \(\Gamma_2\).
  Let \(A(\Gamma)\) be the zig-zag algebra as in~\Cref{def:zig-zag-algebra} and \(B(\Gamma)\) the unsigned zig-zag algebra.
  Consider the map from \(A(\Gamma)\) to \(B(\Gamma)\) in which
  \[e_{ij} \mapsto
    \begin{cases}
    \sgn_{ij}e_{ij} &\text{ if \(i \in \Gamma_1\) and \(j \in \Gamma_2\)}, \\
     e_{ij} & \text{ if \(i \in \Gamma_2\) and \(j \in \Gamma_1\)}.
                   \end{cases}
                 \]
  This map is an algebra isomorphism.
\end{proposition}
The minimal idempotents of $A(\Gamma)$ are the length zero paths $(i)$ for $i \in \Gamma$.
The modules $P_i = A(\Gamma) (i)$ are indecomposable projective left $A(\Gamma)$ modules.
Any finitely-generated graded projective left $A(\Gamma)$ module is isomorphic to a direct sum of grading shifts of the modules $P_i$.

The category of complexes of graded projective left \(A(\Gamma)\) modules admits two ``shift'' functors, namely the homological shift and the grading shift.
We work in a simpler setting where the two shifts are identified.
We construct the desired category as follows.

Regard \(A(\Gamma)\) as a differential graded algebra (dga) where all differentials are zero.
Let \(K(\dgmod A)\) be the category of finite-dimensional differential graded modules (dgms) over the dga \(A(\Gamma)\).
Morphisms in \(K(\dgmod A)\) are homotopy classes of chain maps that are compatible with the action of \(A(\Gamma)\).
Then \(K(\dgmod A)\) is a triangulated category (see, e.g.~\cite[\S 22.8]{the:22}).
\begin{definition}
  Set \(\mathcal{C}_{\Gamma}\) to be the smallest full and strict triangulated subcategory generated by the objects \(P_i\) as \(i\) ranges over the vertices of \(\Gamma\).
\end{definition}

\begin{remark}
The categories considered in~\cite{sei.tho:01} and~\cite{tho:06} are described in a slightly different way.
However,~\cite[Proposition 2.1]{bap.deo.lic:22} shows that they coincide with our category $\calC_{\Gamma}$.
The category considered in~\cite{hue.kho:01} is the category of graded projective \(A(\Gamma)\)-modules; see~\cite[\S~2.3.3]{bap.deo.lic:22} for a discussion on how it is related to \(\mathcal{C}_{\Gamma}\).
\end{remark}

We state the salient properties of $\calC_\Gamma$ from~\cite[\S~2.3]{bap.deo.lic:22}.
\begin{enumerate}
\item $\calC_\Gamma$ is a $\k$-linear triangulated category.
\item $\calC_\Gamma$ is (strongly) 2-Calabi--Yau.
  That is, for every pair of objects $x,y \in \mathcal C$ and \(i \in \mathbf{Z}\), we have a natural non-degenerate graded symmetric pairing
  \[ \Hom^i(x,y) \otimes \gradedHom^{2-i}(y,x) \to \k.\]
  That is, the pairing is symmetric for even \(i\) and skew-symmetric for odd \(i\).
\item $\mathcal C_{\Gamma}$ is classically generated by the objects $P_i$.
  These objects satisfy
  \begin{align*}
    \gradedHom(P_i, P_i[n]) &=
    \begin{cases}
      \k & \text{ if $n=0 $ or $n=2$,}\\
      0 & \text{ otherwise};
    \end{cases}\\
    \gradedHom(P_i, P_j[n]) &=
    \begin{cases}
      \k & \text{ if $n=1$ and $i$ and $j$ are neighbours,}\\
      0 & \text{ otherwise, for $i\neq j$}.
    \end{cases}
  \end{align*}
\item The extension closure of the objects \(P_i\) is an abelian category, which is the heart of a bounded t-structure on \(\mathcal{C}_{\Gamma}\).
  We call this t-structure the \emph{standard t-structure} on \(\mathcal{C}_{\Gamma}\), and refer to its heart as the \emph{standard heart} \(\heart_{\std}\).
\end{enumerate}
The Grothendieck group $K_\Gamma$ of $\mathcal{C}_\Gamma$ is a \(\mathbf{Z}\)-module with a pairing given by
\[ \langle  [x], [y] \rangle = \bigoplus_{i} (-1)^{i} \dim \Hom^{i}(x,y).\]
It is easy to check that we have an isomorphism \(V_{\Gamma} \to K_{\Gamma}\) compatible with the pairing, defined by \(v_i \mapsto [P_i]\).
\begin{remark}\label{rem:strong-2-CY}
  The pairing 
  \[\Hom(x,y) \otimes \Hom^{2-i}(y,x) \to \k\]
  is defined exactly as in~\cite[\S~2.3]{bap.deo.lic:22}.
  The choice of signs in~\Cref{def:zig-zag-algebra} gives the graded symmetry, and hence a strongly 2-Calabi--Yau category (see~\cite{kel:08}).
\end{remark}

\begin{remark}\label{rem:strong-2-CY-end}
  The zig-zag algebra \(A(\Gamma)\) arises naturally as follows.
  Let \(\{P_i \mid i \in \Gamma\}\) be a collection of spherical objects in a strongly 2-CY triangulated category.
  Assume that for \(i \neq j\), we have
  \[\Hom^1(P_i, P_j) \cong \begin{cases}
                         \k & \text{ if \(ij\) is an edge of \(\Gamma\),}\\
                         0 & \text{ otherwise}.
                       \end{cases}
                     \]
                     Set \(P = \bigoplus_{i \in \Gamma} P_i\).
 Then the endomorphism algebra of \(P\) is isomorphic to the zig-zag algebra \(A(\Gamma)\) as defined in~\Cref{def:zig-zag-algebra}.
 We describe an explicit isomorphism
 \[ \psi \colon A(\Gamma) \to \End(P).\]
 Choose a sign \(s_{ij} \in \{\pm 1\}\) for each edge \(ij\) of \(\Gamma^{\dbl}\) such that \(s_{ij} = -s_{ji}\).
 For edge \(ij\) with \(s_{ij} = 1\), let \(\psi(e_{ij}) \in \Hom^1(P_i,P_j)\) be any non-zero element and \(\psi(e_{ji})\) be the unique element such that 
 \[ \langle \psi(e_{ij}), \psi(e_{ji}) \rangle = 1.\]
 Then it is easy to check that \(\psi \colon A(\Gamma) \to \End(P)\) an isomorphism of graded algebras.  
\end{remark}

Recall that the indecomposable projective modules $P_i$ in \(\mathcal{C}_{\Gamma}\) are spherical.
Furthermore, the spherical twists in the $P_i$ satisfy the defining relations of the Artin--Tits braid group $B_\Gamma$ associated to $\Gamma$ (see, e.g.,~\cite{hue.kho:01}).
That is,
\begin{align*}
  \sigma_{P_i}\sigma_{P_j}\sigma_{P_i} &\cong\sigma_{P_j}\sigma_{P_i}\sigma_{P_j} \text{ when $i$ and $j$ are connected by an edge of $\Gamma$},\\ 
  \sigma_{P_i}\sigma_{P_j} &\cong \sigma_{P_j}\sigma_{P_i}\text{ when $i$ and $j$ are not connected by an edge of $\Gamma$}.
\end{align*}
Via the homomorphism $B_\Gamma \to \Aut(\mathcal C_\Gamma)$ defined by  $\sigma_i \mapsto \sigma_{P_i}$, we have a (weak) action of $B_\Gamma$ on $\mathcal C_\Gamma$.
This action is known to be faithful in some cases---for example when $\Gamma$ is an ADE Dynkin diagram~\cite{bra.tho:11,kho.sei:02}---and is conjecturally faithful for all $\Gamma$.
The induced action on $K_\Gamma$ is the standard (geometric) representation of the Coxeter group associated to $\Gamma$.

\subsection{Injectivity of the mass map}\label{sec:massinj}
We now fix the graph $\Gamma$ and let $\calC = \calC_\Gamma$ be the associated strongly 2-CY category.
We say that a stability condition \(\tau\) is \emph{standard} if its \([0,1)\) heart is the standard heart \(\heart_{\std}\).
It is easy to see that the set of standard stability conditions is connected.
Let \(\Stab^{\circ} \mathcal{C}_{\Gamma} \subset \Stab \mathcal{C}_{\Gamma}\) be the connected component containing this set.

Let \(\mathbf{S}\) to be the set of spherical objects \(x\) of \(\mathcal{C}_{\Gamma}\) such that \(x\) is stable with respect to some stability condition in \(\Stab^{\circ} \mathcal{C}_{\Gamma}\).
The goal of~\Cref{sec:massinj} is to show that the map $m_q \colon \Stab^{\circ}(\calC)/\mathbf{C} \to \mathbf{P}^{\bfS}$ is injective.
We need two lemmas.
\begin{lemma}\label{lem:phase-comparison-unequal}
  Let \(\tau\) be a standard stability condition on \(\mathcal{C}\).
  Let \(\tau'\) be any stability condition such that for all spherical objects \(x \in \bfS\), we have
  \[ m_{q,\tau}(x) = m_{q,\tau'}(x).\]
  Suppose that \(\phi_{\tau}(P_1) < \phi_{\tau}(P_2)\) and \(\phi_{\tau'}(P_1) = \phi_{\tau}(P_1)\).
  Then \(\phi_{\tau'}(P_2) = \phi_{\tau}(P_2)\).
\end{lemma}
\begin{proof}
  We introduce some notation.
  Given an edge \(ij\) in \(\Gamma\), let \(P_{ij}\) denote the cone of a non-zero morphism \(P_i[-1] \to P_j\).
  Note that \(P_{ij}\) is spherical and fits in an exact sequence
  \[ 0 \to P_i \to P_{ij} \to P_j \to 0 \]
  in the standard heart.

  Since we have \(\phi_{\tau}(P_1) < \phi_{\tau}(P_2)\), the argument of the central charge of \(P_{21}\) lies strictly between these two;  that is
  \[\phi_{\tau}(P_1) < \arg(Z_{\tau}(P_{21})) < \phi_{\tau}(P_2).\]
  Since the only proper sub-object of \(P_{21}\) in the standard heart is \(P_1\), we observe that \(P_{21}\) is \(\tau\)-stable.
  The previous inequality gives us
  \begin{equation}\label{eq:phase-ineq1}
    \phi_\tau(P_1) < \phi_\tau(P_{21}) < \phi_\tau(P_1) < \phi_\tau(P_1) + 1.
  \end{equation}
  Recall from~\Cref{cor:mass-determines-spherical-stables} that \(\tau\) and \(\tau'\) have the same stable spherical objects.
  In particular, $P_{21}$ is also $\tau'$-stable.
  Since we know that
  \[\Hom(P_1, P_{21}) \neq 0, \quad \Hom(P_{21}, P_2) \neq 0, \quad \text{and } \Hom(P_2, P_1[1])\neq 0,\]
  the analogous inequalities hold for $\phi_{\tau'}$:
  \begin{equation}\label{eq:phase-ineq2}
    \phi_\tau'(P_i) < \phi_\tau'(P_{ji}) < \phi_\tau'(P_j) < \phi_\tau'(P_i)+ 1.
  \end{equation}
  We also know that
  \[m_{q,\tau'}(P_1) = m_{q,\tau}(P_1), \quad m_{q,\tau'}(P_{21}) = m_{q,\tau}(P_{21}), \quad m_{q,\tau'}(P_2) = m_{q,\tau}(P_2).\]
  By~\cite[Lemma~5.2]{bap.bec.lic:23}, it follows that the pair of complex numbers \((Z_{\tau}(P_1), Z_{\tau}(P_2))\) is equal to the pair \((Z_{\tau'}(P_{1}), Z_{\tau'}(P_2))\) possibly after a rotation and a reflection.
  The inequalities~\eqref{eq:phase-ineq1} and~\eqref{eq:phase-ineq2} imply that no reflection is necessary, and the equality \(\phi_{\tau}(P_{1}) = \phi_{\tau'}(P_1)\) implies that no rotation is necessary.
  We conclude that \(Z_{\tau}(P_2) = Z_{\tau'}(P_{2})\); call this number \(z\).
  Since both \(\phi = \phi_{\tau}(P_2)\) and \(\phi = \phi_{\tau'}(P_2)\) have the property that \(e^{i \pi \phi}\) lies on the same real ray as \(z\), and both lie in the interval \([\phi_{\tau}(P_1), \phi_{\tau}(P_1)+1)\), we conclude that they are equal.
\end{proof}
\begin{lemma}\label{lem:phase-comparison-equal}
  Let \(\tau\) be a standard stability condition on \(\mathcal C\).
  Let \(\tau'\) be any stability condition such that for all \(x \in \bfS\), we have
  \[ m_{q,\tau}(x) = m_{q,\tau'}(x).\]
  Moreover, suppose \(\phi_\tau(P_1) = \phi_\tau(P_2)\).
  Then \(\phi_{\tau'}(P_1) = \phi_{\tau'}(P_2)\).
  In particular, if \(\phi_{\tau'}(P_1) = \phi_{\tau}(P_1)\), then we also have \(\phi_{\tau'}(P_2) = \phi_{\tau}(P_2)\).
\end{lemma}
\begin{proof}
  Recall from~\Cref{cor:mass-determines-spherical-stables} that \(\tau\) and \(\tau'\) have the same stable spherical objects.
  In particular \(P_1\) and \(P_2\) are also \(\tau'\)-stable.

  Let \(x = P_{12}\).
  We claim that the \(\tau'\)-stable factors in the \(\tau'\)-HN filtration of \(x\) are precisely \(P_1\) and \(P_2\).
  To see this, consider a filtration
  \begin{equation}\label{eqn:tauprimex}
    0 = x_0 \to x_1 \to \cdots \to x_n = x
  \end{equation}
  whose factors \(a_{i}\) are \(\tau'\)-stable and appear in the order of non-increasing phase.
  By the Mukai lemma~\cite[Lemma 2.2]{huy:12} the \(a_{i}\) are spherical and by~\Cref{cor:mass-determines-spherical-stables} also \(\tau\)-stable.
  We have
  \[m_{q,\tau'}(x) = \sum_im_{q,\tau'}(a_i),\]
  and hence
  \[m_{q,\tau}(x) = \sum_im_{q,\tau}(a_i).\]
    But \(x\) is \(\tau\)-semistable.
    Therefore by~\Cref{prop:ss-geodesic}, we conclude that the \(\tau\)-stable objects \(a_i\) all have \(\tau\)-phase \(\phi = \phi_{\tau}(x)\).
    Since the \(a_i\) are \(\tau\)-stable, they are simple objects of the category \(\mathcal C_{\tau,\phi}\).
    Thus,~\eqref{eqn:tauprimex} is a Jordan--H\"older filtration of \(x\) in \(\mathcal C_{\tau,\phi}\).
    But we know that \(x\) has a unique Jordan--H\"older filtration in \(\mathcal C_{\tau,\phi}\), namely
    \begin{equation}\label{eq:xjh}
      \begin{tikzcd}
        0 \ar{rr}&& P_2 \ar{rr}\ar{dl}&& x. \ar{dl} \\\
        &P_2\ar[dashed]{ul}&& P_1\ar[dashed]{ul}
      \end{tikzcd}
    \end{equation}
    Therefore, the filtration~\eqref{eqn:tauprimex} must coincide with~\eqref{eq:xjh}.
    That is, \(P_1\) and \(P_2\) are the \(\tau'\)-stable HN factors of \(x\), with
    \[\phi_{\tau'}(P_2) \geq \phi_{\tau'}(P_1).\]

    By the same argument, the \(\tau'\)-stable factors of \(y = P_{21}\) are also \(P_1\) and \(P_2\), with
    \[\phi_{\tau'}(P_1) \geq \phi_{\tau'}(P_2).\]
    It follows that \(\phi_{\tau'}(P_1) = \phi_{\tau'}(P_2)\).
\end{proof}

\begin{proposition}[Injectivity]\label{prop:injectivity-general}
  Let $\calC$ be the 2-CY category associated to a finite connected quiver \(\Gamma\).
  Let \(q > 0\) be a real number.
  Then the mass map $m_{q} \from \Stab^{\circ}(\mathcal C)/\mathbf{C} \to \mathbf{P}^\bfS$ is injective.
\end{proposition}
\begin{proof}
  Let $\tau$ be a stability condition in \(\Stab^{\circ}(\mathcal{C})\).
  By~\cite[Proposition 4.13]{ike:14}, \(\tau\) is in the braid group orbit of a standard stability condition up to the action of \(\C\).
  (Although~\cite{ike:14} treats the case of preprojective algebras, the same proof works in our setting.)    
  By applying a braid and an element of \(\C\), suppose that $\tau$ is standard.
 Since the objects \(P_i\) are simple objects of the standard heart, they are \(\tau\)-stable.
  
 Let \(\tau'\) be another stability condition such that \(m_q(\tau) = m_q(\tau') \in \mathbf{P}^{\bfS}\).
 By scaling the central charge of \(\tau'\), we may assume that \(m_{q,\tau} = m_{q,\tau'}\) on \(\bfS\). 
  By~\Cref{cor:mass-determines-spherical-stables}, $\tau$ and $\tau'$ have the same spherical stable objects.

Label the vertices of \(\Gamma\) by \(\{1, \dots, n\}\) so that for each $i > 1$ there is some $j < i$ such that the sub-quiver $\{i,j\}$ is of type $A_2$.
  By translating the slicing of \(\tau'\) if necessary, assume that \(\phi_{\tau'}(P_1) = \phi_{\tau}(P_1)\).
  Rescale the central charge again to ensure that \(m_{q,\tau} = m_{q,\tau'}\) continues to hold on \(\bfS\).
  Then \(Z_{\tau}(P_1) = Z_{\tau'}(P_1)\).

  We now prove by induction on \(i\) that \(\phi_{\tau}(P_i) = \phi_{\tau'}(P_i)\) and \(Z_{\tau}(P_i) = Z_{\tau'}(P_i)\).
  The base case \(i = 1\) holds by construction.
  For the induction step, let $i > 1$.
  Consider $j < i$ such that the sub-quiver $\{i,j\}$ is of type $A_2$.
  By the induction hypothesis, we have $\phi_{\tau'}(P_j) = \phi_\tau(P_j)$ and $Z_{\tau'}(P_j) = Z_\tau(P_j)$.
  Depending on whether or not the quantities \(\phi_{\tau}(P_i)\) and \(\phi_{\tau}(P_j)\) are equal, we are now in the setting of one of~\Cref{lem:phase-comparison-unequal,lem:phase-comparison-equal}.
In both cases, we conclude that \(\phi_{\tau'}(P_i) = \phi_{\tau}(P_i)\).
  Since we already have \(m_{q,\tau}(P_i) = m_{q,\tau'}(P_i)\), we conclude that \(Z_{\tau}(P_i) = Z_{\tau'}(P_{i})\).

  Since \(\phi_{\tau}(P_i) = \phi_{\tau'}(P_i)\) and \(0 \leq \phi_{\tau}(P_i) < 1\), the same holds for \(\phi_{\tau'}(P_i)\).
  In particular, \(P_i\) lies in the \([0,1)\) heart of \(\tau'\).
  As a result, \(\heart_{\std}\) is contained in the \([0,1)\)-heart of \(\tau'\).
  Since both are hearts of bounded \(t\)-structures, they must be equal.
  Furthermore, since the \(P_i\) span the Grothendieck group and \(Z_{\tau}(P_i) = Z_{\tau'}(P_i)\), we have \(Z_{\tau} = Z_{\tau'}\).
  So \(\tau'\) is a stability condition with the same heart and central charge as \(\tau\).
  We conclude that \(\tau' = \tau\).
\end{proof}

\subsection{Homeomorphism onto the image}
Having proved injectivity, we take up the question of \(m_q\) being a homeomorphism onto its image.

As before, let $\calC$ be the 2-CY category associated to a finite connected quiver.
We begin by explicitly identifying convenient open neighbourhoods of points in $\Stab(\calC)/\mathbf{C}$ whose closure is compact.

Let \(W \subset \Stab(\mathcal{C})\) be the set of standard stability conditions.
It is easy to see that \(W \subset \Stab(\mathcal{C})\) is locally closed.
Let \(H \subset \mathbf{C}\) be the semi-closed upper half plane
\[ H = \{z \in \mathbf{C} \mid z = r e^{i\pi\phi} \text{ for } r \in \mathbf{R}_{> 0} \text{ and } \phi \in [0,1)\}.\]
We have a homeomorphism
\begin{equation}\label{eqn:stdhomeo}
  H^{n} \to W
\end{equation}
that sends \((z_1,\dots, z_{n})\) to the stability condition whose \([0,1)\)-heart is the standard heart and whose central charge is defined by \([P_i] \mapsto z_i\).

Fix a small \(\epsilon > 0\).
Let \(W_{\epsilon} \subset W\) be the open subset of stability conditions \(\tau\) that satisfy
\[ \epsilon < m_{q,\tau}(P_i) < \epsilon^{-1}.\]
Likewise, let \(H_{\epsilon} \subset H\) be the open subset containing \(z = re^{i\pi\phi}\) such that
\[ \epsilon < rq^{\phi} < \epsilon^{-1}.\]
Then~\eqref{eqn:stdhomeo} restricts to a homeomorphism
\[ H_{\epsilon}^{n} \to W_{\epsilon}.\]
Let \(\overline H_{\epsilon} \subset \mathbf{C}\) be the closure of \(H_{\epsilon}\).
Then \(\overline H_{\epsilon}\) is closed and bounded, and hence compact.
Let \(\overline W_{\epsilon} \subset \Stab(\mathcal{C})\) be the closure of \(W_{\epsilon}\).
\begin{proposition}\label{prop:wcompact}
  The map
  \(H_{\epsilon}^n \to W_{\epsilon}\)
  extends to a homeomorphism
  \(\overline H_{\epsilon}^n \to \overline W_{\epsilon}\).
  As a result, \(\overline W_{\epsilon}\) is compact.
\end{proposition}
\begin{proof}
  We first extend the map \(H_{\epsilon}^n \to \Stab(\mathcal{C})\) to \(\overline H_{\epsilon}^n\).
  The only points of \(\overline H_{\epsilon}\) not in \(H_{\epsilon}\) are the complex numbers of argument 1.
  Given \((z_1, \dots, z_{n}) \in \overline H_{\epsilon}^n\), let \(S \subset \{1,\dots, n\}\) be the indices such that the argument of \(z_i\) is \(1\).
  Let \(\mathcal{A} \subset \mathcal{C}\) be the extension closure of \(P_{i}\) for \(i \not \in S\) and \(P_{i}[-1]\) for \(i \in S\).
  This is a tilt of the standard heart, and hence also the heart of a bounded \(t\)-structure.
  We send \((z_1,\dots,z_{n})\) to the stability condition whose \([0,1)\)-heart is \(\mathcal{A}\) and whose central charge is defined by \([P_{i}] \mapsto z_{i}\).
  It is easy to check that the resulting map is continuous, and defines the required homeomorphism.
\end{proof}

Let \(U \subset \Stab(\mathcal{C})/\mathbf{C}\) be the image of \(W\) under the quotient map \(\Stab(\mathcal{C}) \to \Stab(\mathcal{C})/\mathbf{C}\).
This is the set of stability conditions (up to translation) in which the objects \(P_{i}\) are stable and their phases lie in an open interval of length \(1\).
It is easy to check that \(U \subset \Stab(\mathcal{C})/\mathbf{C}\) is an open subset.
Let \(U_{\epsilon} \subset U\) be the image of \(W_{\epsilon}\).\
\begin{proposition}\label{prop:ucompact}
  The open set \(U_{\epsilon} \subset \Stab(\mathcal{C})/\mathbf{C}\) has a compact closure.
\end{proposition}
\begin{proof}
  The closure \(\overline U_{\epsilon}\) is the image of \(\overline W_{\epsilon}\), which is compact by~\Cref{prop:wcompact}.
\end{proof}

The following proposition will be used when we show that the mass map is a homeomorphism onto its image in type $A_2$ and $\aonehat$.
First let us establish some notation.
\begin{definition}\label{def:collapsed-triangle}
  Fix some \(q > 0\).
  Let \(\tau\) be a stability condition.
  We say that \(\tau\) \emph{collapses} a distinguished triangle \(x \to y \to z \xrightarrow{+1}\) if
    \[m_{q,\tau}(y) = m_{q,\tau}(x) + m_{q,\tau}(z).\]
\end{definition}
If \(\tau\) does not collapse \(x \to y \to z \xrightarrow{+1}\), then we have a strict triangle inequality
    \[m_{q,\tau}(y) < m_{q,\tau}(x) + m_{q,\tau}(z).\]

\begin{proposition}\label{prop:conditional-local-homeo}
  Let \(\tau\) be a standard stability condition.
  Suppose there exists a finite set \(T\) of distinguished triangles of objects in \(\bfS\) with the following properties:
  \begin{enumerate}
  \item no triangle in \(T\) is collapsed by \(\tau\);
  \item if \(\tau'\) is any stability condition such that no triangle in \(T\) is collapsed by \(\tau'\), then \(\tau'\) is standard up to the action of \(\mathbf{C}\).
  \end{enumerate}
  Then
  \[m_q \colon \Stab(\mathcal{C})/\mathbf{C} \to \mathbf{P}^{\bfS}\]
  is a local homeomorphism onto its image at \(\tau\).
\end{proposition}
\begin{proof}
  We abbreviate \(m_{q}\) by \(m\).

  Since \(\tau\) is standard, it represents a point of \(U \subset \Stab(\mathcal{C})/\mathbf{C}\), which we also denote by \(\tau\).
  For a small enough \(\epsilon\), we have \(\tau \in U_{\epsilon}\).
  The set \(\overline U_{\epsilon}\) is compact by~\Cref{prop:ucompact} and \(m\) is injective by~\Cref{prop:injectivity-general}.
  Therefore the map
  \begin{equation}\label{eq:homeocompact}
    m \colon \overline U_{\epsilon} \to m(\overline U_{\epsilon})
  \end{equation}
  is a continuous bijection between two compact sets and hence a homeomorphism.

  Let \(V\) be the subset of the image of \(m\) consisting of points \([f]\) that satisfy the inequalities
  \[ f(y) < f(x) + f(z)\]
  for each non-collapsed triangle \(x \to y \to z \xrightarrow{+1}\) asserted by the hypotheses.
  Then \(V\) is an open subset of the image of \(m\) containing \(m(\tau)\).
  Furthermore, by the hypotheses, we have
  \[ m^{-1}(V) \subset U.\]
  Let \(V_{\epsilon} \subset V\) be the open subset defined by the conditions
  \begin{enumerate}
  \item \(f(P_1) \neq 0\), and
  \item for all \(i\), we have
      \[ \epsilon < f(P_{i})/f(P_{1}) < \epsilon^{-1}. \]
  \end{enumerate}
  Then it is easy to check that \(m^{-1} (V_{\epsilon}) \subset U_{\epsilon}\).
  If \(\epsilon\) is small enough, then \(m(\tau) \in V_{\epsilon}\) and hence \(\tau \in m^{-1}(V_{\epsilon})\).
  The map
  \[ m \colon m^{-1}(V_{\epsilon}) \to V_{\epsilon}\]
  is a restriction of the homeomorphism~\eqref{eq:homeocompact}, and hence a homeomorphism.
  Thus, \(m\) maps an open neighbourhood of \(\tau\) homeomorphically to an open neighbourhood of \(m(\tau)\), and hence is a local homeomorphism at \(\tau\).
\end{proof}

\begin{remark}[Existence of triangles]
  We prove the existence of the triangles \(T\) required in~\Cref{prop:conditional-local-homeo} in the \(A_2\) and \(\aonehat\) cases, but we them more generally.
\end{remark}

\begin{remark}[Functionals in the closure]
  It follows from~\Cref{cor:limit-mass} that the \(\homBar\) functionals are in the closure of the mass map.
  Whether these functionals are dense in the boundary depends on \(q\).
  In the next two sections we answer this question for \(q = 1\) in the \(A_2\) and \(\aonehat\) cases.
  \cite{bap.bec.lic:23} answers this question for other values of \(q\) in the \(A_2\) case.
\end{remark}

\section{The \texorpdfstring{$A_2$}{A2} case}\label{sec:a2}
The aim of this section is to understand the Thurston compactification at \(q = 1\) of the stability space for the 2-Calabi--Yau category $\calC_\Gamma$ where $\Gamma$ is the $A_2$ quiver.
\subsection{Standard stability conditions}
Let $\mathcal C$ be the 2-Calabi--Yau category associated to the $A_2$ quiver, as defined in~\Cref{sec:background}.
It is a graded, $\k$-linear triangulated category classically generated by two spherical objects $P_1$ and $P_2$.
The standard heart \(\heart_{\std}\) is the extension closure of \(P_1\) and \(P_2\).
It has two simple objects, $P_1$ and $P_2$, and two additional indecomposable objects, denoted by $P_{12}$ and $P_{21}$.
The object $P_{ij}$ is the unique extension of $P_i$ by $P_j$.

Recall that \(\tau \in \Stab(\mathcal{C})\) is standard if its \([0,1)\) heart is \(\heart_{\std}\) and a point in \(\Stab(\mathcal{C})/\mathbf{C}\) is standard if it is the image of a standard \(\tau\).

We divide the set of standard stability conditions in two subsets as follows.
Recall that for a standard $\tau$, both $P_1$ and $P_2$ must be \(\tau\)-stable.
We say that \(\tau\) is of \emph{type I} if \(\phi(P_1) \leq \phi(P_2)\) and of \emph{type II} if \(\phi(P_2) \leq \phi(P_1)\).
If \(\phi(P_1) = \phi(P_2)\), then \(\tau\) is of both types, and we say that it is \emph{on-the-wall}; otherwise, it is \emph{off-the-wall}.
See~\Cref{fig:stdstab} for a sketch of the central charges of the stability conditions of the two types.

The two types are distinguished by which of \(P_{12}\) or \(P_{21}\) is semi-stable.
In type I, \(P_{21}\) is semi-stable; in type II, \(P_{12}\) is semi-stable; on the wall, both are semi-stable.

\begin{figure}[ht]
  \centering
  \begin{tikzpicture}[thick]
    \tikzstyle{every node}=[font=\footnotesize]
    \begin{scope}[xshift=-3cm]
    \draw (0:0) edge [->] (60:2) (60:2) node [right] {$P_1$};
    \draw (0:0) edge [->] (120:2) (120:2) node [left] {$P_2$};
    \draw (0:0) edge [->] (90:2.5) (90:2.5) node [above] {$P_{21}$};
    \node (0:0) {$\bullet$};
    \draw[decorate, decoration={brace, amplitude=10}] (-2,-0.25) -- node[sloped, yshift=0.75cm] {standard} (-2,3);
    \draw[decorate, decoration={brace, amplitude=10}] (3,-1.25) -- node[sloped, yshift=-0.75cm] {type I} (-2,-1.25);
  \end{scope}
  \begin{scope}
    \draw (0:0) edge [->] (90:1) (90:1) node [right] {$P_1$};
    \draw (0:0) edge [->] (90:1.5) (90:1.5) node [left] {$P_2$};
    \draw (0:0) edge [->] (90:2.5) (90:2.5) node [above] {$P_{21}$} (90:3.2) node {$P_{12}$};
    \node (0:0) {$\bullet$};
  \end{scope}
  \begin{scope}[xshift=3cm]
    \draw (0:0) edge [->] (60:2) (60:2) node [right] {$P_2$};
    \draw (0:0) edge [->] (120:2) (120:2) node [left] {$P_1$};
    \draw (0:0) edge [->] (90:2.5) (90:2.5) node [above] {$P_{12}$};
    \node (0:0) {$\bullet$};
    \draw[decorate, decoration={brace, amplitude=10}] (2.5,-1.25) -- node[sloped, yshift=-0.75cm] {type II} (-3.25,-1.25);
  \end{scope}
\end{tikzpicture}
  \caption{Central charges in a standard stability condition of types I and II.  The intersection of the two types are stability conditions on-the-wall.}\label{fig:stdstab}
\end{figure}

The subset of \(\Stab(\mathcal{C})/\mathbf{C}\) represented by standard stability conditions of type I is locally closed.
Let \(\Lambda\) be the closure of this set.
Note that \(\Lambda\) includes some non-standard stability conditions, namely those with \(\phi(P_2) = \phi(P_1) + 1\).
The set $\Lambda$ tiles $\Stab(\mathcal C)/\mathbf{C}$ under the action of $B_3$.
That is, $\Lambda$ satisfies the following properties (see, e.g.,~\cite[Proposition 4.2]{bri.qiu.sut:20}).
\begin{enumerate}
\item Each point of $\Stab(\mathcal C)/\mathbf{C}$ lies in the $B_3$-orbit of a point of $\Lambda$.
\item The stabiliser of $\Lambda$ is the subgroup generated by $\gamma = \sigma_2\sigma_1$ in $B_3$.
\item For any $g \in B_3$ not in the stabiliser, the interiors of $\Lambda$ and $g \Lambda$ have empty intersection.
\end{enumerate}
See~\Cref{fig:exchange-graph} for a picture of the tiling of $\Stab(\mathcal C)/\mathbf{C}$ by the orbit of $\Lambda$.
The interior of $\Lambda$ is the set of standard off-the-wall stability conditions of type I.

\subsection{The spherical objects}
Let \(\bfS\) be the set of spherical objects of \(\mathcal{C}\), up to isomorphism and shift.
It turns out that elements of \(\bfS\) can be naturally thought of as rational points on the circle \(\mathbf{P}^1(\mathbf{R})\).
We explain how.

Recall that the Artin--Tits braid group of the $A_2$ quiver is 
\[ B_3 = \langle  \sigma_1, \sigma_2 \mid  \sigma_1\sigma_2\sigma_1 = \sigma_2\sigma_1\sigma_2\rangle.\]
It acts on \(\calC\) via the homomorphism 
\[ \sigma_i \mapsto \sigma_{P_i}.\]
It turns out to act transitively on the set of all the spherical objects of $\mathcal C$, and hence on $\bfS$~\cite{bap.deo.lic:22}.
The element \((\sigma_2\sigma_1)^3\) generates the center of \(B_3\) and acts on the category by \(x \mapsto x[-2]\).
As a result, the action of \(B_3\) on \(\bfS\) factors through \(B_3/Z(B_3)\).
On the other hand, we have an isomorphism $B_3 / Z(B_3) \to \PSL_2(\Z)$ given by
\begin{align*}
\sigma_1 \mapsto
    \begin{pmatrix}
      1 & 1 \\
      0 & 1
    \end{pmatrix}, \quad
  \sigma_2 \mapsto
             \begin{pmatrix}
      1 & 0 \\
      -1 & 1
    \end{pmatrix}.
\end{align*}
The action of \(B_3\) on \(\mathcal{C}\) is faithful~\cite{kho.sei:02,rou.zim:03}.
It is easy to check that the stabiliser of \(P_1 \in \bfS\) is generated by \(\sigma_1\).
As a result, we have a $\PSL_2(\Z)$-equivariant bijection
\begin{equation}\label{eq:num-to-obj}
  i \from \bfS \to \mathbf{P}^1(\Z)
\end{equation}
defined uniquely by the choice
\[ P_1 \mapsto [1:0].\]

Let us describe the spherical object corresponding to a point \([a:c] \in \mathbf{P}^1(\mathbf{Z})\).
Assume \(c \neq 0\).
Write the rational number $a/c$ as a continued fraction with an odd number of terms:
\[
  \frac{a}{c} = n_0 + \cfrac{1}{n_2 + \cfrac{1}{\ddots + \cfrac{1}{n_{2k}}}}.
\]
Here, each $n_i$ is an integer, with $n_i > 0$ for $i = 1, \ldots, 2k$.
Then $[a:c]$ corresponds to the object of $\bfS$ given by
\begin{equation}\label{eq:cfwriting}
  \sigma_1^{n_0}\sigma_2^{-n_1}\cdots\sigma_1^{n_{2k}} (P_2).
\end{equation}
For example, we get
\[
  P_1 \mapsto [1:0], \quad P_2 \mapsto [0:1], \quad P_{21} \mapsto [1:-1], \text{ and } P_{12} \mapsto [1:1].
\]

\subsection{The automaton}\label{sec:a2-automaton}
Fix an off-the-wall standard stability condition $\tau$ of type I.
Let \(\Sigma = \{P_1, P_2, P_{21}\}\), the set of indecomposable \(\tau\)-stable objects of the heart.
Set \(X = P_{21}\).
Define the braid $\gamma \in B_3$ by
\[ \gamma = \sigma_2\sigma_1 = \sigma_X \sigma_2 = \sigma_1\sigma_X.\]
We now describe an HN automaton $\Theta$ that computes \(\tau\)-HN multiplicities of all spherical objects (see~\Cref{fig:automaton}).
\begin{figure}[ht]
  \centering
  \begin{tikzpicture}[node distance=2cm and 4cm, auto, ->, >=stealth', shorten >=1pt, semithick, font=\footnotesize]
    \tikzstyle{every state}=[fill=black!10!white, draw=black!20!white, rectangle, font=\normalsize]
    \draw (30:4) node[state] (A) {$[P_1, P_2]$};
    \draw (150:4) node[state] (B) {$[X, P_1]$};
    \draw (270:4) node[state] (C) {$[P_2, X]$};
    \path (A)  edge [loop above] node {$\sigma_1, \begin{pmatrix} 1 & 1 \\ 0 & 1\end{pmatrix}$} (A)
    edge node [anchor=center, fill=white, sloped] {$\sigma_X , \begin{pmatrix}1 & 0 \\ 1 & 1\end{pmatrix}$} (B)
    edge [bend left=20] node {$\gamma^{-1}$} (C)
    edge [bend left=20] node {$\gamma$} (B)
    (B) edge [loop above] node  {$\sigma_X , \begin{pmatrix} 1 & 1 \\ 0 & 1\end{pmatrix}$} (B)
    edge node [anchor=center, fill=white, sloped] {$\sigma_2 , \begin{pmatrix}1 & 0\\ 1 & 1 \end{pmatrix}$} (C)
    edge [bend left=20] node {$\gamma^{-1}$} (A)
    edge [bend left=20] node {$\gamma$} (C)
    (C) edge [loop below] node {$\sigma_2 , \begin{pmatrix}1 & 1 \\ 0 & 1 \end{pmatrix}$} (C)
    edge node [anchor=center, fill=white, sloped] {$\sigma_1 , \begin{pmatrix}1 & 0 \\ 1 & 1\end{pmatrix}$} (A)
    edge [bend left=20] node {$\gamma$} (A)
    edge [bend left=20] node {$\gamma^{-1}$} (B)
    ;
  \end{tikzpicture}
  \caption{An automaton describing the dynamics of Harder--Narasimhan filtrations in a stability condition with stable objects $P_1$, $P_2$, and $X = P_{21}$.}\label{fig:automaton}
\end{figure}

Formally (following~\Cref{def:hn-automaton}), the automaton $\Theta$ is defined by the $B_3$-labeled graph with three vertices and three edges from each vertex, as shown in~\Cref{fig:automaton}.
Note that the states are called $[X, P_1]$, $[P_1, P_2]$, and $[P_2, X]$.
The \(\Theta\)-set \(S\) is defined by setting \(S_{[a,b]}\) to be the set of spherical objects whose \(\tau\)-stable HN factors are shifts of \(a\) and \(b\).
For example, the object \(P_{112} = \sigma_1\sigma_1(P_2)\) has stable HN factors \(P_1[-1]\), \(P_1\), and \(P_2\), and hence is supported at the state \([P_1,P_2]\).
The objects \(P_{1}\), \(P_2\), and \(X\) are supported at two states.
Recall that the representation $M$ associates the rank 3 module \(\mathbf{Z}^{\Sigma}\) to each state.
However, by construction, at the state \([a,b]\), the image of \(\HN_{\tau}\colon S \to \mathbf{Z}^{\Sigma}\) lands in the rank 2 sub-module \(\mathbf{Z}^{\{a,b}\).
We replace \(\mathbf{Z}^{\Sigma}\) by this sub-module.
Then the maps \(M(e)\) are given by \(2 \times 2\) matrices in the standard basis, as shown in~\Cref{fig:automaton}.
The arrows labelled $\gamma^{\pm 1}$ correspond to the identity matrix, which we omit in the figure.

To check that this setup defines an HN automaton, we need to check that an arrow takes objects supported at its source to objects supported at its target.
We also need to check that the HN multiplicities transform according to the matrix associated to the arrow, as in~\Cref{def:hn-automaton}.
We carry out both these checks in the next proposition.

\begin{proposition}\label{prop:automaton}
  Let \(v = [a,b]\) and $w = [c,d]$ be two states of \(\Theta\).
  Let $e \from v \to w$ be an arrow labelled by the element \(g \in B_3\).
  Let $x \in \mathcal C$ be any object whose stable \(\tau\)-HN factors are shifts of $a$ and $b$.
  The following hold.
  \begin{enumerate}
  \item The stable \(\tau\)-HN factors of $g \cdot x$ are shifts of $c$ and $d$.
  \item The following diagram commutes.
    \begin{center}
      \begin{tikzcd}
        x \arrow[mapsto]{r}{HN_{\tau}} \arrow[mapsto]{d}{g} & HN_{\tau}(x) \in \mathbf{Z}^{\{a,b\}} \arrow[mapsto, shift right=10]{d}{M(e)}\\
        gx \arrow[mapsto]{r}{HN_{\tau}} & HN_{\tau}(gx) \in \mathbf{Z}^{\{c,d\}}
      \end{tikzcd}
    \end{center}
    In other words,
    \[ HN_\tau(gx) = M(e) \cdot HN_\tau(x).\]
  \end{enumerate}
\end{proposition}
\begin{proof}
  Let
  \[0 \to x_0 \to x_1 \to \dots \to x_n = x\]
  be a filtration of $x$ with whose factors $z_i = \Cone(x_{i-1} \to x_i)$ are stable and appear in non-increasing order of phase.
  Then, by hypothesis each \(z_i\) is a shift of \(a\) or \(b\).
  By applying $g$, we get a filtration
  \begin{equation}\label{eqn:newfiltration}
    0 \to gx_0 \to gx_1 \to \dots \to gx_n = gx
  \end{equation}
  with factors $gz_i$.

  We check that the filtration~\eqref{eqn:newfiltration} is rectifiable in the sense of~\Cref{def:rectifiable-filtration}.
  We use~\Cref{prop:rectifiable-filtration-check}.

  We outline the case $v = [P_1,P_2]$, leaving the others to the reader.
  Denote by \(\lceil x \rceil\) (resp.\ \(\lfloor x \rfloor\)) the HN factor of \(x\) of highest (resp.\ lowest) phase.
  To apply~\Cref{prop:rectifiable-filtration-check} to the filtration~\eqref{eqn:newfiltration}, we must check the following:
  for $i < j$, either
  \begin{enumerate}
  \item $\Hom(gz_j, gz_i[1]) = \Hom(gz_i, gz_j[1]) = 0$, or 
  \item $\lfloor  gz_i \rfloor \geq \lceil gz_j \rceil$.
  \end{enumerate}
  If \(z_i \cong z_j\), then (1) holds.
  Let us assume otherwise, so that \(\phi(z_i) > \phi(z_j)\).
  Since we are in a 2-CY category, \(\Hom(gz_j, gz_i[1])\) vanishes if and only if \(\Hom(gz_i, gz_j[1])\) does.
  We enumerate the pairs $(z_i,z_j)$ such that
  \begin{enumerate}
  \item \(z_i\) and \(z_j\) are shifts of \(P_1\) or \(P_2\),
  \item \(\phi(z_i) > \phi(z_{j})\), and
  \item \(\Hom(z_j, z_i[1]) \neq 0\).
  \end{enumerate}
  Up to a simultaneous shift, the only such pairs are $(P_1[1],P_1)$, $(P_2[1], P_2)$, and $(P_2, P_1)$.
  The outgoing edges \(e\) from \(v\) are labeled \(\sigma_1\),\(\sigma_2\), \(\gamma\), or \(\gamma^{-1}\).
  It is a simple check that for each of these edges and each of the enumerated pair, condition (2) holds.
  
  Since the filtration~\eqref{eqn:newfiltration} is rectifiable, it follows from~\Cref{prop:rect-filt-hn} that
  \[ \HN_{\tau}(gx) = \sum \HN_{\tau}(gz_{i}).\]
  Therefore, it suffices to prove the proposition for \(x = z_i\), and hence, for \(x = a\) and \(x = b\).
  This is another straightforward calculation.
\end{proof}

The automaton recognises every spherical object in the sense of~\Cref{rem:automaton-recognising}.
\begin{proposition}\label{prop:cyclicwriting}
The following hold.
  \begin{enumerate}
  \item Every \(\beta \in B_3\) has an expression of the form
    \[ \beta = \gamma^n \sigma_{a_1}^{m_1} \sigma_{a_2}^{m_2}\cdots \sigma_{a_k}^{m_k},\]
    where $n$ is an integer, $k$ is a non-negative integer, the $m_i$ are positive integers, and the sequence
    $(a_1,a_2 \dots, a_k)$ is a contiguous subsequence of the sequence \[(\ldots, X,1,2, X, 1, 2, \ldots).\]
  \item Every spherical object \(s \in \mathcal{C}\) has an expression
    \[ s = \beta_{1} \cdots \beta_n P_1\]
    that is recognised by \(\Theta\).
  \end{enumerate}
\end{proposition}
\begin{proof}
  For (1), we repeatedly use the commutation relations
  \[
    \gamma \sigma_2 \gamma^{-1} = \sigma_1, \quad     \gamma \sigma_X \gamma^{-1} = \sigma_2, \quad \gamma \sigma_1 \gamma^{-1} = \sigma_X.
  \]
  Begin by writing $\beta$ as any product of the generators $\sigma_1$ and $\sigma_2$, along with their inverses.
  Eliminate the inverses of the generators by rewriting as follows:
  \[
    \sigma_1^{-1} = \sigma_X\gamma^{-1}, \quad \sigma_2^{-1} = \sigma_1\gamma^{-1}.
  \]
  Next, use the commutation relations to rewrite
  \begin{equation}\label{eq:gammacomm}
    \gamma^i\sigma_X = \sigma_2\gamma^i,\quad \gamma^i\sigma_1 = \sigma_X\gamma^i, \quad \gamma^i\sigma_2 = \sigma_1\gamma^i,
  \end{equation}
  and thus move all powers of to the left.
  The rest of $\beta$ is now a product of elements from $\{\sigma_1,\sigma_2,\sigma_X\}$.  Replace any occurrences of $\sigma_2\sigma_1$, $\sigma_1\sigma_X$, or $\sigma_X\sigma_2$ by $\gamma$ and again move $\gamma$ to the left as before.
  It is easy to see that eventually, $\beta$ reaches the desired form.

  For (2), begin by writing $s = \beta P_1$ for some \(\beta \in B_3\).
  Write
  \[ \beta = \gamma^n \sigma_{a_1}^{m_1} \sigma_{a_2}^{m_2}\cdots \sigma_{a_k}^{m_k}\]
  as in (1).
  Note that \(P_1\) is supported at two states, namely \([P_1,P_2]\) and \([X,P_1]\).
  Each of the letters \(\{\sigma_1,\sigma_2,\sigma_X\}\) is applicable at least at one of the two states above, so we can apply \(\sigma_{a_k}\).
  The condition on our cyclic writing guarantees that we can apply each subsequent letter.  
  Therefore, \(\Theta\) recognises the expression
  \[ s = \gamma^n \sigma_{a_1}^{m_1} \sigma_{a_2}^{m_2}\cdots \sigma_{a_k}^{m_k} P_1.\]
\end{proof}
As a consequence of~\Cref{prop:cyclicwriting}, every spherical object is supported on at least one of the three states.
In particular, its HN factors involve (up to shift) only 2 out of the 3 objects \(P_1\), \(P_2\), and \(X\).

\begin{proposition}\label{prop:unsink}
  Let \(s\) be a spherical object supported at the state \([P_1,P_2]\).
  Assume that \(s\) is not a shift of \(P_1\) or \(P_2\).
  Then
  \begin{enumerate}
  \item we can write \(s = \sigma_1^{r} t\) for some \(r > 0\) and an object \(t\) supported at \([P_2,X]\);
  \item we can write \(s = \sigma_2^{-r}t\) for some \(r > 0\) and an object \(t\) supported at \([X,P_{1}]\).
  \end{enumerate}
  Analogous statements hold for the other two states.
\end{proposition}
\begin{proof}
  Write a recognised expression
  \(s = \gamma^n \sigma_{a_1}^{m_1} \sigma_{a_2}^{m_2}\cdots \sigma_{a_k}^{m_k} P_1\) as in the proof of~\Cref{prop:cyclicwriting}.
  Since \(s\) is not a shift of \(P_1,P_2\), or \(X\), we have \(k \geq 1\).
  Using the commutation relations~\eqref{eq:gammacomm}, move \(\gamma\) to the right to obtain another expression \(s = \sigma_{b_1}^{m_1} \sigma_{b_2}^{m_2}\cdots \sigma_{b_k}^{m_k} \gamma^nP_1 \), which is also recognised by \(\Theta\).
  Observe the state at which the recognising path ends is dictated by \(b_{1}\).
  For \(b_1 = 1\), the end state is \([P_1,P_2]\); for \(b_1 = 2\), it is \([P_2,X]\); and for \(b_1 = X\), it is \([X,P_1]\).
  By assumption, we must have \(b_1 = 1\).

  Taking \(r = m_1\) and \(t = \sigma_{b_2}^{m_2}\cdots \sigma_{b_k}^{m_k} \gamma^nP_1\) yields the first assertion.
  For the second assertion, consider
  \[\sigma_2s = \sigma_2\sigma_{b_1}^{m_1} \cdots \sigma_{b_k}^{m_k}\gamma^n P_1.\]
  Write \(\sigma_2\sigma_{b_1} = \sigma_2\sigma_1\) as \(\gamma\) and move \(\gamma\) to the right using the commutation relations~\eqref{eq:gammacomm}, to obtain a recognised expression for \(\sigma_2s\).
  Its leftmost letter is \(\gamma\), \(\sigma_{X}\), or \(\sigma_1\).
  In the first two cases, we stop; in the last case we consider \(\sigma_2^2s\).
  After finitely many steps, we get \(t = \sigma_2^rs\) supported at \([X,P_1]\) with \(r > 0\).
  \end{proof}

\begin{remark}\label{rem:new-degegenerate-tau-automaton}
  Let $\tau$ be an on-the-wall standard stability condition.
  \Cref{fig:degenerate-tau-automaton} describes a \(\tau\)-HN automaton that recognises all spherical objects.
  Since the automaton in~\Cref{fig:automaton} suffices for mass computations, we omit the details.
\end{remark}

\begin{remark}
  For every off-the-wall stability condition, we have an HN automaton obtained by applying an appropriate braid to the automaton in~\eqref{fig:automaton}.
  Likewise, for a on-the-wall stability condition, we have an HN automaton obtained by applying a braid to the automaton in~\eqref{fig:degenerate-tau-automaton}.
  Thus, we have an automaton at every point of the stability space.
  Observe that these automata exhibit a wall-and-chamber behaviour: they remain constant in a chamber, get more complicated along a wall, and flip to a different automaton as we cross the wall.  It would be interesting to understand the wall-crossing formulas for recognized expressions of a fixed braid in different chambers.
\end{remark}

\begin{figure}
  \centering
  \begin{tikzpicture}[node distance=2cm and 4cm, auto, ->, >=stealth', shorten >=1pt, semithick, font=\footnotesize]
    \tikzstyle{every state}=[fill=black!10!white, draw=black!20!white, rectangle, font=\normalsize]
    \tikzstyle{label}=[fill=white, anchor=center]
    \draw (45:4) node[state] (A) {$[X', P_2]$};
    \draw (135:4) node[state] (B) {$[P_1, X']$};
    \draw (225:4) node[state] (C) {$[X, P_1]$};
    \draw (315:4) node[state] (D) {$[P_2, X]$};
    \path
    (A) edge[bend left=10] node[label]{$\sigma_1$} (B)
    (A) edge[bend right=10] node[label]{$\gamma'$} (B)
    (B) edge[loop left] node {$\sigma_1$} ()
    (B) edge[bend left=10] node[label] {$\sigma_2$} (C)
    (B) edge[bend right=10] node[label] {$\gamma$} (C)
    (C) edge[bend left=10] node[label] {$\sigma_2$} (D)
    (C) edge[bend right=10] node[label] {$\gamma$} (D)
    (D) edge[bend left=10] node[label] {$\sigma_1$} (A)
    (D) edge[bend right=10] node[label] {$\gamma'$} (A)
    (D) edge[loop right] node {$\sigma_2$} ()
    (A) edge[bend left=10] node[pos=0.2,label,sloped] {$\gamma$, $w_0^\pm$} (C)
    (C) edge[bend left=10] node[pos=0.2,label,sloped] {$\gamma'$, $w_0^{\pm}$} (A)
    (B) edge[bend left=10] node[pos=0.2,label,sloped] {$\gamma'$, $w_0^\pm$} (D)
    (D) edge[bend left=10] node[pos=0.2,label,sloped] {$\gamma$, $w_0^\pm$} (B)
    ;
  \end{tikzpicture}
  \caption{An automaton describing the dynamics of HN filtrations in an on-the-wall stability condition with semistable objects $P_1$, $P_2$, $X = P_{21}$, and $X' = P_{12}$.  Here \(\gamma = \sigma_2 \sigma_1\); \(\gamma' = \sigma_1\sigma_2\); and \(w_0 = \sigma_1\sigma_2\sigma_1 = \sigma_2\sigma_1\sigma_2\).}
  \label{fig:degenerate-tau-automaton}
\end{figure}
\subsection{Consequences of the automaton}\label{sec:a2-automaton-consequences}
In this section, we use the automaton from~\Cref{sec:a2-automaton} to prove results about $m \colon \Stab(\mathcal C)/\mathbf{C} \to \mathbf{P}^{\bfS}$ and \(h \colon \bfS \to \mathbf{P}^{\bfS}\).
We prove that \(m\) is a homeomorphism onto its image (\Cref{prop:homeo1}) and describe the closure of \(h\) (\Cref{prop:pmf}).
In subsequent sections, we establish that the closure of \(h\) is indeed the boundary of the image of \(m\).

Recall that \(X = P_{21}\).
Recall from~\Cref{def:collapsed-triangle} that \(x \to y \to z \xrightarrow{+1}\) is a collapsed triangle if \(m(y) = m(x)  + m(z)\).
\begin{proposition}\label{prop:non-collapsing-triangles-non-deg}
  Consider the following collection of distinguished triangles:
  \[T = \{P_1 \to X \to P_2 \xrightarrow{+1},\quad X \to P_2 \to P_1[1] \xrightarrow{+1},\quad P_2[-1] \to P_1 \to X \xrightarrow{+1}\}.\]
  If \(\tau\) is an off-the-wall standard stability condition of type I, then no triangle in \(T\) is collapsed by \(\tau\).
  Furthermore, if \(\tau'\) is any stability condition such that no triangle in \(T\) is collapsed by \(\tau'\), then \(\tau'\) is standard of type I, up to the action of \(\mathbf{C}\).
\end{proposition}
 
\begin{proof}
  It is immediate from~\Cref{prop:ss-geodesic} (or a direct check) that no triangle in \(T\) is collapsed by any off-the-wall standard stability condition of type \(I\).
  
  Let \(\tau'\) be any stability condition.
  Recall that \(\tau'\) is in the braid group orbit of a standard stability condition \(\tau\) of type I.
  Write \(\tau' = \beta\tau''\) for some braid \(\beta\) and some standard stability condition \(\tau''\) of type I.

  If \(\beta\) is a power of \(\gamma\), then \(\tau'\) is already standard of type I.
  Otherwise, we exhibit a triangle of \(T\) that is collapsed by \(\tau'\).
  Equivalently, we exhibit a triangle of \(\beta^{-1}T\) that is collapsed by \(\tau''\).

Suppose that \(\tau''\) is off-the-wall.
  Consider a cyclic writing of \(\beta^{-1}\) as described in~\Cref{prop:cyclicwriting}.
  Since $\beta$ is not a power of $\gamma$, this writing has rightmost letter either \(\sigma_1\), \(\sigma_2\), or \(\sigma_X\).
  Suppose the rightmost letter of \(\beta^{-1}\) is \(\sigma_1\).
  Consider the first triangle in \(T\), namely
  \(X \to P_2 \to P_1[1] \xrightarrow{+1}\).
  Observe from the automaton in~\Cref{fig:automaton} that the objects \(P_1\), \(P_2\), and \(X\) are all sent to the state \([P_1, P_2]\) after applying \(\sigma_1\).
  Furthermore, we have the equation
  \[HN_{\tau''}(\sigma_1(X)) = HN_{\tau''}(\sigma_1(P_1)) + HN_{\tau''}(\sigma_1(P_2)).\]
  The equation above persists by linearity as we apply each successive letter of the cyclic writing of \(\beta^{-1}\), resulting in the equation
  \[HN_{\tau''}(\beta^{-1}(X)) = HN_{\tau''}(\beta^{-1}(P_1)) + HN_{\tau''}(\beta^{-1}(P_2)),\]
  and hence
  \[m_{\tau''}(\beta^{-1}(X)) = m_{\tau''}(\beta^{-1}(P_1)) + m_{\tau''}(\beta^{-1}(P_2)).\]
  Hence the triangle \(\beta^{-1}X \to \beta^{-1}P_2 \to \beta^{-1}P_1[1] \xrightarrow{+1}\) is collapsed by \(\tau''\).
  Similarly, we can check that if the cyclic writing of \(\beta^{-1}\) ends in \(\sigma_2\) (respectively \(\sigma_X\)), then the second (respectively third) triangle of \(T\) is collapsed by \(\tau'\).

  If \(\tau''\) is on-the-wall, then it is a limit of an on-the-wall stability condition.
  By continuity of the masses, we still conclude that \(\tau' = \beta \tau''\) collapses one of the triangles in \(T\).
\end{proof}
Recall that \(X' = P_{21}\).
\begin{proposition}\label{prop:non-collapsing-triangles-deg}
  Consider the following collection of distinguished triangles:
  \[T = \{X \to P_2 \to P_1[1] \xrightarrow{+1},\quad P_2[-1] \to P_1 \to X \xrightarrow{+1}, \quad X' \to P_1 \to P_2[1] \xrightarrow{+1}\}.\]
  If \(\tau\) is an on-the-wall standard stability condition, then no triangle in \(T\) is collapsed by \(\tau\).
  Furthermore, if \(\tau'\) is any stability condition such that no triangle in \(T\) is collapsed by \(\tau'\), then \(\tau'\) is standard of type I, up to the action of \(\mathbf{C}\).
\end{proposition}
\begin{proof}
  Analogous to the proof of~\Cref{prop:non-collapsing-triangles-non-deg}.
\end{proof}

\begin{proposition}\label{prop:homeo1}
  The map $m \from \Stab(\mathcal C)/\mathbf{C} \to \mathbf{P}^\bfS$ is a homeomorphism onto its image.
\end{proposition}
\begin{proof}
  \Cref{prop:injectivity-general} shows that \(m\) is injective.
  \Cref{prop:conditional-local-homeo} together with~\Cref{prop:non-collapsing-triangles-non-deg} and~\Cref{prop:non-collapsing-triangles-deg} shows that \(m\) is a local homeomorphism at a standard stability condition of type I.
  Since every \(\tau \in \Stab(\mathcal{C})/\mathbf{C}\) is in the braid group orbit of a standard stability condition of type I, it follows that \(m\) is a local homeomorphism at every \(\tau\).
  As a result, $m$ is a homeomorphism onto its image.
\end{proof}

Recall that we have a surjective map $B_3 \to \PSL_2(\mathbf{Z})$.
We use the notation \(/\pm\) to denote the quotient that identifies every element and its negative.
Then we have the standard action of \(\PSL_2(\mathbf{Z})\) on \(\mathbf{Z}^2 / \pm \).
We thus get an action of \(B_3\) on \(\mathbf{Z}^2 / \pm \).
As a result, \(\mathbf{Z}^2/ \pm \) becomes a \(\Theta\)-set.
Recall that we also have the \(\Theta\)-representation \(M\) shown in~\Cref{fig:automaton}, and hence the \(\Theta\)-set \(M/\pm \).

For each of the three states \(v\) of \(\Theta\), we define a linear map \(\phi_v \colon M_v \to \mathbf{Z}^2\) as follows:
  \begin{align*}
    \phi_{[P_1,P_2]} &: (1,0) \mapsto (1,0) \text{ and } (0,1) \mapsto (0,1),\\
    \phi_{[P_2, X]} &: (1,0) \mapsto (0, 1) \text{ and } (0,1) \mapsto (1,-1), \\
    \phi_{[X, P_1]} &: (1,0) \mapsto (1, -1) \text{ and } (0,1) \mapsto (1, 0).
  \end{align*}
\begin{proposition}\label{prop:M-is-standard}
  The maps \(\phi_v\) give an isomorphism of \(\Theta\)-sets \(M/\pm \to \mathbf{Z}^2/ \pm\).
\end{proposition}
\begin{proof}
  All three maps \(\phi_v\) are clearly bijective.
  It remains to check that they are \(\Theta\)-equivariant.
  That is, up to sign, for every edge \(e \colon v \to w\), we have
  \[ e \circ \phi_v = \phi_w \circ M(e).\]
  We omit this straightforward verification.
\end{proof}

Recall that we have a $B_3$-equivariant map $i \from \bfS \to \mathbf{P}^1(\Z)$ defined in~\eqref{eq:num-to-obj}.
The following proposition identifies \(i(s)\) explicitly in terms of the HN multiplicities.
\begin{proposition}\label{prop:iishn}
  For a spherical object \(s\) supported at the state \(v\), we have the equality
  \[i(s) = \phi_v(\HN_{\tau}(s)) \in \mathbf{P}^1(\mathbf{Z}).\]
  Moreover \(\phi_v(\HN_{\tau}(s)) \in \mathbf{Z}^2/\pm\) is the unique representative of \(i(s)\) whose coordinates are relatively prime.
\end{proposition}
\begin{proof}
  The equation evidently holds for \(s = P_1\).
  If it holds for \(s\) supported at \(v\), and \(e \colon v \to w\) is an edge of \(\Theta\) with label \(\beta \in B_{3}\), then using the \(\Theta\)-equivariance of \(i\), \(\phi_v\), and \(\HN_{\tau}\), we see that it also holds for \(\beta s\).
  By~\Cref{prop:cyclicwriting}, it holds for all spherical objects.

  By~\Cref{prop:cyclicwriting}, the vector \(\HN_{\tau}(s)\) is obtained by applying a sequence of matrices in~\Cref{fig:automaton} to the vector \((1,0)\) or \((0,1)\).
  Since all these matrices have determinant \(1\), we see that the coordinates of \(\HN_{\tau}(s)\) are relatively prime.
  Since the linear maps defining \(\phi\) are also of determinant \(\pm 1\), the same is true for \(\phi_v(\HN_{\tau}(s))\).
\end{proof}

By~\Cref{prop:cyclicwriting}, the set \(\bfS\) of spherical objects of \(\mathcal{C}\) is the union of the three sets supported at the three vertices of \(\Theta\).
Under the map \(i \colon \bfS \to \mathbf{P}^1(\mathbf{Z})\), this division corresponds to a geometric division of the circle $\mathbf{P}^1(\R)$, which we now describe.
The three points  $i(P_1) = [1:0]$ and $i(P_2)=[0:1]$, and $i(X) = [-1:1]$ divide $\mathbf{P}^1(\R)$ into three closed arcs (see~\Cref{fig:circle}).
We denote these arcs by $[P_1,P_2]$, $[P_2, X]$, and $[X, P_1]$.
\begin{figure}
  \begin{center}
  \begin{tikzpicture}[thick]
    \draw (0,0) circle (2);
    \draw [fill]
    (180:2) circle (0.05) node [left] {\tiny $X = [-1:1]$}
    (60:2) circle (0.05) node [above right] {\tiny $P_1 = [1:0]$}
    (300:2) circle (0.05) node [below right] {\tiny $P_2 = [0:1]$};
  \end{tikzpicture}
  \caption{The points $P_1$, $P_2$, and $X = P_{21}$ divide $\mathbf{P}^1(\mathbf{R})$ into three arcs.
    The Harder--Narasimhan factors of an object only include the two endpoints of the arc on which the object lies.}\label{fig:circle}
\end{center}
\end{figure}
\begin{proposition}\label{prop:sphericalHN}
  The map \(i\) sends the objects of $\bfS$ supported at the state $[P_1, P_2]$ to the arc $[P_1, P_2] \subset \mathbf{P}^{1}(\mathbf{R})$, and likewise for the other two states.
\end{proposition}
\begin{proof}
  By~\Cref{prop:iishn}, the map $i \colon \bfS \to \mathbf P^1(\mathbf{R})$ is equal to $s \mapsto \phi(HN_{\tau}(s))$.
  For $s$ supported at $[P_1,P_2]$, the vector $HN_{\tau}(s)$ lies in the non-negative cone in $\mathbf{Z}^{\{P_1,P_2\}}$, which is mapped by $\phi_{[P_1,P_2]}$ to the arc $[P_1,P_2]$.
  The argument for the other two states is similar.
\end{proof}

We now use our automaton to give a new and simpler proof of a theorem of Rouquier and Zimmermann~\cite[Proposition 4.8]{rou.zim:03}.
This will be the content of~\Cref{prop:ac}.
Rouquier and Zimmermann work in the homotopy category of complexes of projective modules over the zig-zag algebra, which is formally a different category from the 2-CY category we study in this section.
As explained, e.g., in~\cite[\S 2.3.3]{bap.deo.lic:22}, the effect of passing from their category to ours is to collapse one grading.
Since~\cite[Proposition 4.8]{rou.zim:03} involves forgetting the grading,~\Cref{prop:ac} implies~\cite[Proposition 4.8]{rou.zim:03}.

To explain the argument in more detail, let us first introduce some basic notation.
\begin{definition}\label{def:JH-multiplicity}
  Let \(x\) be an object of \(\mathcal{C}\) belonging to the standard heart.
  The \emph{Jordan--H\"older (JH) multiplicity} of \(P_1\) (resp. \(P_2\)) in \(x\) is the multiplicity of \(P_1\) (resp. \(P_2\)) in any Jordan--H\"older filtration of \(x\).
  More generally, the JH multiplicity of \(P_i\) in an arbitrary \(x\) is the sum of the JH multiplicities of \(P_i\) in each cohomology of \(x\), where cohomology is taken with respect to the \(t\)-structure whose heart is the standard heart.
\end{definition}
\begin{proposition}\label{prop:ac}
  Let $s$ be a spherical object with $i(s) = [a:c] \in \mathbf{P}^1(\mathbf{Z})$, where $a, c$ are relatively prime integers.
  Then the JH multiplicities of \(P_1\) and \(P_2\) in \(s\) are \(|a|\) and \(|c|\) respectively.
\end{proposition}
\begin{proof}
  Since the \(\tau\)-HN filtration refines the cohomology filtration with respect to the standard heart, the JH multiplicities of \(P_1\) and \(P_2\) are linear functions of the \(\tau\)-HN multiplicities.
  Up to sign, these linear functions are precisely the maps \(\phi\) in~\Cref{prop:M-is-standard}.
  The result now follows by~\Cref{prop:iishn}.
\end{proof}

The next proposition relates the JH multiplicities and the \(\homBar\) functionals.
It is an important tool for understanding the closure of the image of the mass map.
\begin{proposition}\label{prop:homs}
  Let $x$ be a spherical object with $i(x) = [a:c] \in \mathbf{P}^1(\mathbf{Z})$, where $a, c$ are relatively prime integers.  
  Then $\homBar(P_2,x) = |a|$ and $\homBar(P_1,x) = |c|$.
\end{proposition}

\begin{proof}
  We prove the assertion for $\homBar(x,P_1)$.
  The other case follows by applying $\gamma$.
  It is easy to check the result for \(x = P_1,P_2,\) and \(X\) by hand.

  Fix a standard off-the-wall stability condition.
  By~\Cref{prop:ac}, \(|c|\) gives the JH-multiplicity of \(P_2\), which is a linear function of the HN-multiplicities.
  The key idea is to prove that \(\homBar(-,P_1)\) is also a linear function of the HN-multiplicities.
  Then the cases \(x = P_{1}, P_2,\) and \(X\) imply the result for all \(x\).

  We now prove the linearity.
  Begin with an object \(x\) supported at $[P_2, X]$.
  Let
  \[ 0 = x_0 \to \cdots \to x_n = x.\]
  be the HN filtration with factors \(z_{i} = \Cone(x_{i-1} \to x_i)\).
  We must prove that
  \[ \homBar(P_1,x) = \sum \homBar(P_1, z_i).\]
  Note that since neither \(x\) nor \(z_i\) is a shift of \(P_1\), we have \(\homBar = \dim\Hom\).

  Apply $\Hom(P_1,-)$ to the HN filtration of \(x\) to get the following filtration in the bounded derived category of graded vector spaces:
  \begin{equation}\label{eqn:hom-filtration}
    0 = \Hom(P_1, x_0) \to \cdots \to \Hom(P_1, x_n) = \Hom(P_1, x).
  \end{equation}
  The factors in~\eqref{eqn:hom-filtration} are \(\Hom(P_1, z_i)\).
  We argue that the connecting map
  \[ \Hom(P_1, z_i) \to \Hom(P_1, x_i[1])\]
  vanishes.
  To see this, consider the map \(z_i \to x_{i-1}[1]\).
  It suffices to prove that for every HN factor \(y\) of \(x_{i-1}[1]\) and every map \(z_i \to y\), the induced map
  \[ \Hom(P_1,z_i) \to \Hom(P_1, y)\]
  vanishes.
  But the HN factors of \(x_{i-1}[1]\) are simply \(z_j[1]\) for \(j < i\).
  Using that both \(z_i\) and \(z_j\) are shifts of \(P_2\) or \(X\) and that \(\phi(z_j) \geq \phi(z_i)\), we see that (up to isomorphisms and shifts) the only non-zero possibilities for \(z_i \to z_{j}[1]\) are 
  \[ P_2 \to P_2[2], \quad X \to X[2], \text{ and } P_2 \to X[2]. \]
  All three are killed by the functor $\Hom(P_1,-)$.

  Since the connecting maps in~\eqref{eqn:hom-filtration} vanish, we get
  \[ \Hom(P_1,x) \cong \bigoplus \Hom(P_1, z_i),\]
  and by taking dimensions
  \[ \hom(P_1,x) = \sum \hom(P_1,z_i).\]
 
  We now treat the case of $x$ supported at the other two states.
  Instead of directly proving linearity, we argue using the group action.
  By~\Cref{prop:ac}, the JH multiplicities of \(P_2\) in $x$ and $\sigma_1^rx$ are equal.
  We also have
  \[\homBar(x, P_1) = \homBar(\sigma_1^rx, P_1).\]
  Hence, the proposition for $x$ implies the proposition for $\sigma_1^rx$.
  By~\Cref{prop:unsink}, any \(x\) supported at \([X,P_1]\) or \([P_{1},P_2]\) can be written as \(x = \sigma_1^ry\) for some \(r \in \mathbf{Z}\) and \(y\) supported at \([P_2,X]\).
Hence the proposition holds for \(x\).
\end{proof}

Consider the action of \(B_3\) on \(\mathbf{R}^2/\pm\) via the homomorphism \(B_3 \to \PSL_2\mathbf{Z}\).
Let \(\pi \colon \mathbf{R}^2 \to \mathbf{R}\) be the second projection.
\begin{proposition}\label{cor:homs}
  Let \(\beta \in B_3\).
  Set \(s = \beta P_1\) and let \(x\) be any spherical object with \(i(x) = [a:c]\), where \(a,c \in \mathbf{Z}\) are relatively prime.
  Then
  \[ \homBar(s,x) = |\pi(\beta^{-1}(a,c))|.\]
  In particular, we have
  \[
   \homBar(P_1,x) = |c|, \quad \homBar(P_2,x) = |a|, \quad \homBar(P_{21}, x) = |a+c|, \text{ and } \homBar(P_{12}, x) = |a-c|.
\]
\end{proposition}
\begin{proof}
  We have
  \[\homBar(s,x) = \homBar(\beta P_1, x) = \homBar(P_1, \beta^{-1}x).\]
  By~\Cref{prop:homs}, the last quantity is \(|\pi(\beta^{-1}(a,c))|\), as required.
  Taking \(\beta = 1, \gamma^{-1}, \sigma_2\), and \(\sigma_2^{-1}\) yield the four equations.
\end{proof}

Recall the map $h \from \bfS \to \mathbf{P}^\bfS$
defined by
\[ h \colon x \mapsto [\homBar(x,-)].\]
Consider $\bfS$ as a subset of $\mathbf{P}^1(\R)$ via the bijection $i \colon \bfS \to \mathbf{P}^1(\Z)$ as defined in~\eqref{eq:num-to-obj}.
\begin{proposition}\label{prop:pmf}
  The map $h \from \bfS \to \mathbf{P}^{\bfS}$ extends to a continuous map $\mathbf{P}^1(\R) \to \mathbf{P}^\bfS$, which is a homeomorphism onto its image.
\end{proposition}
\begin{proof}
  For an \(s \in \bfS\), chose a \(\beta_s \in B_3\) such that \(s = \beta_s P_1\).
  Define
  \(\widetilde h \colon \mathbf{P}^1(\mathbf{R}) \to \mathbf{P}^\bfS\)
  by
  \[ \widetilde h \colon [a:c] \mapsto \left[\left|\pi (\beta_s^{-1}(a,c))\right|\right].\]
  Then \(\widetilde h\) is evidently continuous and by~\Cref{cor:homs}, it extends \(h\).
  
  We check that \(\widetilde h\) is a homeomorphism onto its image.
  Since the domain is compact and the target is Hausdorff, it suffices to check that \(\widetilde h\) is injective.
  Consider the composition of $\widetilde h$ with the projection onto the homogenous coordinates corresponding to the objects \(P_1, P_2\), and \(P_{21}\).
  By~\Cref{cor:homs}, the composition is given by
  \[ [a:c] \mapsto [|a|:|c|:|a+c|].\]
  This map is injective, and hence so is $\widetilde h$.
\end{proof}

\subsection{Gromov coordinates}
Let $\tau$ be a stability condition of type I.
Since the three positive real numbers $m_\tau(P_1)$, $m_\tau(P_2)$, and $m_\tau(X)$ satisfy the triangle inequalities, there exist non-negative real numbers $x$, $y$, $z$ such that
\[ m_\tau(P_1) = y+z, \quad m_\tau(P_2) = z+x, \quad m_\tau(X) = x+y.\]
We call the $x$, $y$, $z$, the \emph{Gromov coordinates}  of $\tau$.
Note that if $\tau$ is off-the-wall, then the Gromov coordinates are all positive.
In the closure \(\Lambda\), one of the coordinates may be zero.
Two of the coordinates cannot be zero.

Recall that $m_\tau \from \bfS \to \mathbf{R}$ is the mass function associated to $\tau$.
Let $\homBar(s) \colon \bfS \to \mathbf{R}$ be the function \(\homBar(s,-)\) as defined in~\Cref{def:hombar}.
\begin{proposition}[Linearity]\label{prop:linearity}
  We have
  \[ m_\tau = x \homBar(P_1) + y \homBar(P_2) + z \homBar(X).\]
\end{proposition}
\begin{proof}
  Let $s \in \bfS$ be a spherical object.
  Suppose $\tau$ is off-the-wall.
  Observe that the HN filtration of $s$ is the same for all off-the-wall type I stability conditions, and its stable factors are $P_1$, $P_2$, and $X$, up to shift.
  Denoting by $a(s)$, $b(s)$, and $c(s)$ the multiplicities of these three, we have
  \[ m_\tau(s) = (y+z) a(s) + (x+z) b(s) + (x+y) c(s).\]
  In particular, $m_\tau(s)$ is linear in the Gromov coordinates.
  Using~\Cref{prop:ac} and~\Cref{prop:homs}, we have
  \begin{align*}
    b(s) + c(s) &= \text{JH multiplicity of $P_2$ in $s$}
                = \homBar(P_1,s), \text{ and }\\
    a(s) + c(s) &= \text{JH multiplicity of $P_1$ in $s$}
                = \homBar(P_2,s).
  \end{align*}
  By applying \(\gamma\) to the first equality, we get
  \(a(s) + b(s) = \homBar(X, s)\).
  It follows that
  \[ m_\tau(s) = x\homBar(P_1,s) + y\homBar(P_2,s)+z\homBar(X, s).\]

  The case of a on-the-wall $\tau$ follows by continuity.
\end{proof}

Let $\tau$ be an arbitrary stability condition.
Then there are three $\tau$ semi-stable objects, say $A$, $B$, $C$, whose classes in the Grothendieck group are (up to sign) the classes of $P_1$, $P_2$, and $X$.
These are obtained simply by applying an appropriate braid to $P_1$, $P_2$, and $X$.
Define the Gromov coordinates $x$, $y$, $z$ for $\tau$ by the condition
\[ m_\tau(A) = y+z, \quad m_\tau(B) = x+z, \quad m_\tau(C) = x+y.\]
Then~\Cref{prop:linearity} implies that we have
\[m_\tau = x \homBar(A) + y \homBar(B) + z \homBar(C).\]

The Gromov coordinates give a nice geometric picture of $\Stab(\mathcal C)/\mathbf{C} \subset \mathbf{P}^\bfS$.
Let $\Delta$ denote the following clipped triangle:
\begin{align*}
  \Delta &= \{(x,y,z) \in \mathbf{R}_{\geq 0}^3 \mid \text{at least two coordinates non-zero} \} / \R_{>0} \\
  &\cong \text{a closed planar triangle minus the three vertices}.
\end{align*}
Recall that $\Lambda \subset \Stab(\mathcal C)/\mathbf{C}$ is the closure of the set of type I stability conditions.
We have a homeomorphism $\Delta \to \Lambda \subset \mathbf{P}^\bfS$ given by the Gromov coordinates:
\begin{equation}\label{eq:gromovhomeo}
  (x,y,z) \mapsto x \cdot \homBar(P_1) + y \cdot \homBar(P_2) + z \cdot \homBar(X).
\end{equation}
Note that the homeomorphism extends to the closed (unclipped) triangle $\overline \Delta$, under which the three vertices are mapped to the points $h(P_1)$, $h(P_2)$, and $h(X)$.
By applying the $B_3$ action, we obtain the picture as shown in~\Cref{fig:exchange-graph}.

Let us describe~\Cref{fig:exchange-graph} in more detail.
Recall that the element $\gamma = \sigma_2\sigma_1$ generates the stabiliser of $\Lambda$ in $\PSL_2(\mathbf{Z})$.
So, we label the various (distinct) translates of $\Lambda$ by $\langle \gamma \rangle $-cosets.
Two translates of $\Lambda$ are either disjoint or intersect along an edge, which is a copy of the open interval.
The three translates that intersect $\Lambda$ along its three edges are $\sigma_1 \Lambda$, $\sigma_X\Lambda$, and $\sigma_2 \Lambda$.
Consequently, the three translates that intersect $\beta\Lambda$ are $\beta \sigma_1 \Lambda$, $\beta \sigma_X\Lambda$, and $\beta\sigma_2\Lambda$.
We can encode the translates and their intersections in a graph (called the \emph{exchange graph} in~\cite{bri.qiu.sut:20}).
Its vertices are left $\gamma$-cosets in $\PSL_2(\Z)$, and a coset $\beta\langle \gamma \rangle$ is connected by an edge with $\beta \sigma_1 \langle\gamma\rangle$, $\beta \sigma_X\langle\gamma\rangle$, and $\beta\sigma_2\langle\gamma\rangle$.

\begin{figure}[ht]
  \centering
  \begin{tikzpicture}[auto, ->, >=stealth', semithick, font=\scriptsize]
    \tikzstyle{every state}=[draw=none, rectangle, rounded corners, font=\footnotesize]
    \tikzstyle{twist}=[fill=white!95!black]
    \tikzstyle{sx}=[bend left=10]
    \tikzstyle{s1}=[bend left=10]
    \tikzstyle{s2}=[bend left=10]
    \tikzstyle{obj}=[fill=white, draw, circle, inner sep=1.5pt, outer sep=0pt]

    \coordinate (P1) at (90:5);
    \coordinate (P2) at (330:5);
    \coordinate (P21) at (210:5);
    \coordinate (P12) at (30:5);
    \coordinate (P221) at (-90:5);
    \coordinate (P211) at (150:5);
    \coordinate (P112) at (60:5.5);
    \coordinate (P122) at (0:5.5);
    \coordinate (P2221) at (-60:5.5);
    \coordinate (X2) at (-120:5.5);
    \coordinate (X1) at (180:5.5);
    \coordinate (P2111) at (120:5.5);
        
    \node[obj] at (P1) {} (P1) node[above] {$P_1$};
    \node[obj] at (P2) {} (P2) node[right] {$P_2$};
    \node[obj] at (P21) {} (P21) node[left] {$X$};
    \node[obj] at (P12) {};
    \node[obj] at (P221) {};
    \node[obj] at (P211) {};
    \node[obj] at (P112) {};
    \node[obj] at (P122) {};
    \node[obj] at (P2221) {};
    \node[obj] at (X2) {};
    \node[obj] at (X1) {};
    \node[obj] at (P2111) {};

    \node[state] (L) at ($0.33*(P1)+0.33*(P2)+0.33*(P21)$) {$\Lambda$};
    \node[state] (L1) at ($0.33*(P1)+0.33*(P2)+0.33*(P12)$) {$\sigma_1\Lambda$};
    \node[state] (L2) at ($0.33*(P21)+0.33*(P2)+0.33*(P221)$) {$\sigma_2\Lambda$};
    \node[state] (Lx) at ($0.33*(P1)+0.33*(P21)+0.33*(P211)$) {$\sigma_X\Lambda$};
    \node[state] (L11) at ($0.25*(P1)+0.25*(P12)+0.5*(P112)$) {\tiny $\sigma_1^2\Lambda$};
    \node[state] (L12) at ($0.25*(P12)+0.25*(P2)+0.5*(P122)$) {\tiny $\sigma_1\sigma_2\Lambda$};
    \node[state] (L22) at ($0.25*(P2)+0.25*(P221)+0.5*(P2221)$) {\tiny $\sigma_2^2\Lambda$};
    \node[state] (L2x) at ($0.25*(P21)+0.25*(P221)+0.5*(X2)$) {\tiny $\sigma_2\sigma_X\Lambda$};
    \node[state] (Lxx) at ($0.25*(P21)+0.25*(P211)+0.5*(X1)$) {\tiny $\sigma_X^2\Lambda$};
    \node[state] (Lx1) at ($0.25*(P211)+0.25*(P1)+0.5*(P2111)$) {\tiny $\sigma_X\sigma_1\Lambda$};

    \begin{scope}[on background layer]
    \path [draw=white!50!black, fill=white!95!black] (P1) -- (P21) -- (P2) -- cycle;
    \path [draw=white!50!black, fill=white!95!black] (P1) -- (P21) -- (P211) -- cycle;
    \path [draw=white!50!black, fill=white!95!black] (P21) -- (P2) -- (P221) -- cycle;
    \path [draw=white!50!black, fill=white!95!black] (P2) -- (P1) -- (P12) -- cycle;
    \path [draw=white!50!black, fill=white!95!black] (P1) -- (P12) -- (P112) -- cycle;
    \path [draw=white!50!black, fill=white!95!black] (P12) -- (P2) -- (P122) -- cycle;
    \path [draw=white!50!black, fill=white!95!black] (P2) -- (P221) -- (P2221) -- cycle;
    \path [draw=white!50!black, fill=white!95!black] (P221) -- (P21) -- (X2) -- cycle;
    \path [draw=white!50!black, fill=white!95!black] (P21) -- (X1) -- (P211) -- cycle;
    \path [draw=white!50!black, fill=white!95!black] (P211) -- (P2111) -- (P1) -- cycle;
    \end{scope}

    \path (L) edge[s1] node[twist] {$\sigma_1$} (L1)
    edge [sx] node[twist] {$\sigma_X$} (Lx)
    edge [s2] node[twist] {$\sigma_2$} (L2)
    (L1) edge[sx] node[twist] {$\sigma_X$} (L) 
    (Lx) edge[s2] node[twist] {$\sigma_2$} (L)
    (L2) edge[s1] node[twist] {$\sigma_1$} (L)
    (L1) edge[s1] node[twist] {$\sigma_1$} (L11) (L11) edge[sx] node[twist] {$\sigma_X$} (L1)
    (L1) edge[s2] node[twist] {$\sigma_2$} (L12) (L12) edge[s1] node[twist] {$\sigma_1$} (L1)
    (L2) edge[s2] node[twist] {$\sigma_2$} (L22) (L22) edge[s1] node[twist] {$\sigma_1$} (L2)
    (L2) edge[sx] node[twist] {$\sigma_X$} (L2x) (L2x) edge[s2] node[twist] {$\sigma_2$} (L2)
    (Lx) edge[sx] node[twist] {$\sigma_X$} (Lxx) (Lxx) edge[s2] node[twist] {$\sigma_2$} (Lx)
    (Lx) edge[s2] node[twist] {$\sigma_1$} (Lx1) (Lx1) edge[s1] node[twist] {$\sigma_X$} (Lx)
    
    ;

    \foreach \x in {0, 60, ..., 300} 
    \path
    (\x+4:5.2) edge [bend left=10] (\x+7:5.5)
    (\x-4:5.2) edge [bend left=10] (\x-7:5.5)
    (\x-9:5.5) edge [bend left=10] (\x-6:5.2)
    (\x+5:5.5) edge [bend left=10] (\x+2:5.2);
    
  \end{tikzpicture}
  \caption{
    The tiling of $\Stab(\mathcal C)/\mathbf{C}$ by the images $\Lambda$ under the braid group.
    The vertices of the triangles correspond to spherical objects.
  }\label{fig:exchange-graph}
\end{figure}

\begin{remark}
The complement of $\Lambda$ in $\Stab(\mathcal C)/\mathbf{C}$ has three connected components.
We can describe these three components in two ways.
For the first, write $\tau'$ in the complement of $\Lambda$ as $\tau' = \beta \tau$, where $\beta \in B_3$ and $\tau \in \Lambda$.
Write $\beta$ in a form recognised by the automaton in~\Cref{fig:automaton}.
Then the three components correspond to the three possible end states of the automaton when it reads the writing of $\beta$.
For the second, recall from the proof of~\Cref{prop:non-collapsing-triangles-non-deg} that the three numbers
\[ m_{\tau'}(P_1), m_{\tau'}(P_2), \text{ and } m_{\tau'}(X)\]
satisfy a collapsed triangle inequality---one is the sum of the other two.
The three connected components correspond to the ways in which the inequality collapses.
\end{remark}

\subsection{The closure}\label{sec:closure}
We know by~\Cref{prop:pre-compactness} that the closure of $\Stab(\mathcal C)/\mathbf{C}$ in $\mathbf{P}^\bfS$ is compact.
Set
\begin{align*}
  P &= \overline{h(\bfS)} = h(\mathbf{P}^1(\R))  \subset \mathbf{P}^\bfS, \text{and}\\
  M &= m\left( \Stab(\mathcal C)/\mathbf{C} \right) \subset \mathbf{P}^\bfS.
\end{align*}
The goal of~\Cref{sec:closure} is to prove that $\overline M = M \cup P$.

Let
\[ \overline \Delta = \left(\R_{\geq 0}^3 \minus 0 \right) / \R_{>0}.\]
Then $\overline \Delta$ is homeomorphic to a closed triangle.
An immediate consequence of the linearity (from~\Cref{prop:linearity}) is the following.
\begin{proposition}[Closure of $\Lambda$]\label{prop:closedtriangle}
  The closure of $\Lambda$ in $\mathbf{P}^\bfS$ is
  \[ \overline\Lambda = \Lambda \cup \left\{\homBar(P_1), \homBar(P_2), \homBar(X)\right\}.\]
  This closure is homeomorphic to $\overline \Delta$.
\end{proposition}
\begin{proof}
  Identify $\Lambda$ with $\Delta$ using the Gromov coordinates as in~\eqref{eq:gromovhomeo}.
  From~\Cref{prop:linearity}, we see that the map $m \from \Delta \to \mathbf{P}^\bfS$ extends to a continuous map $\overline \Delta \to \mathbf{P}^\bfS$.
  Its image is closed, and equals the union of $\Lambda$ and the three additional points described above.
  The map from $\overline \Delta$ is injective, and hence an isomorphism onto its image.
\end{proof}

It is convenient to have a measure of closeness for two elements of a projective space.
Given two non-zero vectors in $\R^n$ for some positive integer $n$, denote by $\sphericalangle(v,w)$ the (acute) angle between them.
By a slight abuse of notation, we use $\sphericalangle(v,w)$ also for two points $v, w \in \mathbf{P}^{n-1}$; this is just the angle between any two representatives in $\R^n$.

Given a finite subset $T \subset \bfS$, let $h_T \from \bfS \to \mathbf{P}^T$ be the composition of $h \from \bfS \to \mathbf{P}^{\bfS}$ and the projection onto the $T$-coordinates.
\begin{proposition}[Group action contracts]\label{prop:contracting}
  Fix two elements $x, y \in \bfS$ and a finite subset $T \subset \bfS$.
  Given an $\epsilon > 0$, for all but finitely many elements $g$ of $\PSL_2(\mathbf{Z})$, we have
  \[ \sphericalangle (h_T(g x), h_T (g y)) < \epsilon.\]
\end{proposition}
\begin{proof}
  Recall the bijection \(i \colon \bfS \to \mathbf{P}^{1}(\mathbf{Z})\).
  Set \(a = i(x)\) and \(b = i(y)\).
  The map $h_T \from \mathbf{P}^1(\R) \to \mathbf{P}^T$ is continuous, and hence uniformly continuous.
  Therefore, it suffices to check that for all but finitely many $g$, the angle between $g a$ and $g b$ in $\R^2$ is small.
  Let \(X\) be the matrix with columns \(a\) and \(b\).
  Then, 
  \[ |\sin \left(\sphericalangle(g a, g b)\right)| = \frac{|\det g| \cdot |\det(X)|}{|g a| \cdot |g b|} = \frac{|\det(X)|}{|g a| \cdot |g b|}.\]
  Since $g \in \mathbf{PSL}_{2}(\mathbf{Z})$ has integer entries, the quantity $|ga| \cdot |g b|$ is greater than $\epsilon^{-1} |\det(X)|$ for all but finitely many $g$.
  The claim follows.
\end{proof}

\begin{proposition}[Closure of $M$]\label{prop:b-closure}
  The sets $M$ and $P$ are disjoint and their union is the closure of $M$.
\end{proposition}
\begin{proof}
  The mass of every object in a stability condition is positive.
  On the other hand, $\homBar(x,x) = 0$, by definition.
  Therefore, $M$ and $P$ are disjoint.

  By~\Cref{prop:closedtriangle}, we know that $P$ is contained in the closure of $M$.
  We now prove that $M \cup P$ is closed.
  Let $\tau_n$ be a sequence in $\Stab(\mathcal C)/\mathbf{C}$ whose images $m(\tau_n)$ in $M$ approach a limit $t \in \mathbf{P}^\bfS$.
  We must show that $t$ lies in $M \cup P$.

  Write $\tau_n = \beta_n \tau'_n$, where $\beta_n$ is a braid and $\tau'_n \in \Lambda$ is a stability condition of type I.
  Let $\overline{\beta_n}$ be the image of $\beta_n$ in $\PSL_2(\Z)$.
  We have two possibilities: the set $\{\overline \beta_n\}$ is finite or infinite.

  If $\{\overline \beta_n\}$ is finite, then the sequence $\{\tau_n\}$ is contained in the union of finitely many translates of the set of standard stability conditions $\Lambda$.
  By~\Cref{prop:closedtriangle}, the closure of $\Lambda$ is contained in $M \cup P$.
  Hence, the closure of the union of finitely many translates of $\Lambda$ is also contained in $M \cup P$.
  So the limit $t$ lies in $M \cup P$.

  Suppose $\{\overline \beta_n\}$ is infinite.
  Set $a_n = \beta_n (P_1)$, $b_n = \beta_n(P_2)$, and $c_n = \beta_n(X)$.
  Since $P$ is compact, we may assume, after passing to a subsequence if necessary, that $h(a_n)$, $h(b_n)$, and $h(c_n)$ have limits in $P$, say $a$, $b$, and $c$.

  We first prove that $a = b = c$.
  To see this, let $T \subset \bfS$ be an arbitrary finite set, and use the subscript $T$ to denote projections to $\mathbf{P}^T$.
  It suffices to show that $a_T = b_T = c_T$.
  Since $\{\overline \beta_n\}$ is infinite, by~\Cref{prop:contracting}, we note that the angles between $h_T(a_n)$, $h_T(b_n)$, and $h_T(c_n)$ approach $0$ as $n$ approaches $\infty$.
  Therefore, the three sequences have the same limits.
  But these limits are $a_T$, $b_T$, and $c_T$.

  Next, we claim that $t = h(a) = h(b) = h(c)$.
  Indeed, by linearity (\Cref{prop:linearity}), we know that there exist non-negative real numbers $x_n$, $y_n$, $z_n$ such that
  \[ m_T(\tau_n) = x_n h_T(A_n) + y_n h_T(B_n) + z_n h_T(C_n).\]
  Taking the limit as $n \to \infty$ in projective space yields
  \[ t = h(a) = h(b) = h(c).\]
  In particular, $t$ lies in $M \cup P$.
  The proof is thus complete.
\end{proof}

\subsection{Homeomorphism to the closed disk}
We take up the final part of the main theorem, namely that $(\overline M, P)$ is a manifold with boundary homeomorphic to the unit disk.
We explicitly construct a homeomorphism from $\overline M$ to the disk, using the unprojectivised, but suitably normalised, mass and hom functions.

Fix
\[ T = \{P_1, P_2, X = P_{21}\}.\]
Define a map
\[ \mu \from \Stab(\mathcal C)/i\R \to \R^T = \R^3\]
by
\[ \mu \from \tau \mapsto (m_\tau(P_1), m_\tau(P_2), m_\tau(X)).\]
Similarly, define
\[ \eta \from \bfS \to \R^T = \R^3\]
by
\[
  \eta(s) = (\homBar(s,P_1), \homBar(s, P_2), \homBar(s,X)).
\]
Thinking of $\bfS$ as $\mathbf{P}^1(\Z) \subset \mathbf{P}^1(\R)$, it is easy to see that $\eta$ extends to a continuous map
\[ \eta \from \mathbf{P}^1(\R) \to \R^3.\]
Indeed, using~\Cref{cor:homs}, we get
\[ \eta \from [a:c] \mapsto (|c|,|a|,|a+c|).\]

Let $\tau$ be a stability condition with (semi)-stable objects $A$, $B$, and $C$ of class $[P_1]$, $[P_2]$, and $[X]$ in the Grothendieck group.
Let $x$, $y$, $z$ be the Gromov coordinates of $\tau$, namely the non-negative real numbers such that
\[m_\tau(A) = y+z, \quad m_\tau(B) = z+x, \quad m_\tau(C) = x+y.\]
Denote by $|-|$ the standard Euclidean norm in $\R^3$.
We say that $\tau$ is \emph{normalised} if
\begin{equation}\label{eq:normalisation}
    x |\eta(A)| + y|\eta(B)| + z|\eta(C)| = 1.
\end{equation}
\begin{remark}
  We point out one subtle aspect of the definition.
  If $\tau$ is on-the-wall, then one of the $A$, $B$, or $C$ is not uniquely determined (even up to shift), since on-the-wall stability conditions have non-isomorphic semistable objects with the same class in the Grothendieck group.
  In this case, however, the corresponding Gromov coordinate is $0$, and hence~\eqref{eq:normalisation} holds for every possible $A,B,C$.
\end{remark}

The normalisation yields a continuous section of
\[ \Stab(\mathcal C)/i\R \to \Stab(\mathcal C)/\mathbf{C}.\]
That is, for every stability condition $\tau$ up to rotation and scaling, there is a unique normalised stability condition $\tau^\nu$ up to rotation, and furthermore the map $\tau \to \tau^\nu$ is continuous.

We now identify $\Stab(\mathcal C)/\mathbf{C}$ with its image $M$ in $\mathbf{P}^\bfS$ and $\mathbf{P}^1(\R)$ with its image $P$ in $\mathbf{P}^\bfS$. 
Define
\begin{equation}\label{eq:pi-interior}
  \pi \from \overline M = M \cup P \to \R^3
\end{equation}
as follows.
For $\tau \in M$, set
\begin{equation}\label{eq:pi-objects}
  \pi(\tau) = \mu (\tau^\nu),
\end{equation}
and for $s \in P$, set
\[ \pi(s) = \eta(s) / |\eta(s)|.\]

By construction, $\pi(a:c)$ is the unit vector in the direction of
\[ \eta(a:c) = (|a|,|c|,|a+c|).\]
That is, $\pi$ maps $P$ to the unit sphere in $\R^3$.
It is easy to see that $\pi|_P$ is injective, and hence a homeomorphism onto its image.
The image consists of three circular arcs, one for each pair of end-points from the three points
\[ \frac{1}{\sqrt 2}(0,1,1), \quad \frac{1}{\sqrt 2}(1,0,1), \text{ and } \frac{1}{\sqrt 2}(1,1,0).\]
The arc joining each pair is a geodesic arc on the unit sphere; that is, the plane it spans passes through the origin.

\Cref{eq:normalisation} and the triangle inequality imply that for every stability condition $\tau$, we have $|\pi(\tau)| \leq 1$.
In fact, for every stability condition, at least two of the Gromov coordinates are non-zero.
Therefore, we actually have a strict inequality $|\pi(\tau)| < 1$.

To get a better understanding of $\pi$, let us study it on the translates of the fundamental domain $\Lambda$.
Let $\beta$ be a braid.
Consider the translate $\beta \overline \Lambda$.
Set
\[A = \beta(P_1),\quad B = \beta(P_2), \quad C = \beta(X).\]
These are the (semi)-stable objects for the stability conditions in $\beta \Lambda$.
Recall that $\beta \overline \Lambda$ is homeomorphic to the projectivised octant $\overline \Delta = \left (\R_{\geq 0}^3 \minus 0\right )/ \R_{>0}$; the homeomorphism is given by the Gromov coordinates.
The normalisation condition~\eqref{eq:normalisation} is a linear condition on the Gromov coordinates.
It cuts out an affine hyperplane slice of the octant, and gives a section of the projectivisation map.
Since $\pi$ is linear in the Gromov coordinates for normalised stability conditions, it maps $\beta \overline \Lambda$ linearly, and homeomorphically, onto the triangle in $\R^3$ with vertices $\pi(A)$, $\pi(B)$, and $\pi(C)$.

Let $\Phi \subset \R^3$ be the union of the triangle
\[ \{(x,y,z) \mid x+y+z = \sqrt{2},\, x+y-z \geq 0,\, y+z-x \geq 0,\, z+x-y \geq 0\},\]
and the three circular segments, each bounded by an edge of the triangle above and the arc in the image of $\pi(P)$ with the same end-points as the edge (see~\Cref{fig:Delta}).
\begin{figure}
  \centering
  \begin{tikzpicture}[scale=2,rotate around y=30, rotate around x=-15]
\draw [->] (0,0,0) -- (1.5,0,0) node [below right] {$x$};
\draw [->] (0,0,0) -- (0,1.5,0) node [above] {$y$};
\draw [->] (0,0,0) -- (0,0,5.0) node [below left] {$z$};
\path[fill=blue!50!white,opacity=0.7] (0,1,1)--(1,0,1)--(1,1,0)--(0,1,1);
\draw[fill=blue!70!white, opacity=0.7, domain=0:60, smooth, variable=\t]plot ({cos(\t) - sin(\t)/sqrt(3)}, {cos(\t) + sin(\t)/sqrt(3)}, {2*sin(\t)/sqrt(3)});
\draw[fill=blue!70!white, opacity=0.7, domain=0:60, smooth, variable=\t]plot ({cos(\t) + sin(\t)/sqrt(3)}, {cos(\t) - sin(\t)/sqrt(3)}, {2*sin(\t)/sqrt(3)});
\draw[fill=blue!70!white, opacity=0.7, domain=0:60, smooth, variable=\t]plot ({2*sin(\t)/sqrt(3)}, {cos(\t) - sin(\t)/sqrt(3)}, {cos(\t) + sin(\t)/sqrt(3)});
\end{tikzpicture}
   \caption{The homeomorphic image $\Phi \subset \R^3$ of the compactified stability manifold.}\label{fig:Delta}
\end{figure}

Observe that $\pi$ maps $\overline \Lambda$, the triangle formed by the standard stability conditions, to the central triangle of $\Phi$.
On the other hand, the translates of $\overline \Lambda$ on the three sides of the identity on the exchange graph (\Cref{fig:exchange-graph}) are mapped to the three circular segments.
Indeed, the segment to which $\pi$ sends a non-standard stability condition $\tau$ depends on how the triangle inequality collapses among the $\tau$-masses of $P_1$, $P_2$, and $X$, and this in turn, is determined by where $\tau$ lies on the exchange graph.

\begin{proposition}\label{prop:r3-triangle}
  The map $\pi \from \overline M \to \Phi$ is a homeomorphism, where $\overline M$ is given the subspace topology from $\mathbf{P}^\bfS$.
\end{proposition}
\begin{proof}
  Let us first show that $\pi$ is continuous.
  Since $\pi$ is linear on the translates of $\overline \Lambda$, and these translates cover $M$, it follows that $\pi$ is continuous on $M$.
  The restriction of $\pi$ to $P$ is given by
  \[ \pi \from [a:c] \mapsto \frac{1}{\sqrt{a^2+c^2+(a+c)^2}}\left( |a|, |c|, |a+c| \right),\]
  which is continuous.
  We must show that $\pi$ is continuous on $\overline M$ at a point $p \in P$.
  
  Let $\tau_n$ be a sequence of normalised stability conditions converging to $p \in P$.
  We already know that $m_T(\tau_n)$ converges to $h_T(p)$ in the projective space $\mathbf{P}^T$.
  Therefore, it suffices to show that $\pi(\tau_n)$ approaches a vector of norm 1 in $\R^3$.

  Let $\epsilon > 0$ be given.
  Write $\tau_n = \beta_n \tau'_n$ for some braid $\beta_n$ and standard stability condition $\tau_n$.
  Denote by $\overline \beta_n$ the image of $\beta_n$ in $\PSL_2(\Z)$.
  If the set $\{\overline \beta_n\}$ is finite, then the sequence $\tau_n$ lies in finitely many translates of $\Lambda$.
  Without loss of generality, we may assume that it lies in one translate, say $\beta \Lambda$.
  Recall that $\pi \from \overline \beta \Lambda \to \R^3$ is linear in the Gromov coordinates and hence continuous.
  Therefore, $\pi(\tau_n) \to \pi(p)$ as $n \to \infty$.

  The harder case is when the set $\{\overline \beta_n\}$ is infinite.
  Set $A_n = \beta_n(P_1)$, $B_n = \beta_n(P_2)$, $C_n = \beta_n(X)$, and let $x_n$, $y_n$, $z_n$ be the Gromov coordinates.
  But in this case, we know by~\Cref{prop:contracting} that the angle between $\eta(A_n)$, $\eta(B_n)$, and $\eta(C_n)$ approaches $0$ as $n$ approaches $\infty$.
  Therefore, the difference between
  \[ |x_n \eta(A_n) + y_n \eta(B_n) + z_n \eta(C_n)| \text{ and } x_n |\eta(A_n)| + y_n |\eta(B_n)| + z_n |\eta(C_n)|\]
  approaches $0$ as $n$ approaches $\infty$.
  Since $\tau_n$ is normalised, the right-hand quantity is $1$, and the left-hand quantity is $|\pi(\tau_n)|$.

  We have now proved that $\pi \from \overline M \to \Phi$ is continuous.
  Since $\overline M$ is compact, $\pi$ is a homeomorphism once we know that it is a bijection.
  We know that the map
  \[\pi \from P \to \partial \Phi = \Phi \cap \{v \mid |v| = 1\}\]
  is a bijection and $\pi$ maps $M$ to
  \[\Phi^\circ  = \Phi \cap \{v \mid |v| < 1\}.\]
  Recall that $M$ is the union of the translates of the fundamental domains $\beta \Lambda$, and each fundamental domain is homeomorphic to a clipped triangle (a planar triangle minus the vertices).
  The map $\pi$ maps each translate bijectively to a (clipped) triangle in $\Phi^\circ$.
  To check that $\pi$ is an injection, we must check that the clipped triangles $\pi(\beta \Lambda)$ have disjoint interiors, and two of them, say $\pi(\beta\Lambda)$ and $\pi(\beta'\Lambda)$, intersect along an edge if and only if $\beta$ and $\beta'$ are adjacent in the exchange graph.
  To check that $\pi$ is a surjection, we must check that the union $\bigcup_\beta \pi(\beta \Lambda)$ is $\Phi^\circ$.

  Take $\beta = \id$.
  Then $\pi(\Lambda)$ is the central triangle of $\Phi^\circ$.
  The only other triangles that intersect this triangle are $\pi(\sigma_1 \Lambda)$, $\pi(\sigma_2\Lambda)$, and $\pi(\sigma_X\Lambda)$, and the intersections are along the three edges, as required.
  
  The exchange graph divides the non-trivial translates of $\Lambda$ into three connected components, and the translates in each of the three components map to the three distinct circular segments in $\Phi^\circ$.
  So it suffices to restrict our attention to one component and the corresponding circular segment.
  Since $\gamma$ permutes the components, it suffices to look at only one of them.
  
  Let us consider the component containing $\sigma_1 \Lambda$.
  The translates in this component are $A \Lambda$ where $A \in \PSL_2(\Z)$ is a matrix that has a cyclic writing (as in~\Cref{prop:cyclicwriting}) that starts with $\sigma_1$.
  Suppose $A = \begin{pmatrix} a & b \\ c & d \end{pmatrix}$.
  Then $\pi(A\Lambda)$ is the clipped triangle with vertices
  \begin{equation}\label{eqn:triangles}
    \pi([a:c]),\quad \pi([b:d]),\quad \pi([a-b:c-d]).
  \end{equation}
  Since we are considering the translates in the $\sigma_1\Lambda$ component, the points $[a:c]$, $[b:d]$, and $[a-b:c-d]$ lie in the arc $[P_1,P_2]$ of $\mathbf{P}^1(\mathbf{Z}) \subset \mathbf{P}^1(\R)$.
  (So, under the bijection $\mathbf{P}^1(\mathbf{Z}) = \Q \cup \{\infty\}$ given by $[a:c] \to a/c$, they correspond to $\Q_{\geq 0} \cup \{\infty\}$).
  By construction, the points $\pi([a:c])$ and $\pi([b:d])$ form an edge of a clipped triangle if and only if $|ad-bc| = 1$.
  Thus, the triangles~\eqref{eqn:triangles} form the Farey tiling~\cite[Chapter 8]{bon:09} of the circular segment, which has the intersection properties as dictated by the exchange graph.
  \end{proof}

\section{The \texorpdfstring{$\widehat A_1$}{affine A1} case}\label{sec:aonehat}
The aim of this section is to construct the Thurston compactification at \(q = 1\) of stability space for the 2-Calabi--Yau category $\calC_\Gamma$ where $\Gamma$ is the $\aonehat$ graph.
\subsection{The 2-Calabi--Yau category of type $\aonehat$}
Since the $\aonehat$ graph is not simply-laced, some mild modifications of the definitions from~\Cref{sec:cgamma} are required in order to define the associated zigzag algebra and 2-Calabi--Yau category.  We give the precise definitions here.

\subsubsection{The Coxeter system of type $\aonehat$}
Let $\Gamma$ be the Dynkin diagram of type $\aonehat$, which has two vertices, called $0$ and $1$, and two (unoriented) edges connecting them.
\begin{center}
  \begin{tikzcd}
    0 \arrow[no head,shift left]{r} \arrow[no head,shift right]{r} & 1
  \end{tikzcd}
\end{center}

\begin{definition}
The \emph{Artin--Tits braid group associated to $\Gamma$} is the free group on two letters $F_2$:
\[ B_\Gamma \cong F_2 = \langle \sigma_0,\sigma_1 \rangle.\]
\end{definition}
\begin{definition}
  The \emph{Coxeter group associated to \(\Gamma\)} is isomorphic to the quotient of $B_\Gamma$ by the relations
  \[\sigma_0^2 = \sigma_1^2 = 1.\]
We denote this group by \(W_{\Gamma}\), and denote the images of \(\sigma_i\) in \(W_{\Gamma}\) by \(s_i\).
\end{definition}
Let \(V_{\Gamma}\) be the rank two \(\mathbf{Z}\)-module with basis vectors \(v_i\) indexed by the vertices of \(\Gamma\).
Define a bilinear form on \(V_{\Gamma}\) by the following formula:
\[
  \langle v_i, v_j \rangle =
  \begin{cases}2&i = j,\\
    -2&i\neq j.\\
  \end{cases}
\]

\begin{definition}
  The \emph{standard representation} of \(W_{\Gamma}\) is the action on \(V_{\Gamma}\) defined by the formula
  \[s_i(v_j) = v_j - \langle v_i, v_j \rangle v_i.\]
\end{definition}

\subsubsection{The zigzag algebra of type $\aonehat$}
Let $\Gamma^{dbl}$ denote the doubled quiver of $\Gamma$.  If the two edges of $\Gamma$ are denoted $e$ and $f$, then there are four oriented edges denoted $e_{01},e_{10},f_{01},f_{10}$ in $\Gamma^{dbl}$, where the subscript $ij$ means that the arrow points from $i$ to $j$.  Fix a sign $s_{ij}\in\{\pm1\}$ for each arrow $e_{ij}$, and a sign $t_{ij}\in\{\pm1\}$ for each arrow $f_{ij}$, such that $s_{ij}=-s_{ji}$ and $t_{ij}=-t_{ji}$.

\begin{center}
  \begin{tikzcd}[column sep=6em]
    0 \arrow[bend left,shift left]{r}{e_{01}} \arrow[bend left,shift right]{r}[below]{f_{01}} & 1 \arrow[bend left,shift left]{l}{e_{10}} \arrow[bend left,shift right]{l}[above]{f_{10}}
  \end{tikzcd}
\end{center}
The zigzag algebra is the quotient of the path algebra $\mbox{Path}(\Gamma^{dbl})$ by the two-sided ideal generated by
\begin{enumerate}

\item the elements $e_{ij}f_{ji}$ and $f_{ij}e_{ji}$, for $i\neq j$, and
\item the elements
$$
s_{ij}e_{ij}e_{ji} - t_{ij}f_{ij}f_{ji},
$$
for $i\neq j$.
\end{enumerate}

The algebras corresponding to the two kinds of sign choices $s_{ij}=t_{ij}$ and
$s_{ij}=-t_{ij}$ are isomorphic, as was the case in simply laced type.  

\subsubsection{The 2-Calabi--Yau category of type $\aonehat$}
The strongly 2-Calabi--Yau category associated to the $\aonehat$ quiver, which we denote as $\mathcal C$ in what follows, is now defined exactly as in simply-laced type in~\Cref{sec:cgamma}, using the above $\aonehat$ zigzag algebra.  
The category $\mathcal C$ is a graded, $\k$-linear triangulated category classically generated by two spherical objects, which we denote $P_0$ and $P_1$.  The hom spaces between these generators are given as follows:
 \begin{align*}
    \gradedHom(P_i, P_i[n]) &=
    \begin{cases}
      \k & \text{ if $n=0 $ or $n=2$,}\\
      0 & \text{ otherwise};
    \end{cases}\\
    \gradedHom(P_i, P_j[n]) &=
    \begin{cases}
      \k^2 & \text{ if $n=1$ and $i\neq j$,}\\
      0 & \text{ otherwise, for $i\neq j$}.
    \end{cases}
  \end{align*}
 The extension closure of $P_0$ and $P_1$ in $\mathcal C$ is an abelian category; it is the heart of a (bounded) $t$-structure.
We refer to it as the \emph{standard heart} $\heart_{\std}$.

\subsection{The spherical objects and the boundary circle}

Denoting the twist in $P_i$ as $\sigma_{P_i}$, we have a weak action of $B_\Gamma \cong F_2$ on $\calC$ given by
\[\sigma_i \to \sigma_{P_i}.\]
The image of $B_\Gamma$ in $\Aut(\calC)$ is precisely the subgroup generated by the spherical twists $\sigma_{P_0}$ and $\sigma_{P_1}$.
It is not difficult to check directly, via e.g.\ a ping-pong argument, that the action of \(B_{\Gamma}\) on \(\mathcal{C}\) is faithful, so we can regard \(B_{\Gamma}\) as a subgroup of \(\Aut(\calC)\).

Set \(\gamma = \sigma_1\sigma_0\).
One can check that for any integer $n$, the object $\sigma_1\gamma^n(P_0)$ is the cone of a morphism from a direct sum of $|2n+2|$ copies of $P_1$ to a direct sum of $|2n+1|$ copies of $P_0$.
Similarly, the object $\gamma^n P_1$ is the cone of a morphism from a direct sum of $|2n+1|$ copies of $P_1$ to a direct sum of $|2n|$ copies of $P_0$.
For this reason, we adopt the following notation for these objects:
\begin{align*}
  \sigma_1\gamma^n(P_0) &= |2n+2|P_1 \to |2n+1|P_0, \text{ and }\\
  \gamma^n P_1 &= |2n+1|P_1 \to |2n|P_0.
\end{align*}

The objects obtained as above are all spherical, and their union consists of each of the objects \(n P_1 \to (n \pm 1)P_0\) exactly once, as \(n\) ranges over the integers.
This set is in bijection with \(\mathbf{Z}\) via the assignment
\[P_n = |n|P_1 \to |n-1|P_0 \text{ for } n \in \mathbf{Z}.\]
The objects \(P_n\) fit into distinguished triangles as follows:
\begin{equation}\label{eq:aonehat-twist-triangles}
  \begin{split}
    P_{2i+1}[-1] \to P_{2i} \oplus P_{2i} \to P_{2i-1} \xrightarrow{+1}, \\
    P_{2i+2} \to P_{2i+1} \oplus P_{2i+1} \to P_{2i}[1] \xrightarrow{+}.
  \end{split}
\end{equation}

Regard $\mathbf{Z}$ as the subset of $\mathbf{P}^1({\mathbf{Z}})$ consisting of points of the form $[k:1]$.
We can write a homomorphism $B_{\Gamma} \to \PSL_2(\mathbf{Z})$, as follows:
\begin{equation}\label{eq:bgamma-psl2-aonehat}
\sigma_0 \mapsto
\begin{pmatrix}
  1&0\\2&1
\end{pmatrix}, \quad
\sigma_1 \mapsto
\begin{pmatrix}
  -3&2\\-2&1
\end{pmatrix}.
\end{equation}

Let $\bfS$ be the set of all spherical objects, up to shift, that are stable with respect to some stability condition in \(\Stab^{\circ}(\mathcal{C})\).
It can be checked, for example using the explicit construction of spherical stable objects in~\cite[\S~4]{bap.deo.lic:22}, that \(\bfS\) is the \(B_{\Gamma}\) orbit of \(\{P_{0}, P_{1}\}\).
We have a $B_{\Gamma}$-equivariant map $i \from \bfS \to \mathbf{P}^1(\mathbf{Z})$ characterised by $P_n \mapsto [n:1]$.
It is easy to check that $i$ is injective.
For convenience, let $P_{\infty}$ be any one of the indecomposable extensions of $P_0$ by $P_1$.
We extend the embedding above to \(\bfS \cup \{P_\infty\}\) by sending \(P_\infty\) to \([1:0]\), which is the point at infinity.

The element $\gamma = \sigma_1\sigma_0$ acts on the objects \(P_n\) by an infinite-order cyclic rotation by two steps, lowering the phases.
More precisely, 
\begin{align*}
  \gamma &\from P_{n} \mapsto P_{n+2}\text{ for }n \neq 0,-1,\infty,\\
  \gamma &\from P_{n} \mapsto P_{n+2}[-1]\text{ for }n = 0,-1,\\
  \gamma &\from P_{\infty} \mapsto P_{\infty}.
\end{align*}
For each $k\in \mathbf{Z}$, set
\[ \sigma_{2k} = \gamma^k\sigma_0\gamma^{-k} \text{ and }\sigma_{2k+1} = \gamma^k\sigma_1\gamma^{-k}.\]
It is easy to see that the image of $\sigma_k$ in \(\Aut(\mathcal{C})\) is the spherical twist in the object $P_{k}$.
The image of $\sigma_k$ in $\PSL_2(\mathbf{Z})$ is given by
\begin{equation}\label{eqn:aonehat-pgl2}
  \sigma_k \mapsto
  \begin{pmatrix}
    2  k + 1 & -2  k^{2} \\
    2 & -2  k + 1
  \end{pmatrix}.
\end{equation}

\subsection{Standard stability conditions}
We say that a stability condition $\tau$ is \emph{standard} if its heart $\heart([0,1))$ is the standard heart $\heart_{\rm std}$.
Let $\phi$ denote the phase function for $\tau$.
Then we either have $\phi(P_0) \leq \phi(P_1)$ or $\phi(P_1) \leq \phi(P_0)$.
In the first case, we say $\tau$ is standard of \emph{type I}, and in the latter case, of \emph{type II}.
Note that, in type I, all objects \(P_n\) are $\tau$-semistable.

Let $\Lambda \subset \Stab(\mathcal C)/\mathbf{C}$ be the closure of the set of standard stability conditions of type I.
Then $\Lambda$ includes non-standard stability conditions where $\phi(P_1) = \phi(P_0)+1$.
We continue to call these stability conditions as of \emph{type I}, reserving the adjective standard for those where $\phi(P_1) < \phi(P_0)+1$.
\Cref{fig:std-stab-a1hat} shows the central charges of the objects \(P_n\) in a standard off-the-wall stability condition of type I.
The dotted line is the direction of the sum of the central charges of $P_0$ and $P_1$, which is also the central charge of $P_{\infty}$.
\begin{figure}[ht]
  \centering
  \begin{tikzpicture}[thick]
    \draw (0:0) edge [->] (0:2) (0:2) node [right] {$P_0$};
    \draw (0:0) edge [->] (150:2) (150:2) node [above left] {$P_1$};
    \draw (0:0) edge [->] (125:2.5) (125:2.5) node [above left] {$P_{2}$};
    \draw (0:0) edge [->] (110:3) (110:3) node [above left] {$P_{3}$};
    \draw (90:2) node {$\cdots$};
    \draw[dashed] (-105:2) edge (75:4) (75:4) node [above right] {$P_{\infty}$};
    \draw (60:2) node {$\ddots$};
    \draw (0:0) edge [->] (20:2.5) (20:2.5) node [above right] {$P_{-1}$};
    \draw (0:0) edge [->] (40:3) (40:3) node [above right] {$P_{-2}$};
  \end{tikzpicture}
  \caption{The central charges of the objects of $T$ in a standard off-the-wall stability condition of type I.}\label{fig:std-stab-a1hat}
\end{figure}

\subsection{The automaton}\label{subsec:aonehat-automaton}
Fix an off-the-wall stability condition $\tau$ of type I.
Recall that \(\Sigma\) is the set of indecomposable \(\tau\)-semi-stable objects in the standard heart, which contains \(P_n\) for \(n \in \mathbf{Z}\) as well as the objects \(P_{\infty}\).
\Cref{fig:aonehat-automaton} describes a $\tau$-HN-automaton $\Theta$ by illustrating the incoming and outgoing edges at one of the vertices.

\begin{figure}[ht]
  \centering
  \begin{tikzpicture}[->, >=stealth', semithick, font=\scriptsize]
    \tikzstyle{every state}=[font=\scriptsize,fill=black!10!white, draw=black!20!white, rectangle]
    \tikzstyle{support}=[font=\scriptsize,fill=black!10!white, draw=black!20!white, rectangle, inner sep=0.5em]    
    \tikzstyle{label}=[anchor=center,sloped, fill=white, font=\scriptsize]        
    \node[support] (k) at (-90:4) {$[P_k,P_{k+1}]$};
    \node[support] (m1) at (-150:4) {$[P_{k-1},P_k]$};
    \node[support] (m2) at (-195:4) {$[P_{k-2}, P_{k-1}]$};
    \node[support] (m3) at (-225:4) {$[P_{k-3}, P_{k-2}]$};
    \node[support] (p1) at (-30:4) {$[P_{k+1}, P_{k+2}]$};
    \node[support] (p2) at (15:4) {$[P_{k+2}, P_{k+3}]$};    
    \node[support] (p3) at (45:4) {$[P_{k+3}, P_{k+4}]$};

    \path[->]
    (k) edge[loop below] node[font=\scriptsize] {$\sigma_k$} ()
    (k) edge node[label] {$\sigma_{k-1}$} (m1)
    (k) edge[bend right=10] node[label] {$\sigma_{k-2}$} (m2)
    (k) edge[bend right=10] node[label] (l) {$\sigma_{k-3}$} (m3)
    (k) edge[bend left=10] node[label] {$\sigma_{k+2}$} (p2)
    (k) edge[bend left=10] node[label] (r) {$\sigma_{k+3}$} (p3)
    (k) edge[bend right=10] node[label] {$\gamma$} (p2)
    (k) edge[bend left=10] node[label] {$\gamma^{-1}$} (m2)    
    ;
    \draw[very thick,dotted,-] (120:4) arc (120:115:4);
    \draw[very thick,dotted,-] (60:4) arc (60:65:4);
    \draw[very thick,dotted,-] (-0.8,0) arc (120:115:3);
    \draw[very thick,dotted,-] (0.8,0) arc (60:65:3);    
  \end{tikzpicture}
\caption{An HN automaton for an off-the-wall type I stability condition for $\aonehat$, showing the outgoing edges from $[P_k,P_{k+1}]$.}\label{fig:aonehat-automaton}
\end{figure}
Here is a more formal description of the automaton, following~\Cref{def:hn-automaton}.
\begin{enumerate}
\item We have a \(B_{\Gamma}\)-labelled graph with vertices indexed by \(\mathbf{Z}\) and edges as indicated in~\Cref{fig:aonehat-automaton}:
  an edge from state $k$ to state $j$ labelled $\sigma_j = \sigma_{P_{j}}$ (unless $j = k+1$); an edge from state  $k$ to state $(k+2)$ labelled $\gamma$; and an edge from state $k$ to state $(k-2)$ labelled $\gamma^{-1}$.
\item The $k$th state is labelled \([P_k,P_{k+1}]\).
  The set \(S_{[P_k, P_{k+1}]}\) consists of the spherical objects of \(\mathcal{C}\) whose HN factors are shifts of $P_k$ and $P_{k+1}$.
\item The representation $M$ associates the free abelian group \(\mathbf{Z}^{\Sigma}\) to each state.
However, as we show in~\Cref{prop:aonehat-automaton}, the image of the HN multiplicity map from \(S\) to \(M\) will only be supported on the sub-representation that assigns \(\mathbf{Z}^{\{P_k,P_{k+1}\}}\) to the state labelled as \([P_k,P_{k+1}]\).
So we simplify notation by setting \(M\) to be this sub-representation.
  The map associated to the edge $\sigma_j \from k \to j$ is represented by the \(2 \times 2\) matrix
  \begin{equation}\label{eqn:aonehat-me}
    \begin{pmatrix}
    |j-k-1| & |j-k-2| \\
    |j-k| & |j-k-1|
  \end{pmatrix}
\end{equation}
in the standard bases.
The maps associated to the edges labelled $\gamma^\pm$ are represented by the identity matrix.
\end{enumerate}

Analogously to~\Cref{prop:automaton}, the next proposition shows that the description above fits into the framework of~\Cref{def:hn-automaton}.
\begin{proposition}\label{prop:aonehat-automaton}
  Let $e \colon [P_k,P_{k+1}] \to [P_j,P_{j+1}]$ be an edge of the automaton labelled by the group element \(g \in B_{\Gamma}\).
  Let \(x \in \calC\) be an object whose HN filtration consists of shifts of \(P_k\) and \(P_{k+1}\).
  \begin{enumerate}
  \item The HN filtration of $gx$ consists of shifts of \(P_j\) and \(P_{j+1}\).
  \item The following diagram commutes.
    \begin{center}
      \begin{tikzcd}
        x \arrow[mapsto]{r}{HN_{\tau}} \arrow[mapsto]{d}{e} & HN_{\tau}(x) \in \mathbf{Z}^{\{P_k,P_{k+1}\}} \arrow[mapsto, shift right=10]{d}{M(e)}\\
        gx \arrow[mapsto]{r}{HN_{\tau}} & HN_{\tau}(gx) \in \mathbf{Z}^{\{P_j,P_{j+1}\}}
      \end{tikzcd}
    \end{center}
    In other words,
    \[ HN_\tau(gx) = M(e) \cdot HN_\tau(x).\]
  \end{enumerate}
\end{proposition}
\begin{proof}
  By an argument similar to~\Cref{prop:automaton}, the image under \(e\) of the HN filtration of \(x\) is rectifiable.
  So it suffices to prove the proposition for \(x = P_k\) and \(x = P_{k+1}\).

  Suppose \(g = \gamma^\pm\) so that $j = k \pm 2$.
  Since $\gamma^\pm$ sends $P_i$ to $P_{i\pm 2}$ (up to shift), preserving the ordering of the phases, it is clear that \(e \cdot x\) is supported at \([P_j,P_{j+1}]\), and that the HN multiplicities transform by the identity matrix.

  Otherwise, \(g = \sigma_j\).
  Suppose that \(k = 0\).
  We can check explicitly that if \(j \neq 1\), we have
  \begin{equation}
    \label{eq:sigma-j}
    \begin{split}
      \sigma_{j}P_0 &= P_{j}^{\oplus |j-1|}[-1] \to P_{j+1}^{\oplus |j|}, \text{ and }\\
      \sigma_j P_1 &= P_{j}^{\oplus |j-2|}[-1] \to P_{j+1}^{\oplus |j-1|}.
    \end{split}
  \end{equation}
  As a result, both $\sigma_j(P_0)$ and $\sigma_j(P_1)$ are both supported at state $j$ and their HN multiplicities transform according to the matrix $M(e)$.

  Let $k$ be arbitrary.
  Observe that we have
  \[\sigma_j(P_k) =
    \begin{cases}
      \gamma^m \sigma_{j-2m}P_0,&\text{if }k = 2m \text{ is even},\\
      \gamma^m \sigma_{j-2m}P_1,&\text{if }k = 2m+1 \text{ is odd}.
    \end{cases}
  \]
  Using the above, we see that for any $k$ except $k = j-1$, we have
  \begin{equation}\label{eqn:update-rules-a1hat}
    \sigma_jP_{k} = P_{j}^{\oplus |j-k-1|}[-1] \to P_{j+1}^{\oplus |j-k|}.
  \end{equation}
  As a result, $\sigma P_k$ and $\sigma P_{k+1}$ are both supported at state $j$ and their HN multiplicities transform according to the matrix $M(e)$.
\end{proof}

The next proposition shows that the automaton recognises every braid and every object of \(\bfS\).
\begin{proposition}\label{prop:cyclicwriting-a1hat}
  Let \(\beta \in F_2\) and let \(s \in \bfS\) be arbitrary.
  \begin{enumerate}
  \item The automaton in~\Cref{fig:aonehat-automaton} recognises \(\beta\).
    More precisely, \(\beta\) has the recognised expression
    \[ \beta = \gamma^n \sigma_{a_1}\sigma_{a_2}\dots\sigma_{a_k},\]
    such that $a_i - a_{i+1} \neq 1$ for each $i$.
  \item The automaton in~\Cref{fig:aonehat-automaton} recognises \(s\).
  \item The \(\tau\)-HN filtration of \(s\) consists of at most two objects out of \(T\), up to shift.
  \end{enumerate}
\end{proposition}
\begin{proof}
  For (1), first write \(\beta\) as any expression in the generators \(\{\sigma_0^{\pm 1},\sigma_1^{\pm 1}\}\).
  For any \(i \in \mathbf{Z}\), we have
  \[ \sigma_i^{-1} = \sigma_{i-1}\gamma^{-1}, \quad \gamma^{-1}\sigma_i = \sigma_{i-2}\gamma^{-1},\quad \gamma = \sigma_i\sigma_{i-1}.\]
  Using these relations, we first rewrite $\beta$ solely in terms of the elements $\sigma_i$ for $i \in \mathbf{Z}$ and $\gamma^{\pm}$.
  Then we successively replace any instances of $\sigma_i\sigma_{i-1}$ by $\gamma$, and commute all powers of $\gamma$ to the left.
  This process terminates, resulting in an expression of the desired form.

  For (2), write \(s = \beta P_i\) where \(i\) is either \(0\) or \(1\).
  Assume that \(i = 0\); the other case is similar.
  Rewrite
  \[\beta = \gamma^n \sigma_{a_1}\sigma_{a_2}\dots\sigma_{a_k}\]
  as in (1).
  Note that the object \(P_0\) is supported at two states, namely \([P_0,P_1]\) and \([P_{-1},P_0]\).
  We can apply \(\sigma_{a_k}\) at least at one of these two states.
  From the condition on our recognised expression, it is clear that all subsequent letters are applicable.
  To finish, any power of \(\gamma\) is applicable at any vertex.
  Therefore \(s\) is recognised by the automaton.

  Finally, (3) follows from (2) and from~\Cref{prop:aonehat-automaton}
\end{proof}
\begin{proposition}\label{prop:unsink-aonehat}
  Let \(x\) be a spherical object of \(\mathcal{C}\) supported at the state \([P_k, P_{k+1}]\).
  Assume that \(x\) is not a shift of \(P_k\) or \(P_{k+1}\).
  Then
  \begin{enumerate}
  \item we can write \(s = \sigma_k^rt\) for some \(r > 0\) and an object \(t\) supported at \([P_{k+1},P_{k+2}]\);
  \item we can write \(s = \sigma_{k+1}^{-r}t\) for some \(r > 0\) and an object \(t\) supported at \([P_{k-1},P_k]\).
  \end{enumerate}
\end{proposition}
\begin{proof}
Analogous to the proof of~\Cref{prop:unsink}.
\end{proof}

\begin{remark}\label{rem:new-degenerate-tau-automaton}
  Let $\tau$ be an on-the-wall standard stability condition.
  Set \(P_i'\) to be the object \(\sigma_0P_{i-1}\).
  The objects \(P_i'\) are also semistable in \(\tau\), in addition to the objects \(P_i\).
  Note that up to shift, we have \(P_1' = P_0\) and \(P_0' = P_1\).
  Recall that \(\gamma = \sigma_1 \sigma_0\).
  Let \(\gamma' = \sigma_0 \sigma_1\).

  \Cref{fig:aonehat-degenerate-tau-automaton} shows an automaton that computes $\tau$-HN multiplicities.
  Only the outgoing arrows labelled \(\sigma_0\) and \(\sigma_1\) are depicted.
  The remaining possible labels for outgoing arrows are \(\gamma^\pm\) and \(\gamma'^{\pm}\), which act as follows:
  \begin{align*}
    \gamma &\colon [P_i,P_{i+1}] \to [P_{i+2},P_{i+3}]\text{ unless }i = -2, \quad
             \gamma \colon [P'_i,P'_{i+1}] \to [P_1,P_2];\\
    \gamma^{-1} &\colon [P_{i},P_{i+1}] \to [P_{i-2},P_{i-3}], \text{ unless }i = 2,\quad
             \gamma^{-1} \colon [P'_i,P'_{i+1}] \to [P_{-2},P_{-1}];\\    
    \gamma'&\colon [P_i',P_{i+1}'] \to [P_{i+2}',P_{i+3}']\text{ unless }i = -2,\quad
             \gamma' \colon [P_i,P_{i+1}] \to [P'_1,P'_2];\\
    \gamma'^{-1}&\colon [P_i',P_{i+1}'] \to [P_{i-2}',P_{i-3}']\text{ unless }i = 2,\quad
             \gamma'^{-1} \colon [P_i,P_{i+1}] \to [P'_{-2},P'_{-1}].             
  \end{align*}
  Since $\Theta$ suffices for mass computations, we omit further details.
\end{remark}
\begin{figure}[ht]
  \centering
  \begin{tikzpicture}[->, >=stealth', semithick, font=\scriptsize]
    \tikzstyle{every state}=[font=\scriptsize,fill=black!10!white, draw=black!20!white, rectangle]
    \tikzstyle{support}=[font=\scriptsize,fill=black!10!white, draw=black!20!white, rectangle, inner sep=0.5em]    
    \tikzstyle{label}=[anchor=center,sloped, fill=white, font=\scriptsize]
    \node[support] (m10) {$[P_{-1},P_0]$};
    \node[support] (m21) [right =1cm of m10] {$[P_{-2},P_{-1}]$};
    \node (d1) [right = 1cm of m21] {\large$\cdots$};
    \node (d2) [right = 1cm of d1] {\large$\cdots$};    
    \node[support] (23) [right = 1cm of d2] {$[P_{2},P_{3}]$};    
    \node[support] (12) [right = 1cm of 23] {$[P_1,P_2]$};
    \node[support] (m10') [below = 3cm of m10] {$[P'_{-1},P'_0]$};
    \node[support] (m21') [right =1cm of m10'] {$[P'_{-2},P'_{-1}]$};
    \node (d1') [right = 1cm of m21'] {\large$\cdots$};
    \node (d2') [right = 1cm of d1'] {\large$\cdots$};    
    \node[support] (23') [right = 1cm of d2'] {$[P'_{2},P'_{3}]$};    
    \node[support] (12') [right = 1cm of 23'] {$[P'_1,P'_2]$};
    \path[->]
    (12) edge[loop right] node {$\sigma_1$} ()
    (23') edge node[label] {$\sigma_0$} (12')
    (m21') edge[bend right=20] node[label] {$\sigma_0$} (12')
    (m10') edge[bend right] node[label] {$\sigma_0$} (12')
    (m21) edge node[pos=0.2, label] {$\sigma_0$} (m10')
    (12) edge node[pos=0.2, label] {$\sigma_0$} (23')
    (23) edge node[pos=0.2, label] {$\sigma_0$} (d2')
    (12') edge[loop right] node {$\sigma_0$} ()
    (23) edge node[label] {$\sigma_1$} (12)
    (m21) edge[bend left=20] node[label] {$\sigma_1$} (12)
    (m10) edge[bend left] node[label] {$\sigma_1$} (12)
    (m21') edge node[pos=0.2, label] {$\sigma_1$} (m10)
    (12') edge node[pos=0.2, label] {$\sigma_1$} (23)
    (23') edge node[pos=0.2, label] {$\sigma_1$} (d2)        
    ;
  \end{tikzpicture}
  \caption{An HN automaton for an on-the-wall standard stability condition, depicting outgoing arrows labelled \(\sigma_0\) and \(\sigma_1\).}
  \label{fig:aonehat-degenerate-tau-automaton}
\end{figure}

\subsection{Consequences of the automaton}
In this section, we use the automaton from~\Cref{subsec:aonehat-automaton} to prove results about our projective embedding of \(\Stab(\calC)/\mathbf{C}\).
We have the following analogue of~\Cref{prop:non-collapsing-triangles-non-deg}.
\begin{proposition}\label{prop:non-collapsing-aonehat}
  Consider the following distinguished triangles:
  \[P_1[-1] \to P_0 \oplus P_0 \to P_{-1} \xrightarrow{+1} \text{ and } P_3[-1] \to P_2 \oplus P_2 \to P_1\xrightarrow{+1}.\]
  If \(\tau\) is a standard stability condition of type I, then neither of the above triangles is collapsed by \(\tau\) in the sense of~\Cref{def:collapsed-triangle}.
  Furthermore, if \(\tau'\) is any stability condition such that neither of the above triangles is collapsed by \(\tau'\), then \(\tau'\) is standard of type I up to the action of \(\mathbf{C}\).
\end{proposition}
\begin{proof}
  It is immediate from~\Cref{prop:ss-geodesic} that neither of the above triangles is collapsed by any standard stability condition of type I.
  Now let \(\tau'\) be any stability condition.
  Recall that \(\tau'\) is in the braid group orbit of a standard stability condition of type I.
  Write \(\tau' = \beta \tau''\) for some braid \(\beta\) and some standard stability condition \(\tau''\).
  As in the proof of~\Cref{prop:non-collapsing-triangles-non-deg}, we may assume that \(\tau''\) is off-the-wall; otherwise the result follows by continuity.
  
  If \(\beta\) is a power of \(\gamma\), then \(\tau'\) is already standard of type I up to the action of \(\mathbf{C}\).
  Otherwise, we show that one of the two triangles above is collapsed by \(\tau'\).
  Equivalently, we show that \(\beta^{-1}\) applied to one of the two triangles above is collapsed by \(\tau''\).

  Consider a recognised expression of \(\beta^{-1}\) as described in~\Cref{prop:cyclicwriting-a1hat}.
  Since \(\beta\) is not a power of \(\gamma\), this expression ends in \(\sigma_i\) for some \(i \in \mathbf{Z}\).
  We observe using the automaton in~\Cref{fig:aonehat-automaton}, analogous to the argument in the proof of~\Cref{prop:non-collapsing-triangles-non-deg} that the first triangle
  is collapsed by \(\tau'\) if \(i \neq 0, 1\).
  Otherwise, the second triangle is collapsed by \(\tau'\).
  We omit the details.
\end{proof}

\begin{proposition}\label{prop:homeo2}
  The map \(m \from \Stab(\calC)/\mathbf{C} \to \mathbf{P}^\bfS\) is a homeomorphism onto its image.
\end{proposition}
\begin{proof}
  The proof is analogous to the proof of~\Cref{prop:homeo1}, using the result of~\Cref{prop:non-collapsing-aonehat}.
\end{proof}

Recall that we set \(M\) to be the \(\Theta\)-representation that associates the free abelian group \(\mathbf{Z}^{\{P_k,P_{k+1}\}}\) to the state \([P_k,P_{k+1}]\).
As in the $A_2$-case, the \(\Theta\)-set $M/\pm$ isomorphic to \(\mathbf{Z}^2/\pm\) with the standard action of $\PSL_2(\mathbf{Z})$, as follows.
Corresponding to each state \([k,k+1]\) of \(\Theta\), we define a linear map \(\phi_k \colon M_{[k,k+1]} \to \mathbf{Z}^2\) by
  \[
    \phi_{k} \from (1,0) \mapsto (k,1), \quad (0,1) \mapsto (k+1, 1).
  \]
  \begin{proposition}\label{prop:aonehat-M-is-standard}
    The maps \(\phi_k\) give an isomorphism of \(\Theta\)-sets \(M/\pm \to \mathbf{Z}^2/\pm\).
\end{proposition}
\begin{proof}
  Let us check that $\phi$ is $\Theta$-equivariant.
  It is evident that $\phi_k$ intertwines the edges labelled $\gamma^\pm$.
  To check that it also intertwines $\sigma_j \from k \to j$, we must check that
  \[
    \sigma_j \circ \phi_k
    =  \phi_j \circ M(e).
  \]
  Recalling the matrix of $\sigma_{P_j}$ from~\eqref{eqn:aonehat-pgl2} and $M(e)$ from~\eqref{eqn:aonehat-me}, we must check
  \[
    \begin{pmatrix}
      2j+1 & -2j^2\\
      2 & -2j+1
    \end{pmatrix}
    \begin{pmatrix}
      k & k+1 \\
      1 & 1
    \end{pmatrix}
    = \pm
    \begin{pmatrix}
      j-k-1 & j-k-2 \\
      j-k & j-k-1
    \end{pmatrix}
    \begin{pmatrix}
      j & j+1 \\
      1 & 1
    \end{pmatrix}.
  \]
  This is straightforward.
\end{proof}

Recall that we have a $B_{\Gamma}$-equivariant map $i \from \bfS \to \mathbf{P}^1(\mathbf{Z})$.
The following proposition identifies \(i(s)\) explicitly in terms of the HN multiplicities.
\begin{proposition}
  For a spherical object \(s\) supported at the \(k\)th state, we have the equality
  \[i(s) = \phi_k(\HN_{\tau}(s)) \in \mathbf{P}^1(\mathbf{Z}).\]
  Moreover, \(\phi_k(\HN_{\tau}(s)) \in \mathbf{Z}^2/\pm\) is the unique representative of \(i(s)\) whose coordinates are relatively prime.
\end{proposition}
\begin{proof}
  Analogous to the proof of~\Cref{prop:iishn}
\end{proof}
The division of \(\bfS\) according to the HN support corresponds nicely to a geometric division of the circle \(\mathbf{P}^1(\mathbf{R})\), which we now describe.
The objects \(P_i\) divide \(\mathbf{P}^1(\mathbf{R})\) into a collection of closed arcs, namely \(\{[P_i,P_{i+1}]\mid i \in \mathbf{Z}\}\).
Together with \(P_\infty\), these arcs cover all of \(\mathbf{P}^1(\mathbf{R})\).
We show that this decomposition of the circle is compatible with the automaton.
\begin{proposition}\label{prop:aonehat-hn-support}
  For each \(i \in \mathbf{Z}\), the objects supported at state $i$ are mapped to the arc \([P_i,P_{i+1}]\).
\end{proposition}
\begin{proof}
Analogous to the proof of~\Cref{prop:sphericalHN}.
\end{proof}

As a consequence of the automaton, we also obtain a Rouquier--Zimmermann--type theorem for $\widehat A_1$. Just as in type $A_2$, this theorem relates the parametrization of rational points $[a:c] \in \mathbf{P}^1(\mathbf{Z})$ with Jordan--H\"older multiplicities of the corresponding spherical object.

Recall the notion of Jordan--H\"older (or JH) multiplicity from~\Cref{def:JH-multiplicity}.
We use the same definition for the \(\widehat A_1\) category, with respect to the standard heart.
\begin{proposition}\label{prop:ac-aonehat}
  Let $x \in \bfS$ such that $i(x) = [a:c] \in \mathbf{P}^1(\mathbf{Z})$, with $a$ and $c$ relatively prime integers.
  Then the JH multiplicities of \(P_0\) and \(P_1\) in $x$ are exactly $|a-c|$ and $|a|$ respectively.
\end{proposition}
\begin{proof}
  Analogous to the proof of~\Cref{prop:ac}, using the result of~\Cref{prop:aonehat-M-is-standard} in lieu of~\Cref{prop:M-is-standard}.
\end{proof}

For future use (see~\Cref{prop:aonehat-gc-linearity}), we also prove the following statement, which computes values of the $\homBar$ function of a spherical object from the JH multiplicities of \(P_0\) and \(P_1\) in that object.
\begin{proposition}\label{prop:achoms-aonehat}
  Let $x \in \bfS$ such that the JH multiplicities of \(P_0\) and \(P_1\) in \(x\) are $n_0$ and \(n_1\) respectively.
  Then $\homBar(x,P_1) = 2n_0$ and $\homBar(x,P_0) = 2n_1$.
\end{proposition}
\begin{proof}
  We prove the assertion for $\homBar(x,P_1)$.
  The proof for $\homBar(x,P_0)$ is similar.
  We can compute directly from the definitions that $\homBar(P_1,P_k) = 2|k-1|$ for each \(k \in \mathbf{Z}\).  
  
  Fix a standard off-the-wall stability condition of type I.
  The key idea, as in the proof of~\Cref{prop:homs}, is to prove that \(\homBar(-,P_1)\) is a linear function of the HN multiplicities.
  Then the cases \(x = P_i\) imply the result for all \(x\).

  Suppose that $x$ is supported at $[P_k,P_{k+1}]$ for $k \neq 0,1$.
  Consider the Harder--Narasimhan filtration of $x$:
  \[ 0 = x_0 \to x_1 \to \cdots \to x_n = x,\]
  with factors $z_i = \Cone(x_i \to x_{i+1})$.
  Each $z_i$ is either $P_k$ or $P_{k+1}$, up to shift.

  By applying $\Hom(P_1,-)$, we obtain the following filtration in the bounded derived category of graded vector spaces:
  \begin{equation}\label{eqn:aonehat-hom-filtration}
    0 = \Hom(P_1,x_0) \to \cdots \to \Hom(P_1,x_n) = \Hom(P_1,x).
  \end{equation}
  As in the first case in the proof of~\Cref{prop:homs}, we obtain
  \[\Hom(P_1,x) = \bigoplus \Hom(P_1,z_i).\]

  We now treat the cases where $x$ is supported either at $[P_0,P_1]$ or at $[P_1,P_2]$.
  By~\Cref{prop:ac-aonehat}, the JH multiplicities of \(P_0\) in the objects $x$ and $\sigma_1^rx$ are equal, and both also have the same value for $\homBar(-,P_1)$.
  Hence the proposition for $x$ implies the proposition for $\sigma_1^rx$.
  By~\Cref{prop:unsink-aonehat}, any \(x\) supported either at \([P_0,P_1]\) or \([P_1,P_2]\) can be rewritten as \(x = \sigma_1^ry\) for some \(r \in \mathbf{Z}\) and \(y\) supported at \([P_k,P_{k+1}]\) for \(k \neq 0,1\).
  Hence the proposition holds for \(x\).
\end{proof}

\subsection{Gromov coordinates}
\begin{definition}[Gromov coordinates]\label{def:gromov-coordinates}
  Let $\tau$ be a stability condition of type I.
  For each integer $i \in \mathbf{Z}$, let $x_i(\tau)$ be the rational number defined as follows:
  \[ x_i(\tau) = \frac{m_\tau(P_{i-1})+m_\tau(P_{i+1}) - 2 m_\tau(P_{i})}{2}.\]
  We call the numbers $(x_i(\tau))$ the \emph{Gromov coordinates} of $\tau$.
\end{definition}

Using the triangle inequality (\Cref{prop:triangle}) on the triangles~\eqref{eq:aonehat-twist-triangles},
we deduce that all Gromov coordinates of $\tau$ are non-negative.
If $\tau$ is off-the-wall, they are all positive.

\begin{proposition}[Linearity]\label{prop:aonehat-gc-linearity}
  Let $\tau$ be a stability condition of type I with Gromov coordinates $(x_i)$.
  For any object $X \in \bfS$, we have
  \begin{equation}\label{eq:aonehat-gc-linearity}
    m_\tau(X) = \frac{1}{2}\sum_{i \in \mathbf{Z}} x_i \homBar(P_{i},X).
  \end{equation}
\end{proposition}
\begin{proof}
  Recall that for any $i \in \mathbf{Z}$, we have
  \[ P_{2i} = \gamma^iP_{0},\quad P_{2i+1} = \gamma^iP_{1}.\]
  Therefore, we get
  \[
    \homBar(P_{2i}, X) = \homBar(P_0, \gamma^{-i}X), \quad \homBar(P_{2i+1}, X) = \homBar(P_1, \gamma^{-i}X).
  \]
  \Cref{prop:achoms-aonehat} allows us to compute $\homBar(P_0, \gamma^{-i}X)$ and $\homBar(P_1,\gamma^{-i}X)$ from the HN multiplicities of $\gamma^{-i}X$.
  In turn, the automaton in~\Cref{fig:aonehat-automaton} allows us to compute the HN multiplicities of $\gamma^{-i}X$ from the HN multiplicities of $X$.
  Suppose $X$ is supported at the $k$th state of the automaton.
  Suppose the HN filtration of $X$ consists of $\alpha$ copies of $P_{k}$ and $\beta$ copies of $P_{k+1}$, up to shift.
  Then $\gamma^{-i}X$ is supported on the $(k-2i)$th state, and its HN filtration contains $\alpha$ copies of $P_{k-2i}$ and $\beta$ copies of $P_{k+1-2i}$.
  In particular, for any $i \in \mathbf{Z}$, the JH multiplicities of $P_1$ and $P_0$ in $\gamma^{-i}X$ are exactly
  \[ \left(\alpha|k-2i| + \beta|k-2i+1|\right) \text{ and }\left(\alpha|k-2i-1| + \beta|k-2i|\right)\]
  respectively.
Using~\Cref{prop:achoms-aonehat}, we deduce that for any $j \in \mathbf{Z}$, we have
  \begin{equation}\label{eq:homBar-aonehat}
    \homBar(P_{j},X) = 2\left(\alpha|k-j| + \beta|k-j+1|\right).
  \end{equation}
  Substitute
  \[x_i = \frac{m_\tau(P_{i-1})+m_\tau(P_{i+1}) - 2 m_\tau(P_{i})}{2}\]
  in  \[\frac{1}{2}\sum_{i \in \mathbf{Z}} x_i \homBar(P_{i},X).\]
  Then the coefficient of $m_\tau(P_{j})$ is exactly
  \[ \frac{\homBar(P_{j-1},X) + \homBar(P_{j+1},X) - 2 \homBar(P_{j},X)}{4}.\]
  Using~\eqref{eq:homBar-aonehat}, we see that the expression above equals $\alpha$ if $j = k$, is $\beta$ if $j = k+1$, and is $0$ otherwise.
  Therefore we see that
  \[\frac{1}{2}\sum_{i\in \mathbf{Z}}x_i\homBar(P_{i},X) = \alpha \cdot m_\tau(P_{k}) + \beta\cdot  m_\tau(P_{k+1}),\]
  which is precisely the $\tau$-mass of $X$.  
\end{proof}

Recall that \(\Lambda\) is the closure in \(\Stab(\mathcal{C})/\mathbf{C}\) of the image of the standard off-the-wall type I stability conditions.
We now study the topology of $\Lambda$, and compute the closure of $m(\Lambda)$ in $\mathbf{P}^\bfS$.
\begin{proposition}\label{prop:lambda-delta}
  Let \(\overline{\mathbf{H}}\) be the closed upper half plane.
  Consider the subset
  \[L = \{-1\} \cup \left\{ \frac{1-n}{n} \mid n \in \mathbf{Z} \setminus {0} \right\}.\]
  Then the map
  \[\tau \mapsto \frac{Z_{\tau}(P_1)}{Z_{\tau}(P_0)}\]
  is a homeomorphism of \(\Lambda\) onto \(\overline{\mathbf{H}} \setminus L\).
\end{proposition}
\begin{proof}
  The map is clearly continuous.
  We first show that its image lies in \(\overline{\mathbf{H}} \setminus L\).
  Recall that $\Lambda$ is the closure of off-the-wall stability conditions of type I (up to the \(\mathbf{C}\)-action).
  These are stability conditions in which \(P_0\) and \(P_1\) are stable and satisfy the phase inequalities
  \[ \phi(P_0) < \phi(P_1) < \phi(P_0)+1.\]
  It follows that for any point \(\tau \in \Lambda\) the quantity $\omega_{\tau} = Z_{\tau}(P_1)/Z_{\tau}(P_0)$ lies in the closed upper half plane.
  To see why $\omega_{\tau}$ avoids $L$, note that the objects \(P_{n}\) are \(\tau\)-semi-stable for each \(n \in \mathbf{Z}\) as well as \(n = \infty\).
  Therefore, we must have $Z_{\tau}(P_n) \neq 0$ for $n \in \mathbf{Z} \cup \{\infty\}$.
  Since \[[P_n] = |n| [P_1] + |n-1|[P_0] \text{ and } [P_\infty] = [P_0]+[P_1],\] it is clear that $\omega_{\tau} \notin L$.

  Now fix some \(\omega \in \overline{\mathbf{H}} \setminus L\).
  We now prove that there is a unique $\tau \in \Lambda$ with $\omega = Z_\tau(P_1)/Z_{\tau}(P_0)$.
  Consider the central charge \(Z \colon K(\mathcal{C}) \to \mathbf{C}\) defined as
  \[Z(P_0) = 1, \quad Z(P_1) = \omega.\]
  Recall from~\cite[Proposition 5.3]{bri:07} that a stability condition is uniquely specified by choosing:
  \begin{enumerate}[(a)]
  \item the heart of a bounded t-structure; and
  \item a central charge that sends the heart to the semi-closed upper half plane \(\mathbf{H} \cup \mathbf{R}_{>0}\) with the Harder--Narasimhan property.
  \end{enumerate}
  The standard heart on \(\mathcal{C}\) is the extension closure of \(P_0\) and \(P_1\).
  Being a finite-length category, any central charge as above with respect to the standard heart automatically enjoys the Harder--Narasimhan property.
  Thus if \(\omega\) lies in the semi-closed upper half plane, the central charge above uniquely specifies a type I standard stability condition.

  We now treat the remaining cases; namely those where \(\omega\) lies on the negative real axis.
  In this case we can no longer use the standard heart for our reconstruction, because the central charge function does not land in the semi-closed upper half plane.
  We will show nevertheless that we can uniquely reconstruct points of \(\Lambda\) in these cases using tilts of the standard heart.
  These tilts will also be finite-length categories for which the Harder--Narasimhan property is automatic.
  
  The set $\mathbf{R} \setminus L$ is a disjoint union of open intervals of two kinds:
  \[ \mathbf{R} \setminus L = \bigcup_{n \geq 0}\left(-\frac{n}{n+1}, - \frac{n-1}{n}\right) \cup \bigcup_{n \geq 0}\left(- \frac{n+1}{n}, - \frac{n+2}{n+1}\right).\]
  The tilt we choose will depend on the interval that contains $\omega$.
  Regard \((0,\infty)\) as an interval of the first kind at the value \(n = 0\); we have already treated this case.
  
  Suppose $\omega$ lies in $(-1/2, 0)$, which is the interval of the first kind at the value $n = 1$.
  Consider the central charge $Z' = Z \circ \sigma_1$.
  Recall that $\sigma_1^{-1}$ takes stability conditions of type I to those (possibly in the closure) of type II.
  Also note that $\omega' = Z'(P_0)/Z'(P_1)$ lies in the interval $(0, +\infty)$.
  The previous argument applied to type II conditions shows that there is a unique $\tau'$ of type II with central charge $Z'$.
  Then $\tau = \sigma_1 \tau'$ gives the unique point of \(\Lambda\) with central charge $Z$.

  Now suppose $\omega = Z(P_1)/Z(P_0)$ lies in $\left( - \frac{n}{n+1}, - \frac{n-1}{n} \right)$ for some $n \geq 2$.
  We construct the corresponding stability condition using the action of the element $\gamma = \sigma_1 \sigma_0$, which tilts the standard heart.
  Let $m = \lfloor  n/2 \rfloor$ and set $Z' = Z \circ \gamma^m$.
  Then $\omega' = Z'(P_1)/Z'(P_0)$ lies in one of the two intervals $(-1/2,0)$ or $(0, +\infty)$, depending on the parity of $n$.
  In either case, the previous argument shows that there is a unique stability condition $\tau'$ of type I with central charge $Z'$.
  Since $\gamma$ preserves \(\Lambda\), the stability condition $\tau = \gamma^{m} \tau'$ gives the unique point of \(\Lambda\) with central charge $Z$.

  For $\omega = Z(P_1)/Z(P_0)$ in the intervals of the second kind, we use a similar argument with $Z' = Z \circ \gamma^{-m}$.
  We omit the details.

  We have shown that the map \(\tau \mapsto \omega_{\tau}\) is a bijection from \(\Lambda\) to \(\overline{\mathbf{H}} \setminus L\).
  It is evident from the proof that the inverse map is continuous; therefore the map \(\tau \mapsto \omega_{\tau}\) is a homeomorphism from \(\Lambda\) to \(\overline{\mathbf{H}} \setminus L\).
\end{proof}

It will be convenient in the remainder of this section to consider a change of coordinates from the previous proposition.
Let $\overline{\Delta}$ be a closed disk, considered as the one-point compactification of the closed upper half plane.
Consider the map
\[ \delta\from \Lambda \to \overline{\Delta}\]
defined as 
\begin{equation}\label{eq:delta-map}
  \delta_\tau = \frac{Z_{\tau}(P_0)}{Z_{\tau}(P_0)+Z_{\tau}(P_1)}.
\end{equation}
By changing coordinates from~\Cref{prop:lambda-delta}, we see that \(\delta\) is injective and a homeomorphism onto its image, and also that the image of \(\delta\) is precisely
\[
  \overline{\Delta} \setminus (\mathbf{Z} \cup \{\infty\}).
\]
\Cref{fig:aonehat-lambda} shows a sketch of $\Lambda$ as a subset of the closed disk $\overline{\Delta}$.
\begin{figure}
  \begin{tikzpicture}
    \tikzstyle{obj}=[fill=white, draw, circle, inner sep=1.5pt, outer sep=0pt]
    \node[obj] (infty) at (-90:2) {} (infty) node [below] {$0$};
    \node[obj] (n21) at (-135:2) {} (n21) node [below left] {$-1$};
    \node[obj] (n32) at (-180:2) {} (n32) node [below left] {$-2$};
    \node[obj] (n43) at (-220:2) {} (n43) node [left] {$-3$};
    \node[obj] (n54) at (-240:2) {} (n54) node [above left] {$-4$};
    \node[obj, inner sep=1.2pt] at (-250:2) {} (-250:2.2) node {$\cdot$};
    \node[obj, inner sep=0.9pt] at (-255:2) {} (-255:2.2) node {$\cdot$};
    \node[obj, inner sep=0.6pt] at (-260:2) {} (-260:2.2) node {$\cdot$};
    \node[obj, inner sep=0.3pt] at (-263:2) {};
    \node[obj, inner sep=0.2pt] at (-266:2) {};
    \node[obj, inner sep=0.1pt] at (-267:2) {};

    \node[obj] (n01) at (-45:2) {} (n01) node [below right] {$1$};
    \node[obj] (n12) at (-0:2) {} (n12) node [below right] {$2$};
    \node[obj] (n23) at (40:2) {} (n23) node [right] {$3$};
    \node[obj] (n34) at (60:2) {} (n34) node [above right] {$4$};
    \node[obj, inner sep=1.2pt] at (70:2) {} (70:2.2) node {$\cdot$};
    \node[obj, inner sep=0.9pt] at (75:2) {} (75:2.2) node {$\cdot$};
    \node[obj, inner sep=0.6pt] at (80:2) {} (80:2.2) node {$\cdot$};
    \node[obj, inner sep=0.3pt] at (83:2) {};
    \node[obj, inner sep=0.2pt] at (86:2) {};
    \node[obj, inner sep=0.1pt] at (87:2) {};

    \node[obj] (n1) at (-270:2) {} (n1) node [above] {$\infty$};

    \begin{scope}[on background layer]
      \draw[fill=white!90!black] (0,0) circle (2);
    \end{scope}
  \end{tikzpicture}
  \caption{The set $\Lambda$ of stability conditions of type I for the CY2 category associated to $\widehat A_1$. The stability conditions are labelled by the ratio $\frac{Z(P_0)}{Z(P_0)+Z(P_1)}$. The interior of the disk represents the upper half plane.
  }
  \label{fig:aonehat-lambda}
\end{figure}

\subsection{The closure}

We use the embedding $\Lambda \stackrel{\delta}{\hookrightarrow} \overline{\Delta}$ to prove the following proposition.
\begin{proposition}[Closure of $\Lambda$]\label{prop:aonehat-lambda-closure}
  The closure of $m(\Lambda)$ in $\mathbf{P}(\mathbf{R}^\bfS)$ is precisely
  \[ \overline{m(\Lambda)} = m(\Lambda) \cup \{\homBar(P_{i}) \mid i \in \mathbf{Z}\}\cup \{\homBar(P_1)- \homBar(P_0)\}.\]
  This closure is homeomorphic to $\overline{\Delta}$.
\end{proposition}
\begin{proof}
  Via the identification \(\Lambda \cong \overline{\Delta}\setminus (\mathbf{Z} \cup \{\infty\})\) via the map \(\delta\) from~\eqref{eq:delta-map}, the set \(\overline{\Delta}\setminus (\mathbf{Z} \cup \{\infty\})\) maps to \(\mathbf{P}(\mathbf{R}^\bfS)\) by \(m \circ \delta^{-1}\).
  We now show that this map extends continuously to \(\overline{\Delta}\).
  By~\Cref{prop:aonehat-gc-linearity}, we can express the mass functional \(m_\tau\) of \(\tau \in \Lambda\) as an infinite linear combination of the Gromov coordinates of \(\tau\).
  We now compute the limit of each Gromov coordinate of \(\tau\) as \(\delta_\tau\) approaches one of the points in \(\mathbf{Z} \cup \{\infty\}\), and thereby check that the map \(m_\tau\) extends continuously.

  Consider some \(\tau \in \Lambda\) with central charge \(Z_{\tau}\); let us compute the limit of the Gromov coordinates as \(\delta_{\tau}\) approaches some point of \(\mathbf{Z} \cup \{\infty\}\).
  We know that up to a complex scalar, we have
  \[Z_\tau(P_0) = 1,\quad Z_\tau(P_1) = \frac{1 - \delta_\tau}{\delta_\tau}.\]
  We see that
  \[Z_\tau(P_{i}) = \frac{|i|(1-\delta_\tau) + |i-1|\delta_\tau}{\delta_\tau}.\]
  First suppose that $\delta_\tau$ approaches some $n \in \mathbf{Z}$.
  In this case,
  \[ Z_\tau(P_{i}) \to \frac{|i|(1-n) + |i-1|n}{n}.\]
  We can compute from~\Cref{def:gromov-coordinates} that the limit of the \(j\)th Gromov coordinate is
  \[ x_j(\tau) \to \begin{cases}0,&j \neq n,\\1,&j = n\end{cases}.\]
  The sum in~\eqref{eq:aonehat-gc-linearity} obviously converges in the limit.
  Up to a simultaneous scalar, the limit of the mass functional is
  \[m_\tau(-) \to \homBar(P_n,-).\]
  
  Now suppose that $\delta_\tau$ approaches $\infty$.
  In this case,
  \[ Z_\tau(P_{j}) \to -|i| + |i-1| = \begin{cases}-1,&i > 0,\\1,&i \leq 0\end{cases}.\]
  We compute that
  \[ x_j(\tau) \to \begin{cases}0,&i < 0\text{ or }i > 1,\\-1,&i = 0,\\1,&i = 1.\end{cases}\]
  Again, the sum in~\eqref{eq:aonehat-gc-linearity} obviously converges in the limit.
  Up to a simultaneous scalar, the limit of the mass functional is
  \[m_\tau(-) \to \homBar(P_1,-) - \homBar(P_0,-).\]
  
  We have exhibited a factoring
  \[
    \begin{tikzcd}
      \Lambda \arrow[hook]{rd}[swap]{\delta} \arrow{rr}{m}&& \mathbf{P}(\mathbf{R}^\bfS)\\
      &\overline{\Delta} \arrow{ru}{d}&
    \end{tikzcd}
  \]
  and identified the additional points in the image of \(\overline{\Delta}\).
  Since $\overline{\Delta}$ is compact, no other points are in the closure of \(m(\Lambda)\).
\end{proof}

\begin{remark}
  The point $\infty$ is already a limit point of the subset $\mathbf{Z} \cup \{\infty\} \subset \overline{\Delta}$.
  Therefore the additional point in the closure, namely the functional $\homBar(P_1)- \homBar(P_0)$, is already a limit point of the set $\{\homBar(P_i)\mid i \in \mathbf{Z}\}$.
\end{remark}

It is now easy to identify the closure of $\Stab(\calC)/\mathbf{C}$, using the following contraction property, analogous to~\Cref{prop:contracting}.
Fix a metric on $\mathbf{P}^1(\mathbf{R}) = S^1$.
\begin{proposition}\label{prop:A1hat-contracting}
  Let $\epsilon > 0$ and let $I \subset \mathbf{P}^1(\mathbf{R})$ be a compact subset.
  Then for all but finitely many elements $g \in B_{\Gamma}$, the image $g(I)$ has diameter less than $\epsilon$.
\end{proposition}
\begin{proof}
  The action of $B_{\Gamma}$ on $\mathbf{P}^1(\mathbf{R})$ is via the map $B_{\Gamma} \to \PSL_2(\mathbf{Z})$ of~\Cref{eq:bgamma-psl2-aonehat}.
  The claim is true for $\PSL_2(\mathbf{Z})$, as shown in the proof of~\Cref{prop:contracting}, and hence also for $B_{\Gamma}$.
\end{proof}

Let $M$ and $P$ be the images of $\Stab(\calC)/\mathbf{C}$ and $\bfS$ in $\mathbf{P}^{\bfS}$, respectively.
\begin{proposition}[Closure of $M$]
  The sets $M$ and $P$ are disjoint, and their union is the closure of $M$.
\end{proposition}
\begin{proof}
  The proof is analogous to the proof of~\Cref{prop:b-closure}.
\end{proof}

We conjecture that the union $M \cup P$ is homeomorphic to a closed disk, as in the $A_2$ case.

\appendix
\section{Rectifiable filtrations}\label{sec:rectifiable-filtrations}
In this section, we formulate sufficient conditions on a filtration to ensure that it can be re-arranged to the Harder--Narasimhan (HN) filtration.

Fix a triangulated category \(\mathcal{C}\) and stability condition on \(\mathcal{C}\).
For a semistable object \(x\), let \(\phi(x)\) denote its phase.
For any object \(x\), denote by $\lfloor x \rfloor$ (resp $\lceil x \rceil$) the semistable factor of lowest (resp. highest) phase in its HN filtration.
We retain the subscript notation for the various truncations induced by the stability condition from \Cref{subsec:stability-conditions}.

The first basic question is as follows.
Consider a distinguished triangle
\[ x \to y \to z \to x[1].\]
When can we easily obtain the HN filtration of \(y\) from that of \(x\) and \(z\)?
The following property of the map $z \to x[1]$ leads to the answer.
\begin{definition}[Rectifiable map]
  \label{def:rectifiable-map}
  We say that a map $z \to x[1]$  is \emph{rectifiable} if for every $\alpha \in \R$, the induced map \(z_{> \alpha} \to x_{\leq \alpha}[1]\) obtained as the composition
  \[
    \begin{tikzcd}
      z_{> \alpha} \arrow{r} & z\arrow{d} &\\
      & x[1] \arrow{r} &x_{\leq \alpha}[1]
    \end{tikzcd}
  \]
  vanishes.
  Furthermore, we say that a triangle \(x \to y \to z \xrightarrow{+1}\) is rectifiable if the connecting map \(z \to x[1]\) is rectifiable.
\end{definition}
\begin{remark}
  A given object has only finitely many HN pieces, and hence only finitely many distinct truncations. 
  Therefore, checking the rectifiability of a given map involves only finitely many vanishing conditions.
  In particular, \(z \to x[1]\) is rectifiable if and only if for all \(\alpha\), the induced map
  \[ z_{\geq \alpha} \to x_{<\alpha}[1]\]
  vanishes.
\end{remark}

\begin{example}\label{ex:easyrect}
  \begin{enumerate}
  \item   The zero map $z \to x[1]$ is rectifiable.
  \item   If both $z \to x_1[1]$ and $z \to x_2[1]$ are rectifiable, then so is the direct sum $z \to (x_1 \oplus x_2)[1]$.
  \item \label{i:dsum} Analogously, if \(z_1 \to x[1]\) and \(z_2 \to x[1]\) are rectifiable, then so is the direct sum \((z_1 \oplus z_2) \to x[1]\).
  \item \label{i:highlow}
    Suppose $\phi(\lfloor  x \rfloor) \geq \phi(\lceil  z \rceil)$.
    Then any map $z \to x[1]$ is rectifiable.
    Indeed, either $z_{> \alpha}$ or $x_{\leq \alpha}$ is zero.
  \end{enumerate}
\end{example}

\begin{proposition}\label{prop:rectifable-factor}
  Suppose $z \to x[1]$ is rectifiable.
  Then for any maps $y \to z$ and $x \to w$, the induced map $y \to w[1]$ is rectifiable.
  That is, rectifiable maps form a two-sided ideal.
\end{proposition}
\begin{proof}
  The map $y_{> \alpha} \to w_{\leq \alpha}[1]$ factors as
  \[ y_{> \alpha} \to z_{> \alpha} \to x_{\leq \alpha}[1] \to w_{\leq \alpha}[1].\]
  Since $z \to x[1]$ is rectifiable, the map $z_{> \alpha} \to x_{\leq \alpha}[1]$ vanishes.
  Hence, the map $y_{>\alpha} \to w_{\leq\alpha}[1]$ vanishes too.
\end{proof}

We will need the following simple observations.
\begin{lemma}\label{lem:unique-induced-map}
  Let \(x \to y \to z \xrightarrow{+1}\) be a distinguished triangle.
  \begin{enumerate}
  \item Let  \(f \colon a \to y\) be a map such that the composite \(a \xrightarrow{f} y \to z\) is zero.
    Further, suppose \(\Hom(a, z[-1]) = 0\).
    Then there is a unique map \(\widehat{f} \colon a \to x\) such that the composite \(a \xrightarrow{\widehat{f}} x \to y\) is \(f\).
  \item Dually, let \(f \colon y \to a\) be a map such that the composite \(x \to y \xrightarrow{f} a\) is zero.
    Further, suppose \(\Hom(x[1], a) = 0\).
    Then there is a unique map \(\widehat f \colon z \to a\) such that the composite \(y \to z \xrightarrow{\widehat f} a\) is \(f\).
  \end{enumerate}
\end{lemma}

\begin{proof}
  The existence of \(\widehat f\) follows from the axioms of triangulated categories.
  The difference of two such maps factors through a map \(x \to z[-1]\), which must be zero by assumption.

  The proof of the dual statement is similar.
\end{proof}

\begin{lemma}\label{lem:swap-terms}
  Consider a filtration
  \[
    \begin{tikzcd}[column sep=1em]
      x_1 \arrow{rr} && x_2 \arrow{rr} \arrow{ld} && x_3. \arrow{ld}\\
      &y_2 \arrow[dashed]{ul}{+1} && y_3\arrow[dashed]{ul}{+1} &
    \end{tikzcd}
  \]
  Suppose that the connecting map \(y_3 \to y_2[1]\) vanishes.
  Then we can flip \(y_2\) and \(y_3\).
  More precisely, for some object \(x_{2}'\), we have a filtration
  \[
    \begin{tikzcd}[column sep=1em]
      x_1 \arrow{rr} && x'_2 \arrow{rr} \arrow{ld} && x_3, \arrow{ld}\\
      &y_3 \arrow[dashed]{ul}{+1} && y_2\arrow[dashed]{ul}{+1} &
    \end{tikzcd}
  \]
  such that the connecting map \(y_2 \to y_3[1]\) also vanishes.
\end{lemma}
\begin{proof}
  Let \(z\) be such that \(x_1 \to x_3 \to z \xrightarrow{+1}\) is a distinguished triangle, where the first map is the composition \(x_1 \to x_2 \to x_3\) from the given diagram.
  We see by the octahedral axiom that \(y_2 \to z \to y_3\xrightarrow{+1}\) is also a distinguished triangle.
  The connecting map \(y_3 \to y_2[1]\) in this distinguished triangle is the same as the one in the given filtration, and is hence zero.
  Therefore, the distinguished triangle is split and we have \(z \cong y_2 \oplus y_3\).

  Now consider the composition of \(x_3 \to z\) with the projection \(z \to y_2\).
  Let \(x_2'\) be an object such that the triangle
  \[x_2' \to x_3 \to y_2 \xrightarrow{+1}\]
  is distinguished.
  Again, the octahedral axiom shows that
  \[x_1 \to x_2' \to y_3\xrightarrow{+1}\]
  is a distinguished triangle, and we obtain the desired filtration.
  It is easy to see that the connecting map \(y_2 \to y_3[1]\) is also zero, as it is part of the split distinguished triangle \(y_3 \to z \to y_2 \xrightarrow{+1}\).
\end{proof}
\begin{proposition}\label{prop:low-high}
  Suppose $e \from z \to x[1]$ is rectifiable.
  \begin{enumerate}
  \item There exist unique maps \(z_{\leq \alpha} \xrightarrow{\elealpha} x_{\leq \alpha}[1]\) and \(z_{> \alpha} \xrightarrow{\egalpha} x_{> \alpha}[1]\) fitting in a morphism of distinguished triangles as follows.
    \[
      \begin{tikzcd}
        z_{> \alpha} \arrow{r}{} \arrow{d}{\egalpha}& z \arrow{r}{} \arrow{d}{e} & z_{\leq \alpha}\arrow{d}{\elealpha}\arrow{r}{+1}& \mbox{}\\
        x_{>\alpha}[1] \arrow{r}{} & x[1] \arrow{r}{} & x_{\leq \alpha}[1] \arrow{r}{+1}& \mbox{}
      \end{tikzcd}
    \]
  \item The maps \(\elealpha\) and \(\egalpha\) are rectifiable.
  \item Complete \(e\) to a distinguished triangle
    \[y \to z \xrightarrow{e} x[1] \xrightarrow{+1}.\]
    Then we have distinguished triangles
    \begin{align*}
      y_{\leq \alpha} &\to z_{\leq \alpha} \xrightarrow{\elealpha} x_{\leq \alpha}[1] \xrightarrow{+1} \text{ and }\\
      y_{> \alpha} &\to z_{> \alpha} \xrightarrow{\egalpha} x_{> \alpha}[1] \xrightarrow{+1}.
    \end{align*}
  \end{enumerate}
\end{proposition}
\begin{proof}
  Let us prove each part in sequence.
  \begin{enumerate}
  \item Recall that we have the following distinguished triangles:
    \begin{equation}\label{eq:truncation-triangles}
      \begin{gathered}
        z_{> \alpha} \to z \to z_{\le \alpha} \xrightarrow{+1} \text{ and }\\
        x_{> \alpha}[1] \to x[1] \to x_{\le \alpha}[1]\xrightarrow{+1}.
      \end{gathered}
    \end{equation}
    The composition \(z_{> \alpha} \to z \xrightarrow{e} x[1] \to x_{\leq \alpha}[1]\) vanishes because \(e\) is rectifiable.
    Furthermore, there are no non-zero maps from \(z_{>\alpha}\) to \(x_{\leq \alpha}\).
    Therefore, by \Cref{lem:unique-induced-map}, we obtain a unique map
    \[\elealpha \colon z_{> \alpha} \to x_{> a}[1]\]
    that makes the following square commute.
   \[
    \begin{tikzcd}
      z_{> \alpha} \arrow{r}\arrow{d}{\egalpha} & z \arrow{d}{e}\\
      x_{> \alpha}[1] \arrow{r} & x[1]
    \end{tikzcd}
  \]
  The argument for \(\elealpha\) is similar.

\item
  Let us check that \(\egalpha\) is rectifiable.
  We must check that for every \(\beta\), the map
  \[ z_{>\max(\alpha,\beta)} \to x_{(\alpha, \beta]}[1]\]
  is zero.
  This is clear if \(\beta \leq \alpha\), so assume \(\beta > \alpha\).
  Then the map above becomes
  \[ z_{> \beta} \to x_{(\alpha,\beta]}[1].\]
  Consider the distinguished triangle
  \[
    x_{\le \alpha} \to x_{(\alpha,\beta]}[1] \to x_{\le \beta}[1] \to  \xrightarrow{+1}.
  \]
  The composite
  \[ z_{> \beta} \to x_{(\alpha,\beta]}[1] \to x_{\leq \beta}[1]\]
  is zero because \(e\) is rectifiable.
  Thus there is an induced map \(z_{> \beta} \to x_{\leq \alpha}\) making the following diagram commute.
  \[
    \begin{tikzcd}
      \mbox{}&z_{> \beta} \arrow[dashed]{dl} \arrow{d}\\
      x_{\leq \alpha} \arrow{r}& x_{(\alpha, \beta]}[1]
    \end{tikzcd}
  \]
  However, since \(\beta > \alpha\), any such induced map must be zero.
  Therefore the original map \(z_{> \beta} \to x_{(\alpha, \beta]}[1]\) is zero.

  By a similar argument, we see that \(\elealpha\) is also rectifiable.

\item   
  The distinguished triangle
  \[ x \to y \to z \xrightarrow{+1}\]
  combined with (appropriate shifts of) the distinguished triangles~\eqref{eq:truncation-triangles} gives us a filtration as follows:
  \[
    \begin{tikzcd}[column sep=0.1cm]
      0 \arrow{rr} && y_1 \arrow{dl} \arrow{rr} &&y_2 \arrow{dl} \arrow{rr} &&y_3 \arrow{dl} \arrow{rr} && y_4 = y \arrow{dl}.\\
      & x_{> \alpha} \arrow[dashed]{ul}{+1} && x_{\leq \alpha} \arrow[dashed]{ul}{+1} && z_{> \alpha} \arrow[dashed]{ul}{+1} && z_{\leq \alpha} \arrow[dashed]{ul}{+1}
    \end{tikzcd}
  \]
  Since $z \to x[1]$ is rectifiable, the connecting map $z_{> \alpha} \to x_{\leq \alpha}[1]$ vanishes.
  Using~\Cref{lem:swap-terms}, we obtain a new filtration
  \[
    \begin{tikzcd}[column sep=0.1cm]
      0 \arrow{rr} && y_1 \arrow{dl} \arrow{rr} &&y'_2 \arrow{dl} \arrow{rr} &&y_3 \arrow{dl} \arrow{rr} && y_4 = y \arrow{dl}.\\
      & x_{> \alpha} \arrow[dashed]{ul}{+1} && z_{> \alpha} \arrow[dashed]{ul}{+1} && x_{\leq \alpha} \arrow[dashed]{ul}{+1} && z_{\leq \alpha} \arrow[dashed]{ul}{+1}
    \end{tikzcd}
  \]
  It is clear from the proof of~\Cref{lem:swap-terms} and the uniqueness statement of~\Cref{prop:low-high} that the connecting maps \(z_{\leq \alpha} \to x_{\leq \alpha}[1]\) and \(z_{> \alpha} \to x_{> \alpha}[1]\) in the diagram above are \(\elealpha\) and \(\egalpha\) respectively.

  Let \(y_2''\) be an object such that \(y_2' \to y \to y_2''\xrightarrow{+1}\) is a distinguished triangle.
By the octahedral axiom, we have distinguished triangles
  \begin{align*}
    x_{>\alpha} & \to y_2' \to z_{> \alpha} \xrightarrow{\egalpha} x_{> \alpha} [1]\\
    x_{\leq \alpha}&\to y_2'' \to z_{\leq \alpha} \xrightarrow{\elealpha} x_{\leq \alpha}[1] \xrightarrow{+1}.
  \end{align*}
  We note because of the distinguished triangles above that
  \[
    (y_2')_{\le \alpha} = 0 \text{ and } (y_2')_{> \alpha} = y_2',
  \]
  and so $y_{> \alpha} = y_2'$.
  Similarly, $y_{\leq \alpha} = y_2''$.
  The proof is complete.
\end{enumerate}
\end{proof}
\begin{warning}
  The maps \(\elealpha\) and \(\egalpha\) constructed in the previous proof are not the same as the maps induced by the truncation functors \((-)_{\leq \alpha}\) and \((-)_{> \alpha}\).
  Indeed, the object \(x_{\leq \alpha}[1]\) is not the same as the object \((x[1])_{\leq \alpha}\), and similarly \(x_{> \alpha}[1]\) is not the same as \((x[1])_{> \alpha}\).
\end{warning}
 For every $\alpha \in \R$, recall that \(\mathcal C_{\alpha}\) is the abelian category of semi-stable objects of phase \(\alpha\).
\begin{proposition}\label{prop:rect-filt-hn}
  Let $x \to y \to z \xrightarrow{e} x[1]$ be a rectifiable triangle.
  For every \(\alpha \in \R\) we have an exact sequence
  \[ 0 \to x_{\alpha} \to y_{\alpha} \to z_{\alpha} \to 0\]
  in \(\mathcal C_{\alpha}\).
  Consequently, we have
  \[ m_q(y) = m_q(x) + m_q(z).\]
\end{proposition}
\begin{proof}
  Since \(e\) is rectifiable, we can apply~\Cref{prop:low-high} twice to obtain a distinguished triangle
  \[x_{\alpha} \to y_{\alpha} \to z_{\alpha} \xrightarrow{+1}.\]
  All three objects lie in the abelian category \(\mathcal{C}_{\alpha}\), which is an abelian subcategory of the \(\mathcal{C}_{[\alpha,\alpha+1)}\).  Since the latter is the heart of a \(t\)-structure, the distinguished triangle becomes a short exact sequence.

  Since \(Z\) is additive for short exact sequences, we get
  \[ Z(y_{\alpha}) = Z(x_{\alpha}) + Z(z_{\alpha}).\]
  But the three numbers all lie on the ray of argument \(\pi\alpha\), so we get
  \[ |Z(y_{\alpha})| = |Z(x_{\alpha})| + |Z(z_{\alpha})|.\]
  Multiplying throughout by \(q^{\alpha}\) gives
  \[ m_{q}(y_{\alpha}) = m_{q}(x_{\alpha}) + m_{q}(z_{\alpha}).\]
  There are only finitely many numbers \(\alpha\) such that one of the objects \(x_{\alpha}\), \(z_{\alpha}\), and \(y_{\alpha}\) is non-zero.
      Summing over all such \(\alpha\) yields \(m_q(y) = m_{q}(x) + m_q(z)\).
  \end{proof}

We now take up filtrations.
\begin{definition}[Rectifiable filtration]\label{def:rectifiable-filtration}
  Consider a filtration
  \[ x_0 \to x_1 \to \dots \to x_n.\]
  We inductively define what it means for such a filtration to be rectifiable.
  All filtrations for $n = 0$ and $1$ are rectifiable.
  For $n \geq 2$, the filtration is rectifiable if the  following hold.
  \begin{enumerate}
  \item The filtration $x_1 \to \dots \to x_n$ is rectifiable.
  \item Suppose \(x(i,j)\) completes \(x_i \to x_j\) to a distinguished triangle
    \[x_i \to x_j \to x(i,j)\xrightarrow{+1}.\]
    Then the connecting map \(x(1,n) \to x(0,1)[1]\) arising from 
    \[
      \begin{tikzcd}[column sep=1em]
        x_0\arrow{rr} && x_1 \arrow{dl} \arrow{rr}&& x_{n}, \arrow{dl}\\
        &x(0,1)\arrow[dashed]{ul}{+1}&&x(1,n)\arrow[dashed]{ul}{+1}&
      \end{tikzcd}
    \]
    is rectifiable.
  \end{enumerate}
\end{definition}
\begin{remark}
  In the previous definition, the objects \(x(i,j)\) are unique up to (possibly non-unique) isomorphism, so the rectifiability of the connecting map is independent of the particular choices.
  
  Henceforth in this section, if \(f \colon x \to z\) is a morphism, we will denote by \(\Cone(f)\) any object that fits into a distinguished triangle \(x \xrightarrow{f} z \to \Cone(f) \xrightarrow{+1}\).
  Although such an object is not canonically defined in general, it will be sufficient for the arguments that it is unique up to isomorphism.
\end{remark}
\begin{remark}\label{rem:translation-invariance}
  The definition of a rectifiable filtration is ``translation invariant'' in the following sense.
  Given a filtration \(x_0 \to x_1 \to \dots \to x_n\), set \(x_i' = \Cone(x_0 \to x_i)\) and consider the filtration \(0 \to x_1' \to \dots \to x_n'\).
  Then the first filtration is rectifiable if and only if the second one is.
  Indeed, the cones \(x(i,j)\) and the connecting maps between them in the first filtration are isomorphic to the corresponding objects in the second filtration.
\end{remark}
\begin{proposition}\label{rem:hn-rectifiable}
  Consider a filtration
  \[ x_0 \to \cdots \to x_n,\]
  and set \(a_i = \Cone(x_{i-1} \to x_i)\).
  If for all \(i\), the lowest HN phase of \(a_{i}\) is greater than or equal to the highest HN phase of \(a_{i+1}\), then the filtration is rectifiable.
  In particular, the HN filtration is rectifiable.
\end{proposition}
\begin{proof}
  We induct on \(n\).
  The base cases \(n = 0 , 1\) are evident.
  Let us use the notation of \Cref{def:rectifiable-filtration}.
  To go from \(n-1\) to \(n\), consider the map \(x(1,n) \to x(0,1) = a_1\).
  The hypothesis implies that the lowest HN phase of \(x(0,1)\) is greater than or equal to the highest HN phase of \(x(1,n)\), so the map is rectifiable by \Cref{ex:easyrect}.
\end{proof}

\begin{proposition}\label{prop:rectifiable-truncation}
  Suppose \(x_0 \to \dots \to x_n\) is rectifiable.
  Then for every $i,j$ with $i < j$, the filtration \(x_i \to \dots \to x_{j}\) is rectifiable.
\end{proposition}  
\begin{proof}
  This is an easy consequence of \autoref{prop:rectifable-factor}.
\end{proof}

Recall that \(x(i,j) = \Cone(x_{i} \to x_{j})\).
\begin{proposition}\label{prop:cleaving}
  Consider a filtration \(x_0 \to x_1 \to \dots \to x_n\).
  The following are equivalent:
  \begin{enumerate}
  \item $x_0 \to x_1 \to \dots \to x_n$ is rectifiable.
  \item for all $i$, both $x_0 \to \dots \to x_i$ and $x_{i} \to \dots \to x_n$ are rectifiable, and the triangle
    \[ x(0,i) \to x(0,n) \to x(i,n) \xrightarrow{+1}\]
    is rectifiable.
  \item for some $i$, both $x_0 \to \dots \to x_i$ and $x_{i} \to \dots \to x_n$ are rectifiable, and the triangle
    \[ x(0,i) \to x(0,n) \to x(i,n) \xrightarrow{+1}\]
    is rectifiable.
  \end{enumerate}
\end{proposition}
\begin{proof}
  Let us show that (1) implies (2).
  We induct on $i$.
  The case $i = 1$ follows from the definition of a rectifiable filtration.
  Assume $i \geq 2$.
  \autoref{prop:rectifiable-truncation} implies that the smaller filtrations $x_0 \to \dots \to x_i$ and $x_{i} \to \dots \to x_n$ are both rectifiable.
  Recall that $x(i,j) = \Cone(x_i \to x_j)$.
  We now prove that the map $x(i,n) \to x(0,i)[1]$ is rectifiable.
  Given \(\alpha \in \mathbf{R}\), we must prove that 
  \begin{equation}\label{eqn:in0i}
    x(i,n)_{> \alpha} \to x(0,i)_{\leq \alpha}[1]
  \end{equation}
  vanishes.
  By \autoref{prop:low-high}, we have the triangle
  \[ x(0,1)_{\leq \alpha} \to x(0,i)_{\leq \alpha} \to x(1,i)_{\leq \alpha} \xrightarrow{+1} \]
  By the inductive hypothesis applied to the rectifiable filtration $x_1 \to \dots \to x_n$ for $(i-1)$, we get that the map
  \begin{equation}\label{eqn:one-less}
    x({i},n) \to x(1,i)[1]
  \end{equation}
  is rectifiable.
  Therefore, the composite
  \[ x(i,n)_{> \alpha} \to x(0,i)_{\leq \alpha}[1] \to x(1,i)_{\leq \alpha}[1]\]
  already vanishes.
  By \Cref{lem:unique-induced-map}, \eqref{eqn:in0i} factors uniquely through a map
  \begin{equation}\label{eqn:in01}
    x(i,n)_{> \alpha} \to x(0,1)_{\leq\alpha}[1].
  \end{equation}
  By \autoref{prop:low-high} applied to the rectifiable map in \eqref{eqn:one-less}, we have the triangle
  \[
    x(1,i)_{> \alpha} \to x(1,n)_{> \alpha}  \to x(i,n)_{> \alpha} \xrightarrow{+1}.
  \]
  Consider the diagram
  \[
    \begin{tikzcd}
      x(1,n)_{> \alpha}  \ar{r}& x(i,n)_{> \alpha} \ar{r}\ar{d}{\eqref{eqn:in01}}& x(1,i)_{> \alpha} [1] \ar{r}{+1}& {}\\
      & x(0,1)_{\leq \alpha}[1].
    \end{tikzcd}
  \]
  The composite
  \[ x(1,n)_{> \alpha} \to x(i,n)_{> \alpha} \to x(0,1)_{\leq \alpha}[1]\]
  vanishes since $x_0 \to \dots \to x_n$ is rectifiable.
  Therefore, the map \eqref{eqn:in01} factors through a map \(x(1,i)_{>\alpha}[1] \to x(0,1)_{\leq \alpha}[1]\).
  But the last map must vanish for phase reasons.
  As a result, \eqref{eqn:in01} vanishes and so does \eqref{eqn:in0i}.
  We have thus proved that \(x(i,n) \to x(0,i)[1]\) is rectifiable.

  That (2) implies (3) is a tautology.

  Let us show that (3) implies (1).
  We again induct on $i$.
  The case $i = 1$ is the definition.
  Assume $i \geq 2$.
  By using \autoref{prop:rectifiable-truncation} and the inductive hypothesis, we conclude that
  \[ x_1 \to \dots \to x_n\]
  is rectifiable.
  It remains to show that the map $x(1,n) \to x(0,1)[1]$ is rectifiable.
  Given \(\alpha \in \mathbf{R}\), consider the map
  \begin{equation}\label{eqn:1n01}
    x(1,n)_{> \alpha} \to x(0,1)_{\leq \alpha}[1].
  \end{equation}
  We have the rectifiable triangle
  \[ x(1,i) \to x(1,n) \to x(i,n) \xrightarrow{+1},\]
  which by \autoref{prop:low-high} gives a distinguished triangle
  \[ x(1,i)_{> \alpha} \to x(1,n)_{> \alpha} \to x(i,n)_{> \alpha} \xrightarrow{+1}.\]
  Since $x_0 \to \dots \to x_i$ is rectifiable, the composite
  \[x(1,i)_{> \alpha} \to x(1,n)_{> \alpha} \to x(0,1)_{\leq \alpha}[1]\]
  vanishes.
  By \Cref{lem:unique-induced-map}, \eqref{eqn:1n01} factors uniquely through a map
  \begin{equation}\label{eqn:in011}
    x(i,n)_{> \alpha} \to x(0,1)_{\leq \alpha}[1].
  \end{equation}
  Since $x_0 \to \dots \to x_i$ is rectifiable, we have the triangle
  \[ x(0,1)_{\leq \alpha} \to x(0,i)_{\leq \alpha} \to x(1,i)_{\leq \alpha} \xrightarrow{+1}.\]
  Consider the diagram
  \[
    \begin{tikzcd}
      & x(i,n)_{> \alpha}\ar{d}{\eqref{eqn:in011}}\\
    x(1,i)_{\leq \alpha} \ar{r}& x(0,1)_{\leq \alpha}[1] \ar{r} \ar{r}& x(0,i)_{\leq \alpha} [1] \ar{r}{+1}& {}.
  \end{tikzcd}
\]
The composite
\[x(i,n)_{>\alpha} \to x(0,1)_{\leq \alpha}[1] \to x(0,i)_{\leq \alpha}[1]\]
vanishes since $x(i,n) \to x(0,i)[1]$ is rectifiable.
As a result, \eqref{eqn:in011} factors through a map \(x(i,n)_{>\alpha} \to x(1,i)_{\leq \alpha}\).
But the last map must vanishes for phase reasons.
As a result, \eqref{eqn:in011} vanishes and so does \eqref{eqn:1n01}.
The proof is now complete.
\end{proof}

\begin{theorem}\label{prop:rect-hn}
  Suppose that a filtration \(0 = x_0 \to x_{1} \to \dots \to x_n = x\) is rectifiable.
  Set $a_i = \Cone(x_{i-1} \to x_i)$.
  Then for any \(\alpha \in \R\), we have an induced filtration
  \[0 \subset (x_{1})_{\alpha} \subset \cdots \subset (x_n)_{\alpha} = x_{\alpha}\]
  in the abelian category \(\mathcal{C}_{\alpha}\), with factors
  \[(x_{i})_{\alpha}/(x_{i-1})_{\alpha} = (a_i)_{\alpha}.\]
  Consequently, we have 
  \[
    m_{q}(x) = \sum m_{q}(a_i).
  \]
\end{theorem}
\begin{proof}
  We induct on \(n\).
  The base cases \(n = 0,1\) are immediate.
  Suppose the statement holds for \(n-1\).
  Since \( 0 = x_{0} \to x_{1} \to \cdots \to x_{n}\) is rectifiable, the triangle
  \[ x(0,n-1) \to x(0,n) \to x(n,n-1) \xrightarrow{+1} \]
  is rectifiable by~\Cref{prop:cleaving}.
  Note that this triangle is equal to
  \[ x_{n-1} \to x_{n} \to a_{n} \xrightarrow{+1}.\]
  By \Cref{prop:rect-filt-hn}, in \(\mathcal{C}_{\alpha}\) we have the short exact sequence
  \begin{equation}\label{eq:subfilt}
    0 \to (x_{n-1})_{\alpha} \to (x_{n})_{\alpha} \to (a_{n})_{\alpha} \to 0.
  \end{equation}
  That is, we have an inclusion \((x_{n-1})_{\alpha} \subset (x_{n})_{\alpha}\) whose quotient is \((a_{n})_{\alpha}\).
  We now apply the inductive hypothesis to the rectifiable filtration
  \[ 0 = x_{0} \to \cdots \to x_{n-1}\]
  to obtain the filtration 
  \[0 \subset (x_{1})_{\alpha} \subset \cdots \subset (x_n)_{\alpha} = x_{\alpha}\]
  with the desired factors.
  Summing over the \(\alpha\) as in~\Cref{prop:rect-filt-hn}, we see that the \(q\)-mass of \(x\) is the sum of the \(q\)-masses of the \(a_{i}\).
\end{proof}

\begin{proposition}\label{prop:swap}
  Suppose
  \[ x_0 \to \dots \to x_{i-1} \to x_i \to x_{i+1}\to \dots \to x_n\]
  is rectifiable, and the map $x(i,i+1) \to x(i-1,i)[1]$ is zero.
  Consider a filtration
  \[ x_0 \to \dots \to x_{i-1} \to x'_i \to x_{i+1}\to \dots \to x_n = x,\]
  obtained by flipping \(x(i-1,i)\) and \(x(i,i+1)\) as in~\Cref{lem:swap-terms}.
  Then this new filtration is also rectifiable.
\end{proposition}
\begin{proof}
  By~\Cref{prop:cleaving}, it suffices to show that
  \begin{enumerate}
  \item the filtration \(x_{0} \to \cdots \to x_{i-1}\) is rectifiable;
  \item the filtration \(x_{i-1} \to x_{i}' \to \cdots \to x_{n}\) is rectifiable; and
  \item the triangle \(x(0,i-1) \to x(0,n) \to x(i-1,n) \xrightarrow{+1}\) is rectifiable.
  \end{enumerate}
  The first and third conditions follow directly from the rectifiability of the original filtration, so it remains to check the second.
  Once again, we apply~\Cref{prop:cleaving}.
  It suffices to show that
  \begin{enumerate}
  \item the filtration \(x_{i-1} \to x_{i}' \to x_{i+1}\) is rectifiable;
  \item the filtration \(x_{i+1} \to \cdots \to x_{n}\) is rectifiable; and
  \item the triangle \(x(i-1,i+1) \to x(i-1,n) \to x(i+1,n) \xrightarrow{+1}\) is rectifiable.
  \end{enumerate}
  The second and third conditions follow directly from the rectifiability of the original filtration, so it remains to check the first.
  However, the connecting map in the first filtration is zero by~\Cref{lem:swap-terms}, and so it is rectifiable.
  The proof is complete.
\end{proof}
\begin{theorem}\label{prop:rectifiable-filtration-check}
  Consider a filtration
  \[  x_0 \to \dots \to x_n = x.\]
  Set $a_i = \Cone(x_{i-1} \to x_i)$.
  Suppose for every $i, j$ with $i < j$, we have
  \[\Hom(a_j, a_i[1]) = \Hom(a_{i},a_{j}[1]) = 0 \text{ or } \phi(\lfloor  a_i \rfloor) \geq \phi(\lceil a_j \rceil).\]
  Then the filtration is rectifiable.
\end{theorem}
\begin{proof}
  We induct on $n$.
  The base cases $n = 0,1$ are trivial.
  Assume $n \geq 2$.

  Suppose \(i\) is such that 
  \[ \phi(\lceil  a_{i} \rceil) < \phi(\lceil a_{i+1} \rceil).\]
  Then we also have \(\phi(\lfloor  a_{i} \rfloor) < \phi(\lceil a_{i+1}\rceil)\).
  By the hypothesis, we must have \(\Hom(a_{i+1}, a_{i}[1]) = 0\).
  In particular, the connecting map \(a_{i+1} \to a_{i}[1]\) vanishes.
  Consider a filtration obtained by flipping \(a_{i}\) and \(a_{i+1}\) as in \Cref{lem:swap-terms}.
  The sequence of factors of the new filtration is
  \[a_{1}, \dots, a_{i-1}, a_{i+1}, a_{i}, a_{i+2}, \dots a_{n}.\]
  Observe that the new filtration also satisfies the hypotheses of the theorem.
  By \Cref{prop:swap}, it suffices to check that the new filtration is rectifiable.

  By repeatedly swapping adjacent terms as above, we may assume that
  \[ \phi(\lceil  a_{1} \rceil) \geq \cdots \geq \phi(\lceil  a_{n} \rceil).\]
  By the inductive hypothesis, the filtration
  \[ x_{1} \to \cdots \to x_{n}\]
  is rectifiable.
  It remains to check that the triangle
  \begin{equation}\label{eqn:conmap}
    x(0,1) \to x(0,n) \to x(1,n) \xrightarrow{+1}
  \end{equation}
  is rectifiable.

  Note that \(x(0,1) = a_{1}\).
  Let \(b = \phi(\lfloor  a_{1} \rfloor)\).
  Partition \(\{2, \dots, n\}\) in two subsets \(\{2, \dots, i\}\) and \(\{i+1, \dots,n\}\) such that:
  \begin{enumerate}
  \item for all \(j \in \{2, \dots, i\}\) we have
    \(b < \phi(\lceil  a_{j} \rceil)\), and
  \item for all \(j \in \{i+1, \dots, n\}\) we have \(b \geq \phi(\lceil  a_{j} \rceil)\).
  \end{enumerate}
  If the first set is empty, take \(i = 1\); in this case, \(x(1,i) = 0\).
  If the second set is empty, take \(i = n+1\), and set \(x(i,n) = 0\).
  Consider the diagram
  \[
    \begin{tikzcd}
      x(1,i) \ar{r}& x(1,n) \ar{r}\ar{d}& x(i,n) \ar{r}{+1} & {} \\
      & a_{1}[1]
    \end{tikzcd}
  \]
  in which the vertical map is the connecting map of~\eqref{eqn:conmap}.
  The object \(x(1,i)\) has the filtration
   \[ 0 = x(1,1) \to x(1,2) \to \cdots \to x(1,i)\]
  with subquotients \(a_{2}, \dots, a_{i}\).
  By construction, for \(j \in \{2, \dots, i\}\), we have
  \[b = \phi(\lfloor  a_{1} \rfloor) < \phi(\lceil a_{j} \rceil)\]
  and hence, by hypothesis, we must have
  \[ \Hom(a_{j}, a_{1}[1]) = 0.\]
  As a result, we get
  \[ \Hom(x(1,i), a_{1}[1]) = 0.\]
  Therefore, the vertical map of the diagram above factors as a composition
  \begin{equation}\label{eqn:conmapfilt}
    x(1,n) \to x(i,n) \to a_{1}[1].
  \end{equation}

  Now, \(x(i,n)\) has the filtration
  \[ 0 = x(i,i) \to x(i,i+1) \to \cdots \to x(i,n)\]
  with factors \(a_{i+1}, \dots, a_{n}\).
  For \(j \in \{i+1, \dots, n\}\), we have
  \[ b \geq \phi(\lceil  a_{j} \rceil),\]
  and therefore
  \[ b = \phi(\lfloor a_{1} \rfloor) \geq \phi(\lceil  x(i,n) \rceil).\]
  By \Cref{ex:easyrect}, any map \(x(i,n) \to a_{1}[1]\) is rectifiable.
  In particular, the second map in \eqref{eqn:conmapfilt} is rectifiable.
  \Cref{prop:rectifable-factor} implies that \(x(1,n) \to a_{1}[1]\) is rectifiable; that is, the triangle \eqref{eqn:conmap} is rectifiable.
\end{proof}

\section{Self-extensions of a spherical object}\label{sec:an}
The aim of this section is to record some properties of self-extensions of a spherical object.
The results of this section plays a significant role in the proof of~\Cref{prop:limit-mass}, and may also be of independent interest.

Fix \(d \geq 2\), a field \(\k\), a \(\k\)-linear triangulated category \(\mathcal{C}\), and a \(d\)-spherical object \(a\) of \(\mathcal{C}\).
Fix a non-zero element \(\loopmap_a \in \Hom^d(a,a)\).
It will be convenient to set \(e = d - 1\).
\begin{proposition}\label{prop:a-an-homs}
  Let \(a_n\) be an object of \(\mathcal{C}\) that admits a filtration
  \begin{equation}\label{eq:a-filtration}
    \begin{tikzcd}[column sep=1em]
      0 = a_{-1} \arrow{rr}{} && a_0 \arrow{rr}{}\arrow{dl}{} && \cdots \arrow{rr}{}\arrow{dl}{} 
      && a_n\arrow{dl}{}\\
      &a\arrow[dashed]{ul}{}&&a[-e]\arrow[dashed]{ul}{}&\cdots&
      a[-ne]\arrow[dashed]{ul}{},
    \end{tikzcd}
  \end{equation}
  where the connecting maps
    \[ a[-ie] \to a[-(i-1)e+1]\]
   are shifts of \(\loopmap_{a}\).
  Then we have the following.
  \begin{enumerate}
  \item The space \(\gradedHom(a,a_n)\) is two dimensional, with generators
    \[i_n \colon a \to a_n \text{ and } l_n \colon a[-ne-d] \to a_n\]
    in degrees \(0\) and \(ne+d\) respectively.
    The map \(i_n\) is a non-zero multiple of the map \(a = a_{0} \to a_n\) in~\eqref{eq:a-filtration}.
  \item The space \(\gradedHom(a_n,a)\) is two dimensional, with generators
    \[t_n \colon a_n \to a[d]\text{ and }q_n \colon a_n \to a[-ne]\]
    in degrees \(d\) and \(-ne\) respectively.
    The map \(q_n\) is a non-zero multiple of \(a_n \to a[-ne]\) in~\eqref{eq:a-filtration}.
  \item The compositions
    \begin{align*}
      t_n \circ i_n &\colon a \to a[d],\\
      q_n \circ l_n &\colon a[-ne-d] \to a[-ne]
    \end{align*}
    are non-zero multiples of shifts of \(\loopmap_a\).
  \end{enumerate}
\end{proposition}
\begin{proof}
  We prove the first two statements by induction on \(n\).

  For \(n = 0\), the proposition follows from the the \(d\)-CY property of \(a\).
  Let us assume it for \((n-1)\), and prove it for \(n\).

 Consider the distinguished triangle
 \begin{equation}\label{eq:antriangle}
   a_{n-1} \to a_n \to a[-ne] \xrightarrow{+1}.
 \end{equation}
 The connecting map \(a[-ne] \to a_{n-1}[1]\) composed with the map \(a_{n-1}[1] \to a[-(n-1)e+1]\) is a shift of \(\loopmap_{a}\).
 In particular, the connecting map is non-zero.
  
  Apply \(\Hom(a,-)\) to~\eqref{eq:antriangle} to get the long exact sequence
  \[\cdots \to \Hom^i(a,a_{n-1}) \to \Hom^i(a,a_n) \to \Hom^i(a,a[-ne]) \to \cdots.\]
  Both \(\Hom^i(a,a_{n-1})\) and \(\Hom^i(a,a[-ne])\) are zero unless \(i \in \{0, ne, ne+1, ne + d\}\).
  So \(\Hom^{i}(a,a_{n}) = 0\) for all \(i\) unless \(i \in \{0, ne, ne+1, ne + d\}\).
  For \(i = 0\), the long exact sequence gives
  \[0 \to \Hom^0(a,a_{n-1}) \to \Hom^0(a,a_n) \to 0.\]
  From the inductive hypothesis, we conclude that \(\Hom^0(a,a_n)\) is one-dimensional and is spanned by a copy of the map \(i_n \colon a_0 \to a_n\) from \eqref{eq:a-filtration}.
  For \(i = ne\) and \(i = ne+1\), the hom spaces fit in the sequence
 \[0 \to \Hom^{ne}(a,a_{n}) \to \Hom^{ne}(a,a[-ne]) \xrightarrow{\delta} \Hom^{ne+1}(a,a_{n-1}) \to \Hom^{ne+1}(a,a_{n})\to 0.\]
 The source of \(\delta\) isomorphic to \(\k\), spanned by the identity.
 The target of \(\delta\) isomorphic to \(\k\) by the inductive hypothesis.
 The map \(\delta\) is induced by the connecting map \(a[-ne] \to a_{n-1}[1]\) in \eqref{eq:antriangle}, which is non-zero.
 It follows that \(\delta\) is an isomorphism, and hence \(\Hom^i(a,a_{n}) = 0\) for \(i = ne\) and \(ne+1\).

  Apply \(\Hom(-,a)\) to~\eqref{eq:antriangle} to get the long exact sequence
  \[\cdots \to \Hom^i(a[-ne], a) \to \Hom^i(a_n,a) \to \Hom^i(a_{n-1},a) \to \cdots.\]
  The \(d\)-CY property of \(a\) and the vanishing of \(\Hom^i(a,a_n)\) for all \(i \not \in \{0,-ne\}\) imply that \(\Hom^i(a_n,a) = 0\) for all \(i \not \in \{-ne+d,d\}\).
  For \(i = -ne\), the long exact sequence gives
  \[ 0 \to \Hom^{-ne}(a[-ne], a) \to \Hom^{-ne}(a_n,a) \to 0.\]
  We conclude that \(\Hom^{-ne}(a_{n},a)\) is one-dimensional and is spanned by a copy of the map \(q_{n} \colon a_{n} \to a[-ne]\) from \eqref{eq:a-filtration}.

   We have proved that \(\Hom^{0}(a,a_n)\) and \(\Hom^{-ne}(a_n,a)\) are one-dimensional and spanned by the maps \(i_n\) and \(q_{n}\) in \eqref{eq:a-filtration}.
   The \(d\)-CY property of \(a\) implies that \(\Hom^{d}(a_n,a)\) and \(\Hom^{ne+d}(a,a_n)\) are one-dimensional and spanned by the maps that compose with \(i_n\) and \(q_n\) to give \(\loopmap_a\).
   We have also proved that all other hom spaces vanish.
   The induction step is complete.
\end{proof}

\begin{proposition}\label{prop:an-objects}
  For each \(n \geq 0\), there exists an object \(a_n\), unique up to isomorphism, which admits a filtration whose factors are \(a, a[-e], \cdots, a[-ne]\) such that for all \(i\), the connecting maps \(a[-ie] \to a[-ie+d]\) are shifts of \(\loopmap_{a}\).
\end{proposition}
\begin{proof}
  We induct on \(n\).
  For the base case \(n = 0\), we have \(a_0 \cong a\).
  
  We assume the result for \(n\), and prove it for \(n+1\).
  Observe that~\Cref{prop:a-an-homs} applies to \(a_n\).
  Let \(a_{n+1}\) be the object that fits into a distinguished triangle
  \begin{equation}\label{eq:actual}
    a[-ne-d] \xrightarrow{l_n} a_n \to a_{n+1} \xrightarrow{+1}.
  \end{equation}
  Using  the rotated triangle
  \[ a_n \to a_{n+1} \to a[-(n+1)e] \xrightarrow{+1}\]
  and the assumed filtration of \(a_n\), we obtain the desired filtration for \(a_{n+1}\).
  
  It remains to prove that the object \(a_{n+1}\) is unique up to isomorphism.
  Consider another object \(a_{n+1}'\) that also admits a filtration
  \begin{equation}\label{eq:imposter}
    0 \to a_0' \to \cdots \to a_n' \to a_{n+1}'
  \end{equation}
  satisfying the hypotheses.
  By the inductive hypothesis, we have an isomorphism \(a_n' \cong a_n\).
  The last map in~\eqref{eq:imposter} then fits into a distinguished triangle
  \[a[-ne-d] \xrightarrow{\iota'} a_n \to a'_{n+1} \xrightarrow{+1}.\]
  By assumption, the composite of \(\iota'\) with the next map \(a'_n \to a[-(n-1)e]\) in~\eqref{eq:imposter} is the loop map; so \(\iota' \neq 0\).
  By~\Cref{prop:a-an-homs}, \(\Hom^{ne+d}(a,a_n)\) is one-dimensional, so \(\iota'\) is a non-zero scalar multiple of \(l_n\).
  By comparing~\eqref{eq:actual} and~\eqref{eq:imposter}, we obtain an isomorphism \(a_{n+1} \cong a'_{n+1}\).
\end{proof}

\begin{proposition}\label{lem:twomaps}
  For each \(n \geq 0\), consider non-zero maps
  \begin{align*}
    i_n &\colon a \to a_n, \quad l_n \colon a[-ne-d] \to a_n\\
    t_n &\colon a_n \to a[d], \quad q_n \colon a_n \to a[-ne]
  \end{align*}
  as in~\Cref{prop:a-an-homs}.
  Then we have distinguished triangles
  \begin{align*}
    &a_{n-1}[-d] \xrightarrow{t_{n-1}[-d]} a \xrightarrow{i_n} a_n \xrightarrow{+1} \text{ and }\\
    &a_n\xrightarrow{q_n} a[-ne] \xrightarrow{l_{n-1}[1]} a_{n-1}[1] \xrightarrow{+1}.
  \end{align*}
\end{proposition}
\begin{proof}
  Using~\Cref{prop:an-objects}, the objects \(a_i\) fit into a filtration
  \begin{equation}\label{eq:an-filtration}
    \begin{tikzcd}[column sep=1em]
      0 = a_{-1} \arrow{rr}{} && a_0 \arrow{rr}{}\arrow{dl}{} && \cdots \arrow{rr}{}\arrow{dl}{} 
      && a_n\arrow{dl}{}\\
      &a\arrow[dashed]{ul}{}&&a[-e]\arrow[dashed]{ul}{}&\cdots&
      a[-ne]\arrow[dashed]{ul}{},
    \end{tikzcd}
  \end{equation}
  where the connecting maps are loop maps.
  By the uniqueness statement in~\Cref{prop:an-objects}, the cone of the map \(i_n \colon a = a_0 \to a_n\) is \(a_{n-1}[-e]\).
  It gives rise to the first triangle.
  Likewise, the cone of the map \(q_{n} \colon a_n \to a[-ne]\) is \(a_{n-1}[1]\).
  It gives rise to the second triangle.
\end{proof}

\begin{proposition}
  For each \(n \geq 0\), the object \(a_n\) is indecomposable in \(\mathcal{C}\).
\end{proposition}
\begin{proof}
  The statement is clear for \(n = 0\): spherical objects are indecomposable because they have a one-dimensional space of endomorphisms of degree zero.
  It can be checked, e.g., via induction on \(0 \leq k \leq n\) and using the triangles
  \[a_{k-1} \to a_k \to a[-ke] \xrightarrow{+1}\]
  that \(\Hom^0(a_k,a_n)\) is one-dimensional for each \(k \leq n\).
  In particular, \(\Hom^0(a_n,a_n)\) is one-dimensional for each \(n\), and so the objects \(a_n\) are indecomposable.
\end{proof}

In the remainder of the section, we prove some homological properties of the objects \(a_n\).
These results are important ingredients in the proof of~\Cref{prop:limit-mass}.
Suppose we have a stability condition \(\tau\) on \(\mathcal{C}\) such that \(a\) is a \(\tau\)-semi-stable object in the \([0,1)\)-heart, with phase strictly between \(0\) and \(1\).
Then the defining filtration of \(a_n\) from~\Cref{prop:an-objects} is also its \(\tau\)-Harder--Narasimhan filtration, because the factors are semi-stable and of decreasing phase.

Let \(I\) be the interval \([0,e)\).
For \(m \in \Z\), let by \(I-me\) be the translate \([-me, -(m-1)e)\).
For an object \(x \in \mathcal{C}\), we study the truncations \(x_{I-me}\).
For the object \(a_{n}\), observe that
\[
  (a_{n})_{I-me} = \begin{cases}
                     a[-m] & \text{ if } m \in \{0,\dots,n\} \\
                     0 & \text{ otherwise}.
                   \end{cases}
\]
The objects \(a_n\) satisfy the following ``lifting property''.
\begin{lemma}\label{lem:cohan}
  Let \(f \colon x \to a_n\) be a map such that the induced map
  \[ x_{I-ne} \to (a_{n})_{I-ne}\]
  vanishes.  Then the map \(x_{\geq -(n-1)e} \to a_{n-1}\), obtained by truncating \(f\), extends to a map \(x \to a_{n-1}\).
\end{lemma}
\begin{proof}
  We have the distinguished triangle
  \[ x_{\geq -(n-1)e} \to x \to x_{< -(n-1)e} \xrightarrow{+1}.\]
  It suffices to prove that the composition \(x_{< -(n-1)e}[-1] \to x_{\geq -(n-1)e} \to a_{n-1}\) vanishes.
  We have the commutative diagram
  \[
  \begin{tikzcd}
    x_{<-(n-1)e}[-1] \ar{r}\ar{d}{\overline f}& x_{\geq -(n-1)e} \ar{r}\ar{d}& x\ar{d} \\
    a[-ne-1] \ar{r}{b}& a_{n-1} \ar{r}& a_n. 
  \end{tikzcd}
\]
We prove that the composite \(b \circ \overline f\) vanishes.
The map
\[ x_{I-ne}[-1] \to a[-ne-1]\]
induced by \(f\) vanishes by assumption.
Therefore, \(\overline f\) factors through a map \(g\) as follows
\[
  \begin{tikzcd}
    x_{I-ne}[-1] \ar{r}& x_{< -(n-1)e}[-1] \ar{r}\ar{d}{\overline f}& x_{< -ne} [-1]\ar{dl}{g} \\
    & a[-ne-1].
  \end{tikzcd}
\]
Now it suffices to prove that the composite
\[ b \circ g \colon x_{<-ne} [-1] \to a_{n-1}\]
vanishes.
The phases of the HN factors of the domain lie in the interval \((-\infty, -ne-1)\).
The HN factors of the codomain are \(d\)-CY and their phases lie in the interval \([-(n-1)e, +\infty)\).
Let \(s\) be a semi-stable HN factor of the domain and \(t\) a semi-stable HN factor of the codomain.
Then \(\phi(t) > \phi(s) + d\).
By the \(d\)-CY property of \(t\), we have
\[ \dim\Hom(s,t) = \dim\Hom(t,s[d]),\]
and the right hand side must vanish since \(\phi(t) > \phi(s[d])\).
We conclude that \[\Hom(x_{<-ne}[-1], a_{n-1}) = 0,\] and hence the composite \(b \circ g\) vanishes.
\end{proof}

We can use the previous lemma to show that the map \(i_n \colon a \to a_n\) of degree zero is irreducible in the following way.
\begin{lemma}\label{lem:i-does-not-factor}
  Let \(x\) be an object that does not contain \(a\) as a direct summand and does not have a self hom of negative degree.
  Then, for any \(n \geq 0\), the map \(i_n \colon a \to a_n\) does not factor as \(a \to x \to a_n\).
\end{lemma}
\begin{proof}
  We prove the contrapositive.
  We continue to use truncations to the translates of the interval \(I\).

  Suppose we have maps \(s \colon a \to x\) and \(q \colon x \to a_n\) that compose to \(i_n \colon a \to a_n\).
  Let \(n\) be the smallest with this property.
  If \(n = 0\), then we obtain \(a\) as a direct summand.
  If \(n > 0\), we claim that the composite map
  \begin{equation}\label{eqn:negmap}
    x \xrightarrow{q} a_n \to a[-ne] \xrightarrow{s[-ne]} x[-ne]
  \end{equation}
  is non-zero.

  To see this, we first consider the map
  \begin{equation}\label{eq:trohomology}
    x_{I-ne} \to (a_{n})_{I-ne}.
  \end{equation}
  If this map is zero,~\Cref{lem:cohan} yields a map \(r \colon x \to a_{n-1}\) which also has the property that the composite of \(s \colon a \to x\) and \(r \colon x \to a_{n-1}\) is \(i_{n-1} \colon a \to a_{n-1}\).
  Since we chose \(n\) to be the smallest with this property, we conclude that the map~\eqref{eq:trohomology} is non-zero.

  Second, we see that the composite
  \[
    a_{I} \to x_{I} \to (a_{n})_{I}
  \]
  is the identity.
  Therefore, the map \(a_{I} \to x_{I}\) has a left inverse.

  Now, consider the maps on the truncations induced by~\eqref{eqn:negmap}
  \[ x_{I-ne} \xrightarrow{q} (a_{n})_{I-ne} \xrightarrow{=} a[-ne] \xrightarrow{s[-ne]} (x[-ne])_{I-ne}.\]
  The first map is non-zero, the middle map is an equality, and the last map has a left inverse.
  It follows that the composite is non-zero, and hence the composite in~\eqref{eqn:negmap} is also non-zero.
\end{proof}

\bibliographystyle{siam}

\end{document}